\documentclass[11pt,reqno]{article}

\usepackage[text={153mm,235mm},centering]{geometry}
\usepackage{enumerate}

\usepackage[mathscr]{eucal}
\usepackage[all]{xy}
\usepackage{mathrsfs}
\usepackage{authblk}
\usepackage{xypic}
\usepackage{amsfonts}
\usepackage{amsmath}
\usepackage{amsthm}
\usepackage{amssymb,bm}
\usepackage{latexsym}
\usepackage{tabularx}
\usepackage{graphicx}
\usepackage{wrapfig}
\usepackage{pict2e}
\usepackage{amsmath}
\usepackage{pifont}
\usepackage{hyperref}

\usepackage{pdfpages}

\allowdisplaybreaks
\newtheorem{theorem}{Theorem}[section]

\newtheorem{proposition}[theorem]{Proposition}
\newtheorem{lemma}[theorem]{Lemma}
\newtheorem{remark}[theorem]{Remark}
\newtheorem{corollary}[theorem]{Corollary}

\newtheorem{lemma-definition}[theorem]{Lemma-definition}

\theoremstyle{definition}

\newtheorem{definition}[theorem]{Definition}
\newtheorem{example}[theorem]{Example}
\newtheorem{notation}[theorem]{Notation}
\newtheorem{convention}[theorem]{Convention}

\newtheorem*{ack}{Acknowledgements}

\begin{document}

\title{Singularity categories 
and singular loci of certain abelian quotient singularities
\footnotetext{Email: xjchen@scu.edu.cn, zengjh662@163.com}}

\author[1,2]{Xiaojun Chen}
\author[3,4]{Jieheng Zeng}

\renewcommand\Affilfont{\small}

\affil[1]{Department of Mathematics, New Uzbekistan University,
Tashkent 100001, Uzbekistan}

\affil[2]{School of Mathematics, Sichuan University, Chengdu 610064, P.R. China}

\affil[3]{School of Mathematics and Statistics, Hunan Normal University, Changsha 410081, P.R. China }

\affil[4]{School of Mathematical Sciences, Peking University, Beijing 100871, P.R. China}

\date{}

\maketitle

\begin{abstract} 
Let $V$ be an affine space over field $k$, which is  characteristic zero.
Let $G\subseteq\mathrm{SL}(V)$ be a finite
abelian group, and denote by $S$ 
the $G$-invariant  
subring of the polynomial ring $k[V]$. 
It is shown that the singularity category $D_{sg}(S)$ recovers the reduced 
singular locus of $\mathrm{Spec}(S)$.

\end{abstract}

\setcounter{tocdepth}{3} \tableofcontents

\section{Introduction}

Let $A$ be an associative algebra over a base field $k$ of characteristic zero. 
Its singularity category $D_{sg}(A)$ is the Verdier quotient 
$D^{b}(A) / \mathrm{Perf}(A)$, where $\mathrm{Perf}(A)$ 
is the full subcategory consisting of perfect complexes over $A$. 
It was first introduced by Buchweitz in \cite{RB} in his 
study of algebraic representations of 
Gorenstein rings. $D_{sg}(A)$
measures the smoothness of $A$ in the sense that it is 
homologically smooth if and only if $D_{sg}(A)$ is a trivial category. 
Moreover, Buchweitz showed that $D_{sg}(A)$ is equivalent to the stable category 
$\underline{\mathrm{CM}}(A)$ of Cohen-Macaulay $A$-modules  
as triangulated categories when $A$ is Gorenstein. 
Later, in \cite{DR,DR1,DR2}, Orlov rediscovered this notion
from the perspective of algebraic  
geometry and mathematical physics,
which has a deep relationship with Homological Mirror Symmetry.

In recent years, 
the singular equivalent invariants of singularity categories
have attracted much attention.
For example, in \cite{CS}, 
Chen and Sun introduced the notion of singular 
equivalence of Morita type. 
The famous Kn\"orrer periodicity theorem 
can be realized by singular equivalence of Morita type. 
In \cite{ZZi}, Zhou and Zimmermann showed that for two  
Noetherian algebras $A$ and $B$ which are singular equivalent
of Morita type, their $n$-th Hochschild cohomology groups are isomorphic,
for $n$ big enough.
Later, Wang generalized this result to 
singular equivalence of Morita type with level  
(see \cite{WZF}). 
He also showed that 
the Tate-Hochschild cohomology together with the 
Gerstenhaber bracket 
is invariant under 
such equivalence (see \cite{W1}).  
Another example is the result of Hua and Keller obtained in \cite{HK0}. 
They showed that the singularity category 
of a local hypersurface ring with isolated singularity
recovers the algebra itself via the isomorphism between the zeroth 
Tate-Hochschild cohomology and the Tyurina algebra of this hypersurface.  
For more results on the singular equivalent invariants, 
one may refer to, just to name a few, 
\cite{HuC, KB, LXC, MU, XCC} and
references therein.

In particular, recently
there has been an increasing study of the relationships 
between the singular locus and the
singularity category of a given algebra.
In \cite{RB}, Buchweitz showed that the Jacobian ideal of a quotient $S$ of a formal 
power series ring modulo regular sequence 
annihilates the singularity
category of $S$. Later, in \cite{ITR}, 
Iyengar and Takahashi extended this result to more general rings, 
including equicharacteristic complete Cohen-Macaulay local rings. Recently, in 
\cite{Liu}, Liu proved that when $S$ is  
either an equidimensional finitely generated 
algebra over a perfect field, or an 
equidimensional equicharacteristic 
complete local ring with a
perfect residue field, some power of 
 the Jacobian ideal of $S$ 
annihilates the singularity category of $S$.

Among all these results, of special interest is the 
``reconstruction theorem" for singularity categories.
In his research of the topological reconstruction of singular locus, 
Yu showed that the category of matrix factorizations $\mathrm{MF}(R, f) 
\cong D_{sg}(R/f)$, equipped with some tensor product structure on it, 
recovers the spectrum of the singular locus of the hypersurface $R/f$ 
as a topological space (see \cite{XYY}). 
In \cite{YHHH}, Hirano extended Yu's result to the case of relative singular locus.    
Later, for the 
Gorenstein ring $S$ which is locally a hypersurface on the punctured spectra, 
Matsui showed that the topological structure of its singular locus 
can be reconstructed by its singularity 
category $D_{sg}(S)$ 
(see \cite{MHH}).

Our objective in this paper is certain
Gorenstein normal rings, which are the coordinate rings of 
the quotients of affine spaces by 
finite abelian subgroups of the special linear group. 
In the literature, 
the quotient of an affine
space by an abelian group
is usually called an
{\it abelian}
quotient singularity.
In general, the singularities in this case are not 
isolated. The purpose of this paper is to use 
the singularity category of $S$ to reconstruct the {\it reduced} singular 
locus, denoted by $\sqrt{\mathrm{Sing}\big(\mathrm{Spec}(S)\big)}$,
of $\mathrm{Spec}(S)$. Our main result is the following.

\begin{theorem}\label{main4}
Let $S_i$, $i=1,2$,
be two coordinate 
rings of the quotients of affine 
spaces $V_i$ by finite abelian subgroups of 
$\mathrm{SL}(V_i)$ respectively. 
Suppose that there is a triangle equivalence
$
\Upsilon: D_{sg}(S_1) \rightarrow D_{sg}(S_2)
$
of their 
singularity categories,
then their reduced singular locus are isomorphic:
 $$
 \sqrt{\mathrm{Sing}\big(\mathrm{Spec}(S_1)\big)} \cong 
 \sqrt{\mathrm{Sing}\big(\mathrm{Spec}(S_2)\big)}.
 $$ 
\end{theorem}

In other words, the reduced singular
locus of $S$ only depends on its singularity category but nothing else; 
one can reconstruct the reduced
scheme structure of the singular locus from the singularity category. 

In \cite{DO1} Orlov proved that the 
completion of a variety along 
its singular locus determines its singularity category, 
up to the idempotent completion of a triangulated category. 
However, the converse to 
Orlov's result does not hold by, for example, the Kn\"orrer periodicity theorem. 
Nevertheless, our theorem gives a partial 
answer to this converse problem up to some extent.

To show the above result, we heavily use techniques from the theory of McKay quivers,
non-commutative
resolutions introduced by Van den Bergh (see \cite{V,V2}), 
and their contraction algebras introduced by 
Donovan and Wemyss (see \cite{DW}).

The rest of this paper is devoted to the proof of the above theorem. 
It is organized as follows. 
In \S\ref{Pre}, 
we introduce some necessary notions and results 
on singularity categories. 
In \S\ref{SlN}, we describe 
the relations between the singular locus of 
$\mathrm{Spec}(S)$ and the 
contraction algebra of the canonical 
noncommutative resolution of $S$, and prove that 
the reduced center of this contraction algebra is isomorphic to the coordinate ring 
of the reduced singular locus of $\mathrm{Spec}(S)$ (see Theorem \ref{Theo22}).  
In \S\ref{ST}, we introduce a inverse system $\mathcal{R}^S$ from $D_{sg}(S)$, 
and then show that the inverse limit of $\mathcal{R}^S$ is exactly the coordinate ring 
of the reduced singular locus. 
This isomorphism gives the proof of Theorem \ref{main4}. 
Finally, in \S\ref{EX}, we give 
two examples of our main theorem. 
This paper also contains several appendices,
where the proofs of several results in the main
context, usually lengthy and technical, are given.

\begin{convention}\label{Not}
In this paper, we assume the base field $k$ is algebraically closed of characteristic zero.  
All modules are right modules and all complexes are cochain complexes 
unless otherwise specified. 
\end{convention}

\begin{ack} 
We would like to thank Youming Chen, Leilei Liu, 
Song Yang and Xiangdong Yang for several helpful conversations. 
The second author also would like to thank Huijun Fan 
for his encouragement, support and suggestions. 
This work is supported by NSFC No. 12271377 and 12261131498.
\end{ack}

\section{Preliminaries}\label{Pre}

In this section, we collect several necessary concepts and notions which will be used
in later sections.

\subsection{Gorenstein rings and Cohen-Macaulay modules}

\begin{definition}
A commutative Noetherian ring $S$ over $k$ 
is called {\it Gorenstein} if for any prime ideal 
$\mathfrak{p} \subseteq S$, 
$
 \mathrm{Ext}^{i}_{S_\mathfrak{p}}(S_{\mathfrak{p}} / \mathfrak{m}, S_{\mathfrak{p}}) = 0   
$
for all $i \neq d$ and $\mathrm{Ext}^{d}_{S_\mathfrak{p}}(S_{\mathfrak{p}} / 
\mathfrak{m}, S_{\mathfrak{p}}) \cong k$ 
as vector spaces, where 
$\mathfrak{m}$ is the unique maximal ideal of the
local ring $S_{\mathfrak{p}}$ and 
$d := \mathrm{Krull.\, dim}(S_{\mathfrak{p}})$ is the Krull dimension of $S_{\mathfrak{p}}$. 
\end{definition}

By using the
Koszul resolution, it is straightforward to check that the polynomial ring
$k[x_1,\cdots, x_n]$ is Gorenstein. 
More generally, we have the following well-known result.

\begin{proposition}[{\cite[\S2.2]{IT}}]\label{ASstru}
Let $R =k[V]=k[x_1,\cdots, x_n]$ and $G$ be a finite
subgroup in $\mathrm{SL}(V)$, which naturally acts on $V$ and hence on $R$.
Then $R^G$ is a Gorenstein normal domain. 
\end{proposition}

Let $S$ be a Noetherian local ring, $\mathfrak{m}$ be the unique 
maximal ideal of $S$
and $M$ be a finitely generated $S$-module. Recall that the {\it depth} of 
$M$, denoted by $\mathrm{depth}(M)$,
is defined as follows:
\begin{enumerate}
\item[(1)] If $M \mathfrak{m} = M$, then $\mathrm{depth}(M) = \infty$;

\item[(2)] If $M \mathfrak{m} \neq M$, then $\mathrm{depth}(M)$ is 
the supremum of the lengths of
$M$-regular sequences in $\mathfrak{m}$.
\end{enumerate}
Here, an $M$-regular sequence means a sequence of  elements 
$\{f_r \}_{1 \leq r \leq m}$ of $S$ such that for any 
element $f_r$ in this sequence, 
\begin{enumerate}
\item[(1)] $f_r$ is not a zero divisor on $M\big/ M(f_{1}, \cdots, f_{r-1})$, and 
\item[(2)] $M\big/ M(f_{1}, \cdots, f_{m})$
is not trivial as $S$-module. 
\end{enumerate}

\begin{definition}
Let $S$ be a commutative Noetherian ring with Krull dimension $d$ and 
$M$ be a finitely generated $S$-module.
$M$ is called {\it maximal Cohen-Macaulay} if 
$\mathrm{depth}(M_{\mathfrak{m}}) = d$
for any maximal ideal $\mathfrak{m}$ of $S$, 
or $M \cong 0$. 
We simply call such modules Cohen-Macaulay. 
\end{definition}

\subsection{The singularity category}
Suppose $S$ is an 
associative Noetherian $k$-algebra.
Let $\mathrm{mod}(S)$ be the abelian category of 
finitely generated $S$-modules.
Let $D^{b}(S)$ be the bounded derived 
category of $\mathrm{mod}(S)$ and $\mathrm{Perf}(S)$ be the full triangulated subcategory of 
$D^{b}(S)$ such that its objects are consist of complexes which are quasi-isomorphic to 
bounded complexes of projective $S$-modules. 
The {\it singularity category} of $S$, denoted by $D_{sg}(S)$, 
is defined to be the Verdier quotient 
$D^{b}(S) / \mathrm{Perf}(S)$ of triangulated categories.  

\begin{definition}\label{Coh}
Let $S$ be a commutative Noetherian ring and 
$\mathrm{CM}(S)$ be the full subcategory of $\mathrm{mod}(S)$ 
consisting 
of Cohen-Macaulay $S$-modules. The {\it stable category
of Cohen-Macaulay $S$-modules}, denoted by
$\underline{\mathrm{CM}}(S)$, is
the category where
\begin{enumerate}
\item[(1)] the objects of $\underline{\mathrm{CM}}(S)$ are the same as
$\mathrm{CM}(S)$; 

\item[(2)] for any objects $X, Y$ in $\underline{\mathrm{CM}}(S)$,
$$\mathrm{Hom}_{\underline{\mathrm{CM}}(S)}(X, Y) 
 := \mathrm{Hom}_{\mathrm{CM}(S)}(X, Y) \big/ I_{X, Y},$$ 
where $I_{X, Y} \subseteq \mathrm{Hom}_{\mathrm{CM}(S)}(X, Y)$ is the 
vector subspace consisting of homomorphisms which factor through some projective $S$-module.   
\end{enumerate}
\end{definition}

When $S$ is a Gorenstein ring, $\underline{\mathrm{CM}}(S)$ is a triangulated category. 
Moreover, we have
$
D_{sg}(S) \cong  \underline{\mathrm{CM}}(S)
$ as triangulated categories 
(see \cite{RB}).

\subsection{Generator of a triangulated  category}

We next recall the notions of generator and 
classical generator of a triangulated category. 

\begin{definition}
Let $\mathcal{T}$ be a triangulated category.
A set $\mathcal{E}$ of objects in $\mathcal{T}$ is said to {\it generate}
$\mathcal{T}$ if for any given $X \in \mathrm{Ob}(\mathcal{T})$,
$
\mathrm{Hom}_{\mathcal{T}}(D, X[i]) = 0
$
for all $D \in \mathcal{E}$ and all $i \in \mathbb{Z}$
implies $X \cong 0$ in $\mathcal{T}$.
$\mathcal{E}$ is called a {\it generator} of $\mathcal{T}$.
\end{definition}

Let $\mathcal{I}_1$ and $\mathcal{I}_2$ be two full triangulated subcategories of $\mathcal{T}$.
Denote by $\mathcal{I}_1 \ast \mathcal{I}_2 $ the full subcategory of $\mathcal{T}$,
whose objects are the objects like $M$ of $\mathcal T$
such that there exists a distinguished triangle 
$$
M_{1} \longrightarrow M \longrightarrow M_2 \rightarrow M_1 [1]
$$
in $\mathcal{T}$ with $M_{i} \in \mathrm{Ob}(\mathcal{I}_i)$. 
Let $\mathcal{E}$ be a set of objects in $\mathcal{T}$.  
Denote by $\langle\mathcal{E}\rangle_1 = \langle\mathcal{E}\rangle$ the smallest full  subcategory of 
$\mathcal{T}$ containing the objects in $\mathcal{E}$ and closed under direct summands, 
finite direct sums and shifts. 
Let $\langle\mathcal{E}\rangle_0$ be the trivial subcategory of $\mathcal{T}$,
 $\langle\mathcal{E}\rangle_{i} := \langle \langle\mathcal{E}\rangle_{i-1} 
 \ast\langle\mathcal{E}\rangle_1 \rangle$, and $\langle
 \mathcal{E}\rangle_{\infty} := \bigcup_{i \geq 0}\langle\mathcal{E}\rangle_{i}$ 
as full subcategories of $\mathcal{T}$.

\begin{definition}
Let $\mathcal{T}$ be a triangulated category.
A set $\mathcal{E}$ of objects 
in $\mathcal{T}$ is said to {\it classically generate} $\mathcal{T}$ if
$$
\mathcal{T} \cong \langle\mathcal{E}\rangle_{\infty}.
$$
$\mathcal{E}$ is called a {\it classical generator} of $\mathcal{T}$. 
\end{definition}

For an object $M$ in $\mathcal T$, it is
classically generated by 
$\mathcal{E}$ if $M \in \langle\mathcal{E}\rangle_{m}$ 
for some $m \in \mathbb{N}$. 
It is direct to see that a set $\mathcal{E}$ of objects in $\mathcal{T}$ classically generates
$\mathcal{T}$ if and only if for any object $M$ in $\mathcal{T}$, it is classically 
generated by $\mathcal{E}$. 
A fact is that
that a classical generator in a triangulated category 
must be a generator in this triangulated category (see \cite[Lemma 13.36.5]{Stacks0}).

\begin{proposition}\label{GeneToSing}
Let
$\pi: D^{b}(S) \rightarrow D^{b}(S) / \mathrm{Perf}(S) \cong  D_{sg}(S)$ 
be
the Verdier quotient functor. Suppose that 
$X \in \mathrm{Ob}\big(D^{b}(S) \big)$ is a classical generator of $D^{b}(S)$. 
Then $\pi(X)$ is a classical generator of 
$D_{sg}(S)$.
\end{proposition}
	
\begin{proof}
Pick any object $M$ in $D_{sg}(S)$. 
There is an object $\widetilde{M}$ of $D^{b}(S)$ 
such that 
$\pi(\widetilde{M}) = M$. 
Since $X$ is a classical generator of 
$D^{b}(S)$, $\widetilde{M} \in \langle X \rangle_{m}$ for some $m \in \mathbb{N}$. 	
Next, via the triangle functor $\pi$, we obtain that 	
$
M = \pi(\widetilde{M} ) \in \langle \pi(X) \rangle_{m}
$
in $D_{sg}(S)$. 	
\end{proof}

\section{Contraction algebra and the singular locus}\label{SlN}

We first recall the definition of the singular locus of an affine variety
(see, for example, \cite{Stacks}).

\begin{definition}[Singular locus]
Let $S$ be a commutative Noetherian ring over $k$.
The {\it singular locus} 
of $\mathrm{Spec}(S)$, denoted by $\mathrm{Sing}\big(\mathrm{Spec}(S)\big)$, 
is the subscheme consisting of prime ideals  
$\mathfrak{p}$
of $S$ such that 
the local ring $S_{\mathfrak{p}}$ is not regular.
\end{definition}

The purpose of this section is to give an algebraic  characterization of the singular locus
of $S$, especially for 
the invariant subring of the polynomial ring under the action of a finite abelian 
subgroup of the special 
linear group.

\subsection{Non-commutative resolution and the contraction algebra}\label{CA}

We start with the notion of non-commutative resolutions (NCR), which is introduced by Van
den Bergh (\cite{V,V2}), and has been intensively studied in recent years.

Let $S$ be a Gorenstein normal domain over $k$.
Recall that an $S$-module $M$ is called {\it reflexive} if 
the natural homomorphism 
$$
M \rightarrow (M^\vee)^\vee	\,\,,  m \mapsto (f \mapsto f(m))
$$
is an isomorphism of $S$-modules, where 
$(-)^\vee$ means $\mathrm{Hom}_{S}(-, S)$.

\begin{definition}[Van den Bergh]\label{NCCR}
A {\it non-commutative resolution} (NCR) of
$S$ is a $k$-algebra of the form $\mathrm{End}_{S}(M)$ 
for some reflexive $S$-module $M$,
such that the global dimension of $\mathrm{End}_{S}(M)$ is finite.
\end{definition}

\begin{example}\label{keyexample}
Let $V$ be a $k$-vector space with dimension $n$.
Let $G$ be a finite abelian subgroup in 
$\mathrm{SL}(V)$ and let $R=k[V]$ and
$S = R^G$. Let $\hat{G}$ to be the set of irreducible representations of $G$
(recall that the dimension of any irreducible representation of $G$ is one). 
For any $W \in \hat{G}$, the $S$-module $(W \otimes R)^{G}$ 
is a Cohen-Macaulay $S$-module (see \cite[Section J]{GJL}). 

It is well-known that if an $S$-module $M$ is 
Cohen-Macaulay, then it is reflexive. 
Let
$\Lambda:
= \mathrm{End}_{S}\big(\bigoplus_{W \in \hat{G}}
(W \otimes R)^G \big)$. 
Then 
by Auslander's theorem,
$\Lambda\cong G \sharp R$, 
and is an NCR of $S$ (see \cite[Example 2.4]{IN}), where $G \sharp R$ is 
the skew group algebra of $G$ and $R$. 
\end{example}

Let $\mathrm{ref}(S)$ and $\mathrm{ref}(R, G)$ be 
the categories
of reflexive $S$-modules and of 
$G$-equivariant reflexive $R$-modules respectively. 
Then we have the following.

\begin{lemma}[{\cite[Lemma 3.3]{SV0}}]\label{Equi} 
Let $R$, $G$ and $S$ be as in Example \ref{keyexample}. 
Then the following two functors
\begin{equation*}
\mathrm{ref}(R, G) \to \mathrm{ref}(S),\,
M \mapsto M^{G}
\end{equation*}
and
\begin{equation*}
\mathrm{ref}(S) \to \mathrm{ref}(R, G),\,
N \mapsto \big((R \otimes_{S} N)^{\vee}\big)^{\vee}
\end{equation*}
are inverse equivalences of
two symmetric monoidal categories,
where $(-)^{\vee}:=\mathrm{Hom}_{R}(-, R)$.
\end{lemma}

By Lemma \ref{Equi}, 
for any $W \in \hat{G}$, we have
\begin{align}\label{Fir789342}
\mathrm{End}_{S}\big((W \otimes R)^{G} \big)	 
 &\cong \mathrm{End}_{\mathrm{ref}(R, G)}(W \otimes R ) 
 \cong \mathrm{Hom}_{\mathrm{ref}(k, G)}(W \otimes W^{\ast},  R ) \nonumber \\ 
& \cong \mathrm{Hom}_{\mathrm{ref}(k, G)}(k,  R ) = S,
\end{align}
where $W^\ast$ is the linear dual of $W$.

We next move to the notion of contraction algebras, 
which is introduced by Donovan and Wemyss in their research of 
NCR and 
will play an important role 
in this paper.

\begin{definition}[Contraction algebra; see \cite{DW}]\label{Connn}
Let $S$ be a Gorenstein commutative ring over $k$ and
$\Lambda^{M} := \mathrm{End}_{S}(S \oplus M)$, 
where $M$ is a Cohen-Macaulay $S$-module.
Let
$[S]$ be the two-sided
ideal of $\Lambda^{M}$ consisting of those $S$-module homomorphisms 
$(S \oplus M) \longrightarrow (S \oplus M)$ which factor
through some object $P \in \mathrm{add}(S)$:
$$
\xymatrix{
  (S \oplus M) \ar[rr] \ar@{.>}[dr]
   &  &    (S \oplus M)    \\
   & P     \ar@{.>}[ur]            }
$$
The {\it contraction algebra} of $\Lambda^M$, denoted by 
$\Lambda^{M}_{\mathrm{con}}$,  
is the quotient algebra
$\Lambda^{M}/[S]$. 
\end{definition}

In the above definition, 
$\mathrm{add}(S)$ is the full subcategory of $\mathrm{mod}(S)$ 
consisting of direct summands of 
some direct sums of $S$, i.e., 
the subcategory of finitely generated projective $S$-modules.

From the above definition, it is direct to see that 
$\Lambda^{M}_{\mathrm{con}} \cong \Lambda^{(M \oplus Q)}_{\mathrm{con}}$ 
as algebras, for any projective $S$-module $Q$. 
By the definition of stable categories (see (\ref{Coh})), 
we immediately get the following. 

\begin{proposition}\label{CML}
Let $M$ be a Cohen-Macaulay $S$-module. Then  	
$$
\Lambda^{M}_{\mathrm{con}} \cong 
\mathrm{End}_{\underline{\mathrm{CM}}(S)}(S\oplus M) \cong 
\mathrm{End}_{\underline{\mathrm{CM}}(S)}(M)
$$
as algebras. 	
\end{proposition}

\subsection{Singular locus of a Gorenstein domain}\label{Slgd}

In this subsection, we give a description of the singular locus
of a Gorenstein domain in terms of its NCR (Theorem  \ref{Singlocus1}).
Let us start with the following.

\begin{lemma}\label{Local}
Suppose $M \in \mathrm{Ob}
\big(\underline{\mathrm{CM}}(S)\big)$ gives an NCR of $S$ and
contains a direct summand $S$ as an $S$-module. Suppose
$\mathfrak{p}$ is a prime ideal of $S$. 
Denote by 
$M_{\mathfrak{p}}$ the localization of $M$ at 
$\mathfrak{p}$.
Then
$M_{\mathfrak{p}}$ is a generator of 
$\underline{\mathrm{CM}}(S_{\mathfrak{p}})$. 
\end{lemma}

\begin{proof}
To prove this lemma, it is 
sufficient to show that  $M_{\mathfrak{p}}$ is a classical generator of 
$\underline{\mathrm{CM}}(S_{\mathfrak{p}})$. 
Since
$\underline{\mathrm{CM}}(S_{\mathfrak{p}}) 
\cong D_{sg}(S_{\mathfrak{p}})$ as triangulated categories, 
this is equivalent to showing that 
$M_{\mathfrak{p}}$ is a classical generator of 
$D_{sg}(S_{\mathfrak{p}})$.
By Proposition \ref{GeneToSing} it is enough to show that  
$M_{\mathfrak{p}}$ is a classical generator of 
$D^{b}(S_{\mathfrak{p}})$. 

From the definition of NCR, we know that  
$\Lambda := \mathrm{End}_{S}(M)$ is homologically smooth. 
Let $P^{\bullet}$ be a bounded projective $\Lambda^e$-module resolution of $\Lambda$. 
Note that $\Lambda$ is an $S$-algebra and thus $\Lambda^e$ is an $S^e$-algebra. 
Localizing 
at $\mathfrak{p}$, 
we obtain a bounded projective 
$\Lambda^e_\mathfrak{p}$-module 
resolution $P^{\bullet}_\mathfrak{p}$ of $\Lambda_\mathfrak{p}$, which 
then implies that $\Lambda_\mathfrak{p}$ is homologically smooth. 

We thus obtain a triangle equivalence
$D^{b}(\Lambda_\mathfrak{p}) \cong \mathrm{Perf}(\Lambda_\mathfrak{p})$, which gives
$D^{b}(\Lambda_\mathfrak{p})$ a classical generator $\Lambda_\mathfrak{p}$. 
Next, let $e_\mathfrak{p}$ be the indecomposable idempotent of $\Lambda_{\mathfrak{p}}$ 
corresponding to the direct summand $S_\mathfrak{p}$ of $M_\mathfrak{p}$. 
Then 
since $\mathrm{mod}(S_{\mathfrak{p}})$ is equivalent to 
the Serre quotient 
$\mathrm{mod}(\Lambda_{\mathfrak{p}})/ 
\mathrm{mod}(\Lambda_{\mathfrak{p}} e_{\mathfrak{p}} \Lambda_{\mathfrak{p}})$ 
as abelian categories,
there is the following localization functor 
$$
(-) \otimes^{\mathbb{L}}_{\Lambda_\mathfrak{p}} 
\Lambda_\mathfrak{p} e_\mathfrak{p} : 
D^{b}(\Lambda_\mathfrak{p}) \rightarrow D^{b}(S_\mathfrak{p}).
$$ 
Thus similarly to
the proof of Proposition \ref{GeneToSing},
$\Lambda_\mathfrak{p} e_\mathfrak{p} \cong M_\mathfrak{p}$
is a classical generator of $D^{b}(S_\mathfrak{p})$.
\end{proof}

\begin{remark}
With the same argument, one can show that
$M$ is a generator of 
$\underline{\mathrm{CM}}(S)$.
\end{remark}

\begin{theorem}\label{Singlocus1} 
Let $S$ be a Gorenstein normal domain, 
and $\Lambda = \mathrm{End}_{S}(\bigoplus^{m}_{i = 0} M_i)$ 
be an NCR of
$S$ such that $M_0 = S$ and $M_i$'s 
are indecomposable 
Cohen-Macaulay $S$-modules for all $i$.
Let $\varphi: S \hookrightarrow \Lambda$ be the canonical injection of algebras 
induced by the $S$-module structure of 
$\bigoplus^{m}_{i = 0} M_i$, 
and let $e \in \Lambda$ be the
idempotent corresponding to the direct summand $S$.
Then
$$
\mathrm{Sing} \big(\mathrm{Spec}(S) \big) 
= \mathrm{Spec}\Big( \varphi(S)/ \big(\varphi(S) \cap (\Lambda e \Lambda)\big)\Big).
$$
as subschemes of $\mathrm{Spec}(S)$. 
\end{theorem} 

In the above theorem,
note that since $\bigoplus^{m}_{i = 0} M_i$ 
is a reflexive $S$-module and $S$ is a domain, 
there is no 
zero divisor in $\bigoplus^{m}_{i = 0} M_i$, and therefore 
$\varphi$ is an injection.

\begin{proof}[Proof of Theorem \ref{Singlocus1}] 
(1) We first show
$$
\mathrm{Spec}\Big(\varphi(S)/ \big(\varphi(S) \cap (\Lambda e \Lambda) \big) \Big)
 \subseteq \mathrm{Sing}\big(\mathrm{Spec}(S) \big).
$$
To this end, let 
$x \in \mathrm{Spec}(S)$ be a smooth point 
which corresponds to a prime ideal 
$\mathfrak{p}$. 
Consider the regular ring $S_{\mathfrak{p}}$ and the $S_{\mathfrak{p}}$-algebra 
$\Lambda_{\mathfrak{p}}$. 
To simplify the notations, set $\overline{M}:=\bigoplus^{m}_{i = 0} M_i$. 
Since $S_{\mathfrak{p}}$ is regular
and $\overline{M}_{\mathfrak{p}}$ is a  Cohen-Macaulay $S_{\mathfrak{p}}$-module,
$ \overline{M}_{\mathfrak{p}}$ is 
free over $S_{\mathfrak{p}}$. 
Thus we have
\begin{equation}\label{eq:otoh}
\Lambda_{\mathfrak{p}} \cong \mathrm{End}_{S}(\overline{M})_{\mathfrak{p}} \cong 
\mathrm{End}_{S_\mathfrak{p}}(\overline{M}_{\mathfrak{p}}) \cong M_{r \times r}(S_{\mathfrak{p}}),
\end{equation}
where 
$r$ is rank of  
$ \overline{M}_{\mathfrak{p}}$.
Moreover, 
if we denote by $\Psi: \Lambda e \otimes_{S} e 
\Lambda \cong  \mathrm{Hom}_{S}(S, \overline{M}) \otimes_{S} 
\mathrm{Hom}_{S}(\overline{M}, S)\twoheadrightarrow 
\Lambda e \Lambda \subseteq \mathrm{End}_{S}(\overline{M})$ 
the composition
of homomorphisms,
we get
\begin{align}
(\Lambda e \Lambda)_{\mathfrak{p}} & \cong \Psi\big((\Lambda e \otimes_{S} e\Lambda)\big)_{\mathfrak{p}} 
\cong  \Psi_{\mathfrak{p}} 
\big( (\Lambda e)_{\mathfrak{p}} \otimes_{S_{\mathfrak{p}}} 
(e \Lambda)_{\mathfrak{p}} \big)\nonumber \\
 & \cong \Psi_{\mathfrak{p}} 
 \Big( \mathrm{Hom}_{S_{\mathfrak{p}}}\big( S_{\mathfrak{p}}, 
 (S_{\mathfrak{p}})^{\oplus r}\big) \otimes_{S_{\mathfrak{p}}} \mathrm{Hom}_{S_{\mathfrak{p}}}\big((S_{\mathbf{p}})^{\oplus r}, 
 S_{\mathfrak{p}}\big) \Big)\nonumber \\
 & \cong M_{r \times r}(S_{\mathfrak{p}}),\label{eq:otohd}
\end{align} 
where $\Psi_{\mathfrak{p}}$ is the localization of $\Psi$ at 
$\mathfrak{p}$. 
Combining \eqref{eq:otoh} and \eqref{eq:otohd} we obtain that 
$\Lambda_{\mathfrak{p}} / (\Lambda e \Lambda)_{\mathfrak{p}} \cong 0$. 
In the meantime, there is a natural injection
$$
\varphi(S) / \big(\varphi(S) \cap (\Lambda e \Lambda)\big) 
\hookrightarrow \Lambda / (\Lambda e \Lambda) 
$$
induced by $\varphi$. 
This injection induces the following injection after localization
$$
\varphi(S)_{\mathfrak{p}} / \big(\varphi(S)_{\mathfrak{p}} 
\cap (\Lambda e \Lambda)_{\mathfrak{p}} \big) 
\hookrightarrow \Lambda_{\mathfrak{p}} / (\Lambda e \Lambda)_{\mathfrak{p}}.
$$
Since 
$\Lambda_{\mathfrak{p}} / (\Lambda e \Lambda)_{\mathfrak{p}} \cong 0$,
we get that $ \Big( \varphi(S) / \big(\varphi(S) \cap (\Lambda e \Lambda) \big) \Big)_{\mathfrak{p}} 
\cong \varphi(S)_{\mathfrak{p}} / \big(\varphi(S)_{\mathfrak{p}} 
\cap (\Lambda e \Lambda)_{\mathfrak{p}} \big) \cong 0$.
Thus $x$ is not in the subscheme $\mathrm{Spec}\Big(\varphi(S)/ 
\big(\varphi(S) \cap (\Lambda e \Lambda) \big) \Big)$
of $\mathrm{Spec}(S)$, which implies that
$$
\mathrm{Spec}\Big(\varphi(S)/ \big(\varphi(S) \cap (\Lambda e \Lambda) \big) \Big) 
\subseteq \mathrm{Sing}\big(\mathrm{Spec}(S) \big).
$$

\noindent (2) We next show
$$
\mathrm{Sing}\big(\mathrm{Spec}(S) \big) \subseteq 
\mathrm{Spec}\Big(\varphi(S)/ \big(\varphi(S) 
\cap (\Lambda e \Lambda) \big) \Big).
$$
To this end, let $\mathfrak{q}$ be an arbitrary prime ideal of $S$ corresponding to a point in 
$\mathrm{Sing}\big( \mathrm{Spec}(S)\big)$.  
To show the above inclusion, it is sufficient to show that the algebra
$
\varphi(S)_{\mathfrak{q}} / \big(\varphi(S) \cap (\Lambda e \Lambda) \big)_{\mathfrak{q}}
$
is nontrivial. 

In fact, $\Lambda_{\mathfrak{q}} / (\Lambda e \Lambda)_{\mathfrak{q}}$ 
is an algebra over 
$ \varphi(S)_{\mathfrak{q}} / \big(\varphi(S) \cap (\Lambda e \Lambda) 
\big)_{\mathfrak{q}} \cong 
\varphi(S)_{\mathfrak{q}} / \big(\varphi(S)_{\mathfrak{q}} 
\cap (\Lambda e \Lambda)_{\mathfrak{q}}\big)$. 
Meanwhile, 
$\varphi(S)_{\mathfrak{q}} / \big(\varphi(S)_{\mathfrak{q}} 
\cap (\Lambda e \Lambda)_{\mathfrak{q}}\big)$ as a subalgebra  
contains the unit of 
$\Lambda_{\mathfrak{q}} / (\Lambda e \Lambda)_{\mathfrak{q}}$. 
Then $\varphi(S)_{\mathfrak{q}}/ (\varphi(S)_{\mathfrak{q}} 
\cap (\Lambda e \Lambda)_{\mathfrak{q}})$ is trivial 
 if and only if $\Lambda_{\mathfrak{q}} / (\Lambda e \Lambda)_{\mathfrak{q}}$ is trivial. 
Thus, to prove that 
$
\varphi(S)_{\mathfrak{q}} / \big(\varphi(S) \cap (\Lambda e \Lambda) \big)_{\mathfrak{q}}
$
is nontrivial, 
it suffices to show that 
$$
 \Lambda_{\mathfrak{q}} / (\Lambda e \Lambda)_{\mathfrak{q}} 
 \cong \Lambda_{\mathfrak{q}} / \Lambda_{\mathfrak{q}} e_{\mathfrak{q}} 
 \Lambda_{\mathfrak{q}} \cong 
\mathrm{End}_{S_{\mathfrak{q}}}( \overline{M}_{\mathfrak{q}} ) 
/ [S_{\mathfrak{q}}] \cong 
\mathrm{End}_{\underline{\mathrm{CM}}(S_{\mathfrak{q}})}( \overline{M}_{\mathfrak{q}} )
$$
is nontrivial.  
We prove this by contradiction. 

Assume that 
$\mathrm{End}_{\underline{\mathrm{CM}}(S_{\mathfrak{q}})}
(\overline{M}_{\mathfrak{q}} ) = 0$.
By Lemma \ref{Local}
we know that $ \overline{M}_{\mathfrak{q}}$ is a generator of 
$\underline{\mathrm{CM}} (S_{\mathfrak{q}})$. 
Now for any object 
$M \in \mathrm{Ob}\big(\underline{\mathrm{CM}} (S_{\mathfrak{q}})\big)$, 
$ 
\mathrm{Hom}_{\underline{\mathrm{CM}} 
(S_{\mathfrak{q}})}(\overline{M}_{\mathfrak{q}}, M)
$ 
is an $\mathrm{End}_{\underline{\mathrm{CM}} (S_{\mathfrak{q}})}
(\overline{M}_{\mathfrak{q}})$-module. 
Furthermore, the following composition
$$
 \mathrm{Hom}_{\underline{\mathrm{CM}} (S_{\mathfrak{q}})}(\overline{M}_{\mathfrak{q}}, 
M[i]) \otimes \mathrm{End}_{\underline{\mathrm{CM}} 
(S_{\mathfrak{q}})}(\overline{M}_{\mathfrak{q}}) \rightarrow 
\mathrm{Hom}_{\underline{\mathrm{CM}} (S_{\mathfrak{q}})}(\overline{M}_{\mathfrak{q}}, M[i])
$$
is surjective for any $i \in \mathbb{Z}$. It implies that 
$\mathrm{Hom}_{\underline{\mathrm{CM}} 
(S_{\mathfrak{q}})}(\overline{M}_{\mathfrak{q}}, M[i]) 
= 0$ for all $i \in \mathbb{Z}$. 
Meanwhile, since   
$\overline{M}_{\mathfrak{q}}$ is a generator of 
$\underline{\mathrm{CM}}(S_{\mathfrak{q}})$,  
we get that  $M \cong 0$ 
in $\underline{\mathrm{CM}}(S_{\mathfrak{q}})$. 
Hence, the objects in 
$\underline{\mathrm{CM}}(S_{\mathfrak{q}})$ are all trivial. 
Therefore, the triangulated category 
$\underline{\mathrm{CM}}(S_{\mathfrak{q}})$ is trivial,
and thus $S_{\mathfrak{q}}$ is a homologically smooth ring. 
However, this contradicts to that ${\mathfrak{q}}$ is 
a point in 
$\mathrm{Sing}\big(\mathrm{Spec}(S)\big)$. 
Therefore 
$
 \Lambda_{\mathfrak{q}} / (\Lambda e \Lambda)_{\mathfrak{q}} 
 \cong \underline{\mathrm{End}}_{S_{\mathfrak{q}}}
 (\overline{M}_{\mathfrak{q}})
 $
 is nontrivial, which then implies that 
\[
 \mathrm{Sing}\big(\mathrm{Spec}(S) \big) \subseteq 
 \mathrm{Spec}\Big(\varphi(S)/ \big(\varphi(S) \cap (\Lambda e \Lambda) \big) \Big). 
\qedhere\]
\end{proof}

By taking the reduced schemes 
of the isomorphism in Theorem \ref{Singlocus1}, 
we immediately get the following (in this paper,
for a scheme $X$, we use $\sqrt{X}$ to denote
its reduced scheme):

\begin{corollary}\label{cor:reducedloci}
The isomorphism in Theorem  \ref{Singlocus1}
induces the following isomorphism
$$\sqrt{ \mathrm{Sing}\big(\mathrm{Spec}(S) \big)}\cong
 \mathrm{Spec}\big(\varphi(S)/\sqrt{\varphi(S) \cap (\Lambda e \Lambda)} \big).
$$
\end{corollary}

\subsection{The singular locus of $\mathrm{Spec}(S)$}\label{subsect:singularlocus}

From now on, we focus on the case in Example \ref{keyexample}. 
For reader's convenience, let us list some notations that we will be repeatedly
using:

\begin{notation}\label{list:notations}
\begin{itemize}
\item[$-$]
Let $\widehat{M} := \bigoplus_{W \in \hat{G}}(W \otimes R)^G$, which,  
by Auslander's theorem, is 
isomorphic to $R$ as $S$-modules. 

\item[$-$] Let $\Lambda := \mathrm{End}_{S}(\widehat{M})$ and
let
$e$ be the indecomposable idempotent of $\Lambda$ corresponding to the direct summand 
$S$ of $\widehat{M}$. Note that when $W$ is the  trivial $G$-representation, 
then $(W \oplus R)^G = R^G = S$. 

\item[$-$] Let $\varphi: S \hookrightarrow \Lambda$ be the canonical injection of algebras 
induced by the $S$-module structure on
$\widehat{M}$.

\item[$-$]  Let
$\Lambda_{\mathrm{con}} := \Lambda^{\widehat{M}}_{\mathrm{con}} \cong 
\Lambda /[S] \cong \mathrm{End}_{\underline{\mathrm{CM}}(S)}
( \bigoplus_{\chi} M_\chi )$ (see Definition \ref{Connn} and Proposition \ref{CML}). Here, $[S]$ is the 
ideal of $\mathrm{End}_{S}(\widehat{M})$ consisting of homomorphisms which factor
through some finitely generated projective $S$-module. Note that 
$\Lambda_{\mathrm{con}} \cong \Lambda / \Lambda e \Lambda$.

\item[$-$] 
Let $R^{\widehat M}$ be the reduced ring of the center $Z(\Lambda_{\mathrm{con}})$ of 
$\Lambda_{\mathrm{con}}$. 
\end{itemize}
\end{notation}

The purpose of this subsection
is to describe the reduced singular locus of $\mathrm{Spec}(S)$ 
by means of $\widehat M$.
In Theorem \ref{HatM} we show that $\mathrm{Spec}(R^{\widehat{M}}) \cong 
\sqrt{\mathrm{Sing}\big(\mathrm{Spec}(S)\big)}$ as subschemes of $\mathrm{Spec}(S)$. 
Recall that from Corollary \ref{cor:reducedloci},
$
\sqrt{\mathrm{Sing} \big(\mathrm{Spec}(S) \big) }
= \mathrm{Spec}\big( \varphi(S)/\sqrt{\varphi(S) \cap (\Lambda e \Lambda)}\big)
$
as subschemes of $\mathrm{Spec}(S)$. Thus by 
combining these two theorems and considering the associated rings, 
we obtain that
$
\varphi(S)/ \sqrt{\varphi(S) \cap (\Lambda e \Lambda)} \cong R^{\widehat{M}},
$
which is in fact 
induced by the natural map $\varphi: S \rightarrow \Lambda$ (see Theorem \ref{Theo22}).

\subsubsection{Characters of $G$ and components of the singular locus}\label{SSG}

Recall that the characters of $G$ are in one-one correspondence 
with the indecomposable idempotents of the group algebra $kG$, which 
is further in
one-one correspondence with the indecomposable 
idempotents of $\Lambda$. 
Let $\chi: G \rightarrow k^{\ast}$ be a character of $G$.
Denote the corresponding idempotent by $e_\chi$, then
$e_{\chi} =  \frac{1}{|G|} \sum\limits_{g \in G} \chi(g) (g \otimes 1) \in \Lambda $. 
Let $\chi_0$ be the trivial character of $G$. 
It is direct to see that $e$ is the indecomposable 
idempotent corresponding to $\chi_0$.  

Since $G$ is a finite abelian subgroup of $\mathrm{SL}(V)$, $G$ 
can be viewed as a group consisting of diagonal matrices in $\mathrm{SL}(V)$. 
Fix such a diagonalization for $G$, and denote
by $\{E_i \}_{1 \leq i \leq n}$ the basis of $V$ with respect to this diagonalization. 
Notice that
any $E_i$ is associated to a character, denoted by $\chi_{E_i}$, of $G$, 
given by $\chi_{E_i}(g) = g_{i} \in k^*$ for any $g \in G$, 
where $g(E_i) = g_i E_i$ and $g(-)$ is the action of $g$ on $V$.  

Dually, let $R = k[V] = k[x_{1}, \cdots, x_n]$, 
where $\{x_i\}$ is the set of basis dual to $\{E_i\}$.
Then each $x_i$ is also associated to a character, denoted by 
$\chi_{x_i}$ of $G$, such that $\chi_{x_i} = \chi_{E_i}^{-1}$. 
Moreover, for any monomial $f \in R$, 
define a character $\chi_f$ of $G$ as $\chi_{f}(g):= 
g_f \in k$, for any $g \in G$, where $g(f) = g_f f$ and 
$g(-)$ is the dual action of $g$ on $R = k[V]$. 

For any character $\chi$ of $G$, let 
$V_{\chi}$ be the one dimensional irreducible representation of $G$ given by
$\chi: G \rightarrow k^{\ast}$. 
Let 
$M_{\chi} := (V_{\chi} \otimes R)^G$ be an 
indecomposable direct summand of $\widehat{M}$. 
Note that $(V_{\chi} \otimes R)^G \subseteq R$ as $S$-modules. 
Moreover, since 
$(V_{\chi} \otimes R)^G = \mathrm{Hom}_{\mathrm{ref}(k, G)}(V_{\chi^{-1}}, R)$, 
$M_{\chi}$ is generated by monomials like $f$ such that $\chi_f = \chi^{-1}$  as $S$-modules.

At the same time, 
it is obvious that for any monomial $f \in R$,
\begin{align}
e_{\chi} (1 \otimes f) e_{\chi \chi_{f}} 
= e_{\chi} e_{\chi} (1 \otimes f) = e_{\chi} (1 \otimes f)\label{Comppp} 
\quad\mbox{and}\quad
(1 \otimes f) e_{\chi} = e_{\chi \chi_{f}^{-1}} (1 \otimes f) e_{\chi},
\end{align}
where $\chi\chi_f$ and $\chi\chi_f^{-1}$ are the products of characters.

Furthermore, let $\Lambda e_{\chi} \subseteq \Lambda$ be product of $\Lambda$ with $e_\chi$. 
Then we have $\Lambda e_{\chi} = 
(G \sharp R)e_{\chi} \cong R \cong \widehat{M}$ as $S$-modules. 
In the meantime, since $e \in \Lambda$ is the idempotent corresponding to the
summand $S$ of $\widehat{M}$, 
\begin{align}\label{fajb62469}
\Lambda e \cong \mathrm{Hom}_{S}(S, \widehat{M}) \cong \widehat{M}
\end{align}
as $\Lambda^{op}$-modules. Hence, we may identify 
$\Lambda e$ with $\widehat{M}$ as $S \otimes \Lambda^{op}$-modules. 
 
Consider the direct summand 
$e_{\chi} \widehat{M} :=  e_\chi \Lambda \otimes_{\Lambda} \widehat{M}$  
of $\widehat{M}$ for any $e_\chi$. 
We have
\begin{align*}
e_\chi \widehat{M} & \cong e_\chi \Lambda e  \cong e_\chi (G \sharp R) e
 \cong (e_\chi \otimes R) e 
 \cong (e_\chi \otimes \widehat{M}) e 
 \\
& \cong \Big(e_\chi \otimes \big(\bigoplus_{\lambda} M_{\lambda}\big)\Big) e \cong 
\bigoplus_{\lambda} \Big(e_\chi \otimes  M_{\lambda}\Big) e \\
& \cong \bigoplus_{\lambda} \big(e_{\chi} e_{\lambda} \otimes M_{\lambda} \big) \cong e_{\chi} \otimes M_{\chi} \cong M_{\chi} 
\end{align*}
as $S$-modules.
In above identities, we identify the idempotents of 
$\Lambda$ with the ones of $kG$.
Thus, for any two characters $\chi$ and $\chi'$ of $G$,
\begin{align}
e_{\chi} \Lambda e_{\chi'} & \cong e_{\chi} \Lambda \otimes_{\Lambda} 
\mathrm{End}_{S}( \widehat{M} ) \otimes_{\Lambda} \Lambda e_{\chi'} \nonumber \\ 
& \cong \mathrm{Hom}_{S}\big(e_{\chi'} \Lambda \otimes_{\Lambda} \widehat{M}, 
e_{\chi} \Lambda \otimes_{\Lambda} \widehat{M} \big) \nonumber \\  
& \cong \mathrm{Hom}_{S}(M_{\chi'}, M_{\chi}) \label{Che222}
\end{align} 
as $S$-modules. In particular, $e_{\chi} \Lambda e_{\chi} \cong S$ for any $\chi$ 
(see (\ref{Fir789342})), and thus in the following we may identify 
$e_{\chi} \Lambda e_{\chi}$ with $S$. Note that the inverse isomorphism  
$ S \cong e_{\chi} \Lambda e_{\chi}$ maps $f \in S$ to 
$e_\chi(1 \otimes f)e_{\chi} \in e_{\chi} \Lambda e_{\chi}$. 

Replacing $e_\chi$ by  $e$, by the above argument 
we get that $e \widehat{M} \cong M_{\chi_0} \cong S$ as $S$-modules. 
It is direct to see that  
\begin{align}\label{AHYG6739}
e_{\chi} \Lambda e \cong \mathrm{Hom}_{S}(S, M_{\chi}) 
\end{align}
as $S$-modules. 
By directly summing up all $\{M_{\chi}\}_{\chi}$, it follows that  
$e \Lambda \cong \mathrm{Hom}_{S}(\widehat{M}, S)$ as $\Lambda \otimes S^{op}$-modules.
Now recall that from the proof of Theorem \ref{Singlocus1},
$\Psi: \Lambda e \otimes_{S} e \Lambda \rightarrow \Lambda e \Lambda$ 
is given by the composition 
map of $\mathrm{End}_{S}(\widehat{M})$. (Note $\Psi$ is a 
surjection.)
From the above two isomorphisms (\ref{fajb62469}) and (\ref{AHYG6739}) we then obtain that 
$$
\Lambda e \Lambda  \cong \Psi( \Lambda e \otimes_{S} e \Lambda ) \cong \Psi\big(\mathrm{Hom}_{S}(S, \widehat{M})
\otimes_{S} \mathrm{Hom}_{S}(\widehat{M}, S)\big)  
 \cong [S] \subseteq \Lambda 
$$
as $\Lambda^e$-modules, and 
then $\Lambda_{\mathrm{con}} \cong \Lambda  \big/  \Lambda e \Lambda $ 
as $\Lambda^e$-algebras. 
Thus, by $\Lambda e \Lambda \cong [S]$ and by
Proposition \ref{CML}, we get that 
\begin{align}\label{SGUB8735}
e_{\chi} \Lambda_{\mathrm{con}} e_{\chi} & \cong e_{\chi} \Lambda e_{\chi} 
/ e_{\chi} \Lambda e \Lambda e_{\chi} \cong \mathrm{End}_{S}(M_{\chi}) \big/ 
\big( \mathrm{End}_{S}(M_{\chi}) \cap \Lambda e \Lambda \big)  \nonumber \\ 
& \cong \mathrm{End}_{S}(M_{\chi}) \big/ 
\big( \mathrm{End}_{S}(M_{\chi}) \cap [S] \big)  
 \cong \Lambda^{M_{\chi}}_{\mathrm{con}} \cong 
\mathrm{Hom}_{\underline{\mathrm{CM}}(S)}(M_{\chi}, M_{\chi})
\end{align}
as $S$-algebras.

\begin{lemma}\label{IdemSpli}
Let $\bar{I}_{\chi} := e_{\chi} \Lambda e \Lambda e_{\chi} \subseteq e_{\chi} \Lambda e_{\chi} \cong S$ be the ideal of $S$ and 
$I_{\chi} : = \sqrt{\bar{I}_{\chi}} \subseteq S$  
be the radical of $\bar{I}_{\chi}$. 
Then 
$$
\mathrm{Spec}\big(S/ I_{\chi}\big) \subseteq \sqrt{\mathrm{Sing}\big(\mathrm{Spec}(S)\big)}
$$
as subschemes of $\mathrm{Spec}(S)$.
\end{lemma}

\begin{proof}
First, since $\mathrm{Spec}(S) = V/G$ 
is an irreducible scheme,
it is easy to see that $S$ is an equidimensional finitely generated ring over $k$.   
Set ideals
$$ 
\mathrm{ann}_{S}(X):= \{f \in S \, \, | \, \,  \mathrm{End}_{D_{sg}(S)}(X) \, f = 0  \} 
$$
for any object $X$ of $D_{sg}(S)$,  
and 
$$
\mathrm{ann}_{S}\big(D_{sg}(S)\big) := \bigcap\limits_{X \in D_{sg}(S)} \mathrm{ann}_{S}(X) 
$$
(see \cite[\S3]{Liu}). Here, we view
$\mathrm{End}_{D_{sg}(S)}(X)$ as an $S$-modules.   
From \cite[Corollary 4.9]{Liu}, we know that  
$$
\sqrt{\mathrm{Sing}\big(\mathrm{Spec}(S)\big)} =
\mathrm{Spec}\Big(S/ \sqrt{\mathrm{ann}_{S}\big(D_{sg}(S)\big)} \Big). 
$$
Therefore, to prove this lemma, it suffices to show that 
$$
\sqrt{\mathrm{ann}_{S}(D_{sg}(S))} \subseteq I_{\chi},
$$
which is further enough to show that 
$$
\mathrm{ann}_{S}(D_{sg}(S)) \subseteq e_{\chi} \Lambda e \Lambda e_{\chi}. 
$$
Since $D_{sg}(S) \cong \underline{\mathrm{CM}}(S)$ 
and $\mathrm{End}_{\underline{\mathrm{CM}}(S)}(M_{\chi}) \cong S/ e_{\chi} \Lambda e \Lambda e_{\chi}$ (by (\ref{SGUB8735}) above), 
we have 
$
e_{\chi} \Lambda e \Lambda e_{\chi} = \mathrm{ann}_{S}(M_\chi)
$.
This implies that 
$
\mathrm{ann}_{S}(D_{sg}(S)) \subseteq e_{\chi} \Lambda e \Lambda e_{\chi}
$,
and the lemma follows. 
\end{proof}

\subsubsection{The quiver description}

In this paper, we shall also use the quiver 
presentations of $\Lambda$ 
and $\Lambda_{\mathrm{con}}$, which we denote by $Q_\Lambda$ and 
$Q_{\mathrm{con}}$ respectively. 
Indeed, 
$Q_\Lambda$ is the McKay quiver associated to the faithful representation $V$ of $G$,
and
$Q_{\mathrm{con}}$ is obtained from $Q_\Lambda$ by removing 
the vertex that corresponds to idempotent $e$ and the arrows that start from or end at this vertex from $Q_{\Lambda}$. 
For any indecomposable idempotent $e'$, 
we also use $e'$ to denote its corresponding vertex. 

\begin{proposition}\label{Cl11}
$Q_{\mathrm{con}}$ is connected. 	
\end{proposition}

\begin{proof}We prove by contradiction.
Assume that $Q_{\mathrm{con}}$ is not connected. 
Decompose $Q_{\mathrm{con}}$ 
into two disjoint quivers $Q_{\mathrm{con}}^1$ and $Q_{\mathrm{con}}^2$.

Note that $Q_{\mathrm{con}}$ is obtained from $Q_\Lambda$ by removing the vertex  
$e$ and the arrows that start from or end at $e$. 
Since $Q_\Lambda$ is the McKay quiver associated to the faithful representation $V$ of $G$,  
it is well-known that $Q_\Lambda$ is connected (see, for example, \cite[Theorem 3.1]{SMRM}).
This means that, as subquivers of $Q_\Lambda$,
$Q_{\mathrm{con}}^1$ and $Q_{\mathrm{com}}^2$
are connected only by paths that
go through the vertex $e$ in $Q_{\Lambda}$.

Now since $\Lambda = G \sharp R$ and $\hat{G}$ consists of 
one-dimensional representations, 
each arrow in $Q_{\Lambda}$ is  
represented by some formal variable $x_i$, and hence each arrow in $Q_{\mathrm{con}}$ is also
given by $x_i$.

By (\ref{Comppp}), for any $x_i$ in $\{x_r \}_r$, 
$e$ is connected to $e_{\chi_{x_i}}$ by an arrow, denoted by $\iota^e_{x_i}$,  
corresponding to 
$e(1 \otimes x_i)e_{\chi_{x_i}} \in \Lambda$. 
Moreover,  
$e$ is connected with $e_{\chi_{x_i}^{-1}}$ by an arrow, denoted by $\tau^e_{x_i}$,  
corresponding to 
$e_{\chi_{x_i}^{-1}}(1 \otimes x_i)e \in \Lambda$.
Now, by the connectedness of $Q_{\Lambda}$, 
it is easy to see that there are only two types of possible
relations between $\{\iota^e_{x_r}\}_r$ and $\{\tau^e_{x_r}\}_r$: 

\begin{enumerate}
\item[(1)]  $\chi_{x_i} = \chi_{x_j}$ 
for any $x_i \neq x_j$. We then have that for any $x_i$, either arrow $\tau^e_{x_i}$ 
	starts from some vertex 
in $Q_{\mathrm{con}}^2$ and $\iota^e_{x_i}$ 
	ends at some vertex in $Q_{\mathrm{con}}^1$,  
or arrow $\tau^e_{x_i}$ 
	starts from some vertex 
in $Q_{\mathrm{con}}^1$ and $\iota^e_{x_i}$ 
	ends at some vertex in $Q_{\mathrm{con}}^2$. 
	
\item[(2)] There are two $x_i$ and $x_j$ such that $\chi_{x_i} \neq \chi_{x_j}$ and 
one of the following holds: 
\begin{enumerate}	
\item[(i)] $\tau^e_{x_i}$ 
	starts from some vertex 
in $Q_{\mathrm{con}}^2$ and $\iota^e_{x_j}$ 
	ends at some vertex in $Q_{\mathrm{con}}^1$; 

\item[(ii)] $\tau^e_{x_i}$ 
	starts from some vertex 
in $Q_{\mathrm{con}}^1$ and $\iota^e_{x_j}$ 
	ends at some vertex in $Q_{\mathrm{con}}^2$.
\end{enumerate}	
\end{enumerate}

For case $(1)$, since $\chi_{x_i} = \chi_{x_j}$ 
for any $x_i \neq x_j$, the group $G$ must be the acyclic group 
$\mathrm{diag}(\epsilon, \epsilon, \cdots, \epsilon)$ with the natural action on $V$, 
where $\epsilon$ is an $n$-th root of unit. 
Then, in this case, 
the underlying graph of $Q_{\Lambda}$ is 
a cycle with length $n$, which
implies that $Q_{\mathrm{con}}$ is connected.  
It contradicts to our assumption. 

For the subcase (i) in case (2),  
the arrow $\tau^e_{x_i}$ starts from the vertex 
$e_{\chi_{x_i}^{-1}} \in Q_{\mathrm{con}}^1$ and 	
$\iota^e_{x_j}$ ends at the vertex 
$e_{\chi_{x_j}} \in Q_{\mathrm{con}}^2$. 
In the meantime, $x_i$ gives an arrow which starts from the vertex 
$e_{\chi_{x_j} \chi_{x_i}^{-1}}$ and ends at the vertex 
$e_{\chi_{x_j}}$ by (\ref{Comppp}). 
Moreover, $x_j$ gives an arrow which starts 
from the vertex  
$e_{\chi_{x_i}^{-1}}$ and 
ends at the vertex $e_{\chi_{x_j} \chi_{x_i}^{-1}}$. 
Thus, due to $e_{\chi_{x_i}^{-1}} \in Q_{\mathrm{con}}^1$ 
and $e_{\chi_{x_j}} \in Q_{\mathrm{con}}^2$, 
$Q_{\mathrm{con}}^1$ is connected with $Q_{\mathrm{con}}^2$ through 
the vertex $e_{\chi_{x_j} \chi_{x_i}^{-1}}$ in 
$Q_{\Lambda}$. 
But $e_{\chi_{x_j} \chi_{x_i}^{-1}} \neq e$ due to 
$\chi_{x_i} \neq \chi_{x_j}$. This means there is a path which 
connects $Q_{\mathrm{con}}^1$ and $Q_{\mathrm{con}}^2$ and does not   
pass through the vertex $e$. 
It follows that 
$Q_{\mathrm{con}}^1$ and $Q_{\mathrm{con}}^2$ 
are not disjoint in $Q_{\mathrm{con}}$. 
It contradicts to our assumption again.  
 
For the subcases (ii) in (2),
applying the same method, it contradicts to our assumption, too. 
The proposition now follows.
\end{proof}

This proposition implies that for any two characters $\chi'$ and $\chi''$, there is a path
in $Q_{\mathrm{con}}$,
connecting the vertices $e_{\chi'}$ and $e_{\chi''}$.

\subsubsection{Decomposition of $\Lambda_{\mathrm{con}}$}\label{Dec9822415684}

In this subsection, we 
study the relations between the coordinate ring
of the irreducible components
of the singular locus and the algebra $R^{\widehat{M}}$.

Recall that by Proposition \ref{Cl11}, $Q_{\mathrm{con}}$ 
is connected. It suggests that $\Lambda_{\mathrm{con}}$ 
cannot be viewed as a direct sum of algebras.
Nevertheless,
in this subsection we
construct an algebra homomorphism $\zeta_{0}$ (see (\ref{Lab345})) from   
$\Lambda_{\mathrm{con}}$ to a direct sum of some skew group algebras (see Lemma \ref{Key2}).
This algebra homomorphism induces an injection of algebras 
(see Theorem \ref{Key4}), which, interpreted geometrically, says that
the irreducible components of $\sqrt{\mathrm{Sing}\big(\mathrm{Spec}(S)\big)}$
gives a cover of $\mathrm{Spec}(R^{\widehat{M}})$ (see Lemma
 \ref{KeyPr}).

To this end, let us first introduce a finite set $\{(H, \chi)\}$, where
$H$ is a nontrivial subgroup of $G$ such that
it is maximal among all subgroups which
have the same invariant subspace as that of $H$,
and $\chi$ is a character of $G$.
Let $\tilde{G} : = \{(H, \chi)\} /\!\sim$, where
$(H, \chi) \sim (H', \chi')$ if and only if $H = H'$ in $G$ and $\chi'|_{H} = \chi|_{H}$.  
Let $\tilde{G}_0 \subseteq \tilde{G}$ be the subset consisting of classes $(H, \chi)$ such 
that $\chi|_{H}$, the restriction of $\chi$ on $H$, is nontrivial. 

Suppose $(H, \chi)$ represents an element in $\tilde{G}_0$.
Let $W_H$ be the $H$-invariant subspace of $V$.
Let $R_H:=k[W_H]$ and $S_H:=R_H^{G/H}$.
We are going to show in Theorem \ref{Key4}
that there is an injection of algebras
$$\bar\zeta_0: R^{\widehat{M}}\to\bigoplus_{(H,\chi)\in\tilde G_0}S_H.$$
We proceed to construct $\bar\zeta_0$ in several steps,
which is the composition of several homomorphisms (from the first to the
fifth homomorphisms below).
Let us start with some preparations.

Fix a pair $(H, W_H)$ as above. There is a set 
$\{\lambda_{i}^{H}\}_i$ of characters of $G$ consisting of all characters satisfied that $(H, \lambda_{i}^{H}) \sim (H, \chi_0)$. 
It means that $\lambda_{i}^{H}(H) = \{1 \}$.
For any $\lambda^{H}_i$, it induces a character 
of $G/H$. In the meantime, by the group homomorphism $G \rightarrow G/H $, 
any character of $G/H$ 
gives a character of $G$. Furthermore, it is obvious that there is a 
bijection between the characters in $\{\lambda_{i}^{H}\}_i$ and the 
characters of $G/H$. Hence, we also use 
$\lambda_{i}^{H}$ to represents the induced 
character of $G/H$ for simplicity.

Since for any pair $(H, W_H)$, $H$ is also viewed as a group consisting of 
diagonal matrices in $\mathrm{SL}(V)$, 
the vector space $W_H$ has set of a basis,
denoted by $\mathfrak{S}_H$, which is a subset 
of $\{ E_i\}_i$. 
Then there is a $G$-subrepresentation 
$W'_{H} \subseteq V$, which is spanned by 
vectors in 
$\{ E_i\}_i \big\backslash \mathfrak{S}_H$. Hence,  we have  
$W'_{H} \cong V/ W_H$ and 
$
V \cong W_H \oplus W'_H 
$
as $G$-representations.

Now, for any two characters 
$\lambda^{H}_i, \lambda^{H}_j$, by Lemma \ref{Equi}, we have 
\begin{align}
\mathrm{Hom}_{S}(M_{\lambda^{H}_i}, M_{\lambda^{H}_j}) 
& \cong \left(  V_{(\lambda^{H}_i)^{-1} \lambda^{H}_j}\otimes k[W_H ]\otimes k[W'_H]^H\right)^{G/H} \label{Lab11}
\end{align}
as $S$-modules. 
Recall that $R_H = k[W_H]$
and $S_H = R_{H}^{G/H}$. 
Let $M^H_{\lambda^{H}_i} := (V_{\lambda^{H}_i} \otimes R_H)^{G/H}$ 
and $M^H_{\lambda^{H}_j} := (V_{\lambda^{H}_j} \otimes R_H)^{G/H}$, then
by Lemma \ref{Equi} again, we have
\begin{align}\label{lab22}
\mathrm{Hom}_{S_H}(M^H_{\lambda^{H}_i}, M^H_{\lambda^{H}_j}) 
\cong (V_{(\lambda^{H}_i)^{-1} \lambda^{H}_j} \otimes R_H)^{G/H}
\end{align}
as $S_H$-modules.

Consider the following natural $G/H$-equivariant surjection 
$$ 
V_{(\lambda^{H}_i)^{-1} \lambda^{H}_j}\otimes R_H \otimes k[W'_H]^H 
\twoheadrightarrow 
V_{(\lambda^{H}_i)^{-1} \lambda^{H}_j} \otimes R_H \otimes k 
\cong V_{(\lambda^{H}_i)^{-1} \lambda^{H}_j} \otimes R_H,
$$ 
which is induced by the canonical surjection $k[W'_H]^H \twoheadrightarrow k$. 
Combining with (\ref{Lab11}) and (\ref{lab22}), 
we obtain the following surjection of algebras  
\begin{align*}
\tau_{\lambda^{H}_i, \lambda^{H}_j}^{H}: 
& \mathrm{Hom}_{S}(M_{\lambda^{H}_i}, M_{\lambda^{H}_j}) \cong 
\left( V_{(\lambda^{H}_i)^{-1} \lambda^{H}_j}\otimes 
R_H \otimes k[W'_H]^H\right)^{G/H} \\ 
& 
\twoheadrightarrow (V_{(\lambda^{H}_i)^{-1} \lambda^{H}_j} 
\otimes R_H)^{G/H} \cong 
\mathrm{Hom}_{S_H}(M^H_{\lambda^{H}_i}, M^H_{\lambda^{H}_j}).
\end{align*} 
Now let $\widehat{M}^H := \bigoplus_{i} M_{\lambda^{H}_i}$, we have
the following algebra surjection: 
\begin{align}\label{TauH}
\tau^H := \bigoplus_{i, j} \tau_{\lambda^{H}_i, \lambda^{H}_j}^H : &
\bigoplus_{i, j} \mathrm{Hom}_{S}\big( M_{\lambda^{H}_i}, M_{\lambda^{H}_j} \big)
\nonumber\\
& \twoheadrightarrow \bigoplus_{i, j} \mathrm{Hom}_{S}\big( M^{H}_{\lambda^{H}_i}, 
M_{\lambda^{H}_j} \big) \cong \mathrm{Hom}_{S_H}(\widehat{M}^H, \widehat{M}^H).
\end{align}

Now let $M'$ be a Cohen-Macaulay $S$-module such that $M$ is a direct summand of $M'$, i.e., 
$M' \cong M \oplus N$ as $S$-modules for some 
$S$-module $N$. 
Then there is a natural $S$-module injection
$M \hookrightarrow M'$ and a natural $S$-module surjection $M' \twoheadrightarrow M$. 
It follows that there is an $S$-module surjection  
$$
\mathrm{Hom}_{S}(M', M') \twoheadrightarrow \mathrm{Hom}_{S}(M, M).
$$ 
Now, let $M'$ be $\widehat{M}$ and $M$ be $\bigoplus_i  M_{\lambda^{H}_i}$. 
Then there is an $S$-module surjection 
\begin{align}\label{PrHHH}
Pr_H:
\mathrm{Hom}_{S}(\widehat{M}, \widehat{M}) \twoheadrightarrow \mathrm{Hom}_{S}
\big(\bigoplus_i  M_{\lambda^{H}_i}, \bigoplus_i  M_{\lambda^{H}_i} \big) \cong \bigoplus_{i, j} 
\mathrm{Hom}_{S}\big( M_{\lambda^{H}_i}, M_{\lambda^{H}_j} \big).
\end{align} 

With these preparations, in the following, we introduce
five algebra homomorphisms, one is based on the other,
and finally reach the homomorphism given in Theorem \ref{Key4}.

\begin{definition}[The first homomorphism]
Let   
\begin{align}\label{FirstM}
\zeta_{H}:
\Lambda \cong \mathrm{Hom}_{S}(\widehat{M}, \widehat{M})  \twoheadrightarrow 
\mathrm{Hom}_{S_H}(\widehat{M}^H, \widehat{M}^H)
\end{align}
be the composition of $\tau^H \circ Pr_{H}$.
\end{definition}

\begin{lemma}\label{Keyy11}
$\zeta_{H}$ is a surjection of algebras.
\end{lemma}

The proof is quite long, and is postponed to Appendix \ref{App:A}.

\begin{remark}Similarly to the definition of $\tau^H$ given by
\eqref{TauH}, we give an algebra injection,
which we denote by $\rho^H$, as follows.
Start with the natural $G/H$-equivariant injection 
$$
V_{(\lambda^{H}_i)^{-1} \lambda^{H}_j} \otimes R_H 
\cong V_{(\lambda^{H}_i)^{-1} \lambda^{H}_j} 
\otimes R_H \otimes k
 \hookrightarrow V_{(\lambda^{H}_i)^{-1} \lambda^{H}_j}\otimes R_H \otimes k[W'_H]^H,
$$ which is induced by the 
canonical injection $k \hookrightarrow k[W'_H]^H$. 
By (\ref{Lab11}) and (\ref{lab22}), we obtain the following injection of algebras 
\begin{align}\label{RHO08538} 
\rho_{\lambda^{H}_i, \lambda^{H}_j}^{H}: & \mathrm{Hom}_{S_H}(M^H_{\lambda^{H}_i}, 
M^H_{\lambda^{H}_j}) \cong 
(V_{(\lambda^{H}_i)^{-1} \lambda^{H}_j} \otimes R_H)^{G/H}\nonumber \\
& \hookrightarrow 
\left(  V_{(\lambda^{H}_i)^{-1} \lambda^{H}_j}\otimes R_H \otimes k[W'_H]^H\right)^{G/H} 
\cong \mathrm{Hom}_{S}(M_{\lambda^{H}_i}, M_{\lambda^{H}_j}). 
\end{align}
Let $\widehat{M}^H= \bigoplus_{i} M_{\lambda^{H}_i}$ as before, we have an algebra injection:
\begin{align}\label{RhoH}
\rho^H := \bigoplus_{i, j} \rho_{\lambda^{H}_i, \lambda^{H}_j}^H : &
\mathrm{Hom}_{S_H}(\widehat{M}^H, \widehat{M}^H) \cong \bigoplus_{i, j} 
\mathrm{Hom}_{S}\big( M^{H}_{\lambda^{H}_i}, M_{\lambda^{H}_j} \big)
\nonumber\\
& \hookrightarrow \bigoplus_{i, j} \mathrm{Hom}_{S}
\big( M_{\lambda^{H}_i}, M_{\lambda^{H}_j} \big).
\end{align}
\end{remark}

In the following, we fix a character $\chi$ of $G$  such that $\chi(H) \neq \{1 \}$. Let $\{\lambda_{i}^{(H, \chi)}\}_i$ be the 
set of characters of $G$ consisting of all characters 
such that $(H, \chi) \sim (H, \lambda_{i}^{(H, \chi)})$. 
On the one hand, for any $\lambda_{i}^{(H, \chi)}$,  
$\lambda_{i}^{(H, \chi)} \chi^{-1}$ is in $\{\lambda_{i}^{H}\}_i$. On the other hand, 
for any $\lambda_{i}^{H}$,  
$\lambda_{i}^{H} \chi$ is in $\{\lambda_{i}^{(H, \chi)}\}_i$. Thus
it is obvious that 
there is a bijection between characters in 
$\{\lambda_{i}^{(H, \chi)}\}_i$ 
and characters in $\{\lambda_{i}^{H}\}_i$. 
In the meantime, by Lemma \ref{Equi} we have 
\begin{align}
\mathrm{Hom}_{S}(M_{\lambda^{H}_i},   M_{\lambda^{H}_j} ) & \cong 
\mathrm{Hom}_{kG}\big((\lambda^{H}_j)^{-1} \lambda^{H}_i,  R \big) \nonumber \\
& \cong \mathrm{Hom}_{kG}\big((\lambda^{H}_j \chi)^{-1} \lambda^{H}_i \chi,  R \big) \nonumber \\
&  \cong \mathrm{Hom}_{S}(M_{\lambda^{(H, \chi)}_i},   M_{\lambda^{(H, \chi)}_j} ) \label{Lab333}
\end{align}
as $S$-modules, where $\lambda^{(H, \chi)}_i = 
\lambda^{H}_i \chi$ and $\lambda^{(H, \chi)}_j = \lambda^{H}_j \chi$.
Furthermore, there is a commutative diagram of $S^e$-module 
homomorphisms:
$$  
\xymatrix{
\mathrm{Hom}_{S}(M_{\lambda^{H}_i},   M_{\lambda^{H}_j} ) 
\times \mathrm{Hom}_{S}(M_{\lambda^{H}_j},   M_{\lambda^{H}_r} ) 
\ar[d]^{\wr}  \ar[r] & \mathrm{Hom}_{S}(M_{\lambda^{H}_i},   M_{\lambda^{H}_r}) \ar[d]^{\wr} 
\\
\mathrm{Hom}_{S}(M_{\lambda^{(H, \chi)}_i},   
M_{\lambda^{(H, \chi)}_j} ) \times \mathrm{Hom}_{S}(M_{\lambda^{(H, \chi)}_j},   
M_{\lambda^{(H, \chi)}_r} ) \ar[r]  & \mathrm{Hom}_{S}(M_{\lambda^{(H, \chi)}_j},   M_{\lambda^{(H, \chi)}_r} ),  
}
$$ 
where the vertical morphisms are given by  the
isomorphisms in (\ref{Lab333}) and the horizontal morphisms 
are the compositions of homomorphisms. 
Hence, we get that
\begin{equation}\label{eq:iso1}
 \mathrm{Hom}_{S}\big(\bigoplus_i  M_{\lambda^{(H, \chi)}_i}, 
 \bigoplus_i  M_{\lambda^{(H, \chi)}_i} \big) \cong 
 \mathrm{Hom}_{S}\big(\bigoplus_i  M_{\lambda^{H}_i}, \bigoplus_i  M_{\lambda^{H}_i} \big), 
\end{equation}
as algebra. Denote this homomorphism by $\Delta_{(H, \chi)}$.

Composing the isomorphism in  
\eqref{eq:iso1} with the canonical projection: 
$$
\mathrm{Hom}_{S}(\widehat{M}, \widehat{M}) \twoheadrightarrow \mathrm{Hom}_{S}
\big(\bigoplus_i  M_{\lambda^{(H, \chi)}_i}, \bigoplus_i  M_{\lambda^{(H, \chi)}_i} \big), 
$$ 
which is denoted by $Pr_{(H, \chi)}$, we obtain the
$S$-module surjection 
$$
\mathrm{Hom}_{S}(\widehat{M}, \widehat{M}) \twoheadrightarrow \mathrm{Hom}_{S}
\big(\bigoplus_i  M_{\lambda^{H}_i}, \bigoplus_i  M_{\lambda^{H}_i} \big),
$$
which is denoted by $\Pi_{(H, \chi)}$.

\begin{definition}[The second homomorphism]
Let
\begin{align}\label{Seccc}
\zeta^{\chi}_{H}:
\Lambda \cong \mathrm{Hom}_{S}(\widehat{M}, \widehat{M})
\to 
\mathrm{Hom}_{S_H}(\widehat{M}^H, \widehat{M}^H)
\end{align}
be $\tau^H \circ \Pi_{(H, \chi)}$.
\end{definition}

Note that $\zeta^{\chi}_{H} = \zeta_{H}$ 
if $\chi$ is the trivial character of $G$.
By Auslander's theorem for the representation $W_{H}$ of $G/H$, 
we have 
$$
\mathrm{Hom}_{S_H}(\widehat{M}^H, \widehat{M}^H) \cong \mathrm{Hom}_{S_H}\big(\bigoplus\limits_i  M^{H}_{\lambda^{H}_i}, \bigoplus\limits_i  M^{H}_{\lambda^{H}_i} \big) \cong  (G/H) \sharp R_{H} 
$$
as algebras. We thus thus re-write the above
homomorphism $\zeta_{H}$ as
$\zeta_{H}: \Lambda \to (G/H) \sharp R_{H}$.

\begin{lemma}\label{Keyy12}
$\zeta^{ \chi}_{H}$ is a surjection of algebras. 
\end{lemma}

The proof is the same as Lemma \ref{Keyy11},
and is left to the interested readers. 
Now recall that $\Lambda_{\mathrm{con}}=\Lambda/\Lambda e\Lambda$. We have
the following.

\begin{lemma}\label{Key2}
For any pair $(H, \chi)$ such that
$\chi(H) \neq \{1\}$, $\zeta^{\chi}_{H}$ induces a well-defined surjection of algebras 
($\bm{the\,\, third\,\, homomorphism}$)
\begin{equation}\label{eq:tildezetachiH}
\tilde{\zeta}^{\chi}_{H} : \Lambda_{\mathrm{con}} \twoheadrightarrow (G/H) \sharp R_{H}.
\end{equation}
\end{lemma}

\begin{proof}
By Lemma \ref{Keyy12}, to prove this lemma, we only need 
to show that $\zeta^{\chi}_{H} (\Lambda e \Lambda) = 0$. 
Since the direct summand 
$M_{\chi_0}$ corresponds to the idempotent $e$, 
to prove $\zeta^{\chi}_{H} (\Lambda e \Lambda) = 0$
it suffices to show that for any characters $\chi_1$, $\chi_2$ of $G$ and 
$\alpha \in \mathrm{Hom}_{S}(M_{\chi_0}, M_{\chi_2})$ and $\beta \in \mathrm{Hom}_{S}(M_{\chi_1}, M_{\chi_0})$, 
$\zeta^{\chi}_{H}(\alpha \beta) = 0$. 

Since $\chi(H) \neq \{ 1 \}$, $\chi|_{H} \neq \chi_{0}|_{H}$. Thus, 
$\chi_{0}$ is not in $\{\lambda^{(H, \chi)}_i \}_{i}$. 
By  
$\zeta^{\chi}_{H} = \tau^H \circ \Pi_{(H, \chi)} = \tau^H \circ \Delta_{(H, \chi)} \circ Pr_{(H, \chi)}$, 
we know that 
$\zeta^{\chi}_{H}(\alpha) = \zeta^{\chi}_{H}(\beta) = 0$ 
since $Pr_{(H, \chi)}(\alpha) = Pr_{(H, \chi)}(\beta) = 0$. 
Thus, by Lemma \ref{Keyy11}, we get that  
$$
\zeta^{\chi}_{H}(\alpha \beta) = \zeta^{\chi}_{H}(\alpha) \zeta^{\chi}_{H}(\beta) = 0.
$$ 
This completes the proof.
\end{proof}

Note that $\zeta^{\chi}_H$ is determined by the set 
$\{\lambda^{(H, \chi)}_i \}_{i}$ of characters and the map 
$\tau^{H}$. It follows that $\zeta^{\chi}_H$ 
is determined by the pair $(H, \chi)$.   
Then, there is bijection between the algebra 
homomorphisms in $\{ \zeta^{\chi}_H\}$ and the
classes in $\tilde{G}$. 
We get that
$\tilde{\zeta}^{\chi}_H = \tilde{\zeta}^{\chi'}_H$
if and only if $(H, \chi) \sim (H, \chi')$ in $\tilde{G}_0$.

\begin{definition}[The fourth homomophism]
Let
$$Z(\tilde{\zeta}^{\chi}_{H}):Z (\Lambda_{\mathrm{con}}) 
\rightarrow Z\big((G/H) \sharp R_{H}\big)$$
be the homomorphism induced from \eqref{eq:tildezetachiH}
by taking the centers of both sides' algebras.
\end{definition}
 
\begin{lemma}\label{Key3}
$Z(\tilde{\zeta}^{\chi}_{H})$ is an algebra homomorphism.
\end{lemma}

\begin{proof}
To prove this lemma, it is enough to prove that 
$$
\tilde{\zeta}^{\chi}_{H}\big( Z (\Lambda_{\mathrm{con}}) \big) 
\subseteq Z\big((G/H) \sharp R_{H}\big). 
$$
Let $h \in (G/H) \sharp R_{H}$ and $f \in  Z (\Lambda_{\mathrm{con}})$. 
Since $\zeta^{\chi}_{H}$ is a surjection, $\tilde{\zeta}^{\chi}_{H}$ is also a
surjection. Then there exists 
$\tilde{h} \in \Lambda_{\mathrm{con}}$  such that $\tilde{\zeta}^{\chi}_{H}(\tilde{h}) = h$. 
Thus, we have 
\begin{align*}
\tilde{\zeta}^{\chi}_{H}(f)h & = \tilde{\zeta}^{\chi}_{H}(f) \tilde{\zeta}^{\chi}_{H}(\tilde{h}) 
 =  \tilde{\zeta}^{\chi}_{H}(f \tilde{h}) 
 = \tilde{\zeta}^{\chi}_{H}(\tilde{h} f) = \tilde{\zeta}^{\chi}_{H}(\tilde{h}) \tilde{\zeta}^{\chi}_{H}(f) =
 h \tilde{\zeta}^{\chi}_{H}(f). 
\end{align*}
It implies that $\tilde{\zeta}^{\chi}_{H}(f) \in 
 Z\big((G/H) \sharp R_{H}\big)$, and it follows that 
\[
\tilde{\zeta}^{\chi}_{H}\big( Z (\Lambda_{\mathrm{con}} ) \big) 
\subseteq Z\big((G/H) \sharp R_{H}\big)
\]
and hence the lemma.
\end{proof}
 
In the above lemma, since $Z\big((G/H) \sharp R_{H}\big) = S_H$, 
we can
write $Z(\tilde{\zeta}^{\chi}_{H})$ as a map
$Z(\tilde{\zeta}^{\chi}_{H}): Z (\Lambda_{\mathrm{con}}) \rightarrow S_H$.

In the following, let 
\begin{align}\label{Lab345}
\zeta_0:=\prod_{(H, \chi) \in \tilde{G}_0} \tilde{\zeta}_{H}^{\chi} :  
\Lambda_{\mathrm{con}}  \rightarrow \prod_{(H, \chi) \in \tilde{G}_0} (G/H) 
\sharp R_{H}
\end{align}
and 
\begin{align}\label{Lab245}
Z(\zeta_{0}):=\prod_{(H, \chi) \in \tilde{G}_0} Z(\tilde{\zeta}_{H}^{\chi}) : 
Z ( \Lambda_{\mathrm{con}} ) \rightarrow \prod_{(H, \chi) \in \tilde{G}_0} S_{H}.
\end{align}
be the induced map on the centers.
Since $S_H \subseteq R_{H}$ as algebras, $S_H$ is a reduced ring for any pair 
$(H, W_{H})$. Hence, 
$\prod_{(H, \chi) \in \tilde{G}_0} S_{H}$ is also a reduced ring. 

\begin{definition}[The fifth homomorphism]
Let $\sqrt{Z( \Lambda_{\mathrm{con}})} $ be the reduced ring 
of $Z( \Lambda_{\mathrm{con}})$. 
Let
\begin{align}\label{Fifth}
\tilde{\zeta}_{0} : 
\sqrt{Z( \Lambda_{\mathrm{con}})} \rightarrow 
\prod_{(H, \chi) \in \tilde{G}_0} S_{H},
\end{align}
be the map induced from \eqref{Lab245}
by taking the reduced rings on both sides.
\end{definition}

Recall that $R^{\widehat{M}} 
= \sqrt{Z( \Lambda_{\mathrm{con}} )}$ (see the beginning of \S\ref{Dec9822415684}).
In the following we write
$\tilde{\zeta}_{0}$ as $\tilde{\zeta}_0: R^{\widehat{M}} 
\rightarrow \prod_{(H, \chi) \in \tilde{G}_0} S_{H}$.
The following is the main result of this subsection.

\begin{theorem}\label{Key4}
$\tilde{\zeta}_0: R^{\widehat{M}} 
\rightarrow \prod\limits_{(H, \chi) \in \tilde{G}_0} S_{H}$ 
is an injection of algebras. 
\end{theorem}

The proof of this theorem 
is quite technical, and we postpone it to Appendix \ref{App:B}.

\subsubsection{The singular locus of $\mathrm{Spec}(S)$ and 
$\mathrm{Spec}(R^{\widehat{M}})$}

In this subsection, we show that the affine scheme
of $R^{\widehat{M}}$ is isomorphic to the reduced singular locus 
$\sqrt{\mathrm{Sing}\big(\mathrm{Spec}(S)\big)}$
(see Theorem \ref{HatM}).
The proof roughly goes as follows:
\begin{enumerate} 
\item[(1)] We first construct a canonical cover of 
$\mathrm{Spec}\big(\sqrt{\varphi(S)/\varphi(S) 
\cap \Lambda  e \Lambda}  \big)$ (see surjection (\ref{Cove11234}));
   
\item[(2)] We then
construct the canonical cover of 
$\mathrm{Spec}(R^{\widehat{M}})$ (see surjection (\ref{Cove112}));

\item[(3)] After that, we give the morphism 
$\mu^{\#}:  \mathrm{Spec}(R^{\widehat{M}}) 
\rightarrow \mathrm{Spec}\big(\varphi(S)/ \sqrt{\varphi(S) 
\cap \Lambda  e \Lambda} \big)$ of schemes, 
and show that this morphism takes the cover (\ref{Cove112}) to 
the cover (\ref{Cove11234}) and is surjective;

\item[(4)]  Finally we show $\mu^{\#}$ is injective.
\end{enumerate} 
We do these step by step. First, let
$\tilde{\zeta}^{\chi}_{H}$ and $Z(\tilde{\zeta}^{\chi}_{H})$ be 
as in Lemmas \ref{Key2} and \ref{Key3}.
We first have the following. 

\begin{lemma}\label{Key7}
For any $(H, \chi), (H, \chi') \in \tilde{G}_0$, 
$Z(\tilde{\zeta}^{\chi}_{H}) = Z(\tilde{\zeta}^{\chi'}_{H})$. In other words, the
algebra homomorphism 
$Z(\tilde{\zeta}^{\chi}_{H})$ only depends on $H$ . 
\end{lemma}

Again, since the proof of this lemma is quite long, we postpone
it to Appendix \ref{App:C}.
Now, for the pair $(H, \chi) \in \tilde{G}_0$, taking the 
reduced rings, $Z(\tilde{\zeta}^{\chi}_{H})$ induces the 
following surjection of algebras: 
\begin{align}\label{DEpp}
R^{\widehat{M}} \cong \sqrt{Z(\Lambda_{\mathrm{con}})} 
\rightarrow Z(G/H \sharp R_H) \cong S_H. 
\end{align} 
By Lemma \ref{Key7}, we know that the above algebra homomorphism 
(\ref{DEpp}) is independent of the character $\chi$. Hence, it only depends  
on pair $(H, W_H)$. 
Thus, we can write the above surjection (\ref{DEpp})
just as $\tilde{\xi}_{H}$. 
Since $Z(\zeta_{0}):=\prod_{(H, \chi) \in \tilde{G}_0} 
Z(\tilde{\zeta}_{H}^{\chi})$, considering 
the associated reduced rings, 
it is obvious that 
\begin{align}\label{Summm}
\tilde{\zeta}_0 = \prod_{(H, \chi) \in \tilde{G}_0} \tilde{\xi}_{H}
\end{align} 
by the definition of $\tilde{\zeta}_0$ (see (\ref{Fifth})) and $\tilde{\xi}_{H}$.

Similarly to Lemma \ref{Key3}, by taking the centers of 
the surjection $\zeta_{H}: \Lambda \twoheadrightarrow G/H 
\sharp R_H$ for the pair $(H, W_H)$, we get the following 
algebra homomorphism
\begin{align}\label{HopLasneed}
Z(\zeta_{H}) : S = Z(\Lambda) \rightarrow Z\big( G/H \sharp R_H \big) = S_H .
\end{align}
Now, for any two pairs $(H, W_H)$ and $(H', W_{H'})$, there is a unique 
pair $(G_{H, H'}, W_{H} \cap W_{H'})$, where $(G_{H, H'},  W_{H} \cap W_{H'})$ 
is the maximal subgroup 
whose invariant subspace is $ W_{H} \cap W_{H'}$.
It is direct to see that $H, H' \subseteq G_{H, H'}$. 
Next, 
replace $G$ by $G/H$, $V$ by $W_H$, $H$ by $G_{H, H'}/H$, $W_H$ by $W_{H} \cap W_{H'}$.  
Analogously to $\tau^{H}_{\chi_i, \chi_j}$ 
for any two characters $\chi_i$ and $\chi_j$ of $G$, 
we obtain 
\begin{align*}
 \tau^{(G_{H, H'}/H)}_{\lambda_i, \lambda_j}:
& \mathrm{Hom}_{S_H}\big((V_{\lambda_i} \otimes R_H)^{(G/H)}, 
 (V_{\lambda_j} \otimes R_H)^{(G/H)}\big) \\
&\twoheadrightarrow 
\mathrm{Hom}_{S_{G_{H, H'}}} \big((V_{\lambda_i} \otimes R_{G_{H, H'}})^{(G_{G_{H, H'}}/H)}, 
 (V_{\lambda_j} \otimes R_{G_{H, H'}})^{(G_{H, H'}/H)}\big)  
\end{align*}
for any two characters $\lambda_i$ 
and 
$\lambda_j$ in $\{\lambda_{i}^{(G_{H, H'}/H)}\}_i$.
Then 
analogously to $\zeta_H$ and $Z(\zeta_H)$ respectively 
(see (\ref{FirstM}) and (\ref{HopLasneed})),
we obtain the corresponding algebra homomorphisms 
\begin{align*}
\zeta_{(G_{H, H'}/H)}:G/H \sharp R_H \twoheadrightarrow 
G/G_{H, H'} \sharp R_{G_{H, H'}}
\quad\mbox{and}\quad
Z\big( \zeta_{(G_{H, H'}/H)} \big):  S_H \rightarrow S_{G_{H, H'}}.
\end{align*}

\begin{lemma}\label{Inters}
For any two pairs 
$(H, W_H)$ and $(H', W_{H'})$, 
we have the following
commutative diagram of algebra homomorphisms  
\begin{equation}\label{LemmaN}
\begin{split}
\xymatrix{ 
R^{\widehat{M}}  \ar[rrr]^{\tilde{\xi}_{H}} 
\ar[rrrd]_{\tilde{\xi}_{G_{H, H'}}} &&&  S_H \ar[d]^{Z(\zeta_{(G_{H, H'}/H)})} \\
 &&& S_{G_{H, H'}}
}
\end{split}
\end{equation}
\end{lemma}

We postpone the proof of this lemma to \ref{App:D}.
Before introducing the main result of this subsection, 
we should note that 
$R^{\widehat{M}}$ is a commutative Noetherian ring. 
In fact, $\widehat{M}$ is a finitely generated $S$-module. 
Hence $\Lambda = \mathrm{End}_{S}(\widehat{M})$ is 
finitely generated $S$-module, and then 
$\Lambda_{\mathrm{con}}$ and $Z(\Lambda_{\mathrm{con}})$ 
are both finitely generated $S$-modules. Thus
$R^{\widehat{M}}$ is a finitely generated $S$-module. 
Since $S$ is a commutative Noetherian ring, $R^{\widehat{M}}$ is also a commutative Noetherian ring. 

\begin{theorem}\label{HatM}
$\mathrm{Spec}(R^{\widehat{M}}) \cong 
\sqrt{\mathrm{Sing}\big(\mathrm{Spec}(S)\big)}$ as subschemes of $\mathrm{Spec}(S)$. 	
\end{theorem}

\begin{proof}The proof consists of four steps.
\indexspace

\noindent{\it $\bm{Step}$ 1:
We construct a canonical cover of 
$\mathrm{Spec}\big(\varphi(S)
\big/ \sqrt{\big(\varphi(S) \cap \Lambda  e \Lambda}  \big)$.}

To this end, we consider the following composition of two algebra  
injections: 
$$
\varphi(S) \big/ \sqrt{\varphi(S) \cap (\Lambda e \Lambda)}  \hookrightarrow 
R^{\widehat{M}} \xrightarrow{\tilde{\zeta}_0} 
\prod_{(H, \chi) \in \tilde{G}_0} S_{H}
$$
(for the injection of the second homomorphism see Theorem \ref{Key4}).
We show the first homomorphism is also an injection.
In fact, $\varphi$ gives the injection:
$$
\varphi(S) \big/ \big(\varphi(S) \cap (\Lambda e \Lambda)\big)
\hookrightarrow \Lambda \big/ \Lambda e \Lambda \cong \Lambda_{\mathrm{con}}
$$
and the image of this injection is contained in 
$Z(\Lambda_{\mathrm{con}})$. Then 
by considering the associated reduced rings,
we get an injection 
\begin{equation}\label{map:toRhatM}
\varphi(S) \big/ \sqrt{\varphi(S) \cap (\Lambda e \Lambda)} \hookrightarrow R^{\widehat{M}}.
\end{equation}
Fixing a pair $(H, \chi) \in \tilde{G}_0$ and composing the  above injection
$\varphi(S)/ \sqrt{\varphi(S) \cap (\Lambda e \Lambda)} 
\hookrightarrow \prod_{(H, \chi) \in \tilde{G}_0} S_{H}$ with the canonical projection 
$\prod_{(H, \chi) \in \tilde{G}_0} S_{H} \twoheadrightarrow S_{H}$, 
we get an algebra homomorphism 
\begin{align}\label{VaH}
\varphi(S)/ \sqrt{\varphi(S) \cap (\Lambda e \Lambda)} 
\hookrightarrow R^{\widehat{M}} \xrightarrow{\tilde{\zeta}_0} 
\prod_{(H, \chi) \in \tilde{G}_0} S_{H} \twoheadrightarrow S_H,
\end{align}
which we denote by $\varphi_H$. 
Next, when composing $\varphi_H$ with the quotient map
$$S \cong \varphi(S) \rightarrow \varphi(S)/ \sqrt{\varphi(S) \cap (\Lambda e \Lambda)},$$
we get another algebra homomorphism  
$S \rightarrow S_H$. 
It is easy to check that this algebra homomorphism 
is exactly the quotient  
$$S = k[V]^G \twoheadrightarrow k[W_H]^{G/H} = S_H$$
associated to the canonical injection 
$W_H \big/ (G/H) \hookrightarrow (W_H \oplus W'_H) / G \cong V/G$ of schemes. 
Since $S \rightarrow S_H$ is 
a surjection of algebras, so is $\varphi_H$. 
Thus, $\varphi_H$ gives a closed injection of schemes, say 
$$ 
\varphi_H^{\#}: \mathrm{Spec}(S_H) \hookrightarrow 
\mathrm{Spec}\big(\varphi(S) / \sqrt{ \varphi(S) \cap (\Lambda e \Lambda)}\big).
$$ 

In the meantime, $\prod_{(H, \chi)} \varphi_H$
is exactly the injection 
$
\varphi(S)/ \sqrt{\varphi(S) \cap (\Lambda e \Lambda)}
\hookrightarrow R^{\widehat{M}} \xrightarrow{\tilde{\zeta}_0} 
\prod_{(H, \chi) \in \tilde{G}_0} S_{H}
$. 
Then the induced homomorphism of schemes 
$$
\bigsqcup\limits_{(H, \chi)} \varphi_H^{\#} : \bigsqcup\limits_{(H, \chi)} \mathrm{Spec}(S_H) \twoheadrightarrow \mathrm{Spec}\big( \varphi(S)/ \sqrt{\varphi(S) \cap (\Lambda e \Lambda)}  \big)
$$
is a cover of  
$\mathrm{Spec}\big(\varphi(S) \big/ \sqrt {\varphi(S) \cap (\Lambda e \Lambda)} \big)$. 
Moreover, by replacing $(H, \chi)$ by $H$ for any $(H, \chi) \in \tilde{G}_0$,
the above cover reduces to a new cover of 
$\mathrm{Spec}\big(\varphi(S) \big/ \sqrt{\varphi(S) \cap (\Lambda e \Lambda)} \big)$:
\begin{align}\label{Cove11234}
\bigsqcup\limits_{H} \varphi_H^{\#} : \bigsqcup\limits_{H} 
\mathrm{Spec}(S_H) \twoheadrightarrow \mathrm{Spec}
\big(\varphi(S) \big/ \sqrt{\varphi(S) \cap (\Lambda e \Lambda)} \big).
\end{align}

\noindent{\it $\bm{Step}$ 2:
We construct a canonical cover of $\mathrm{Spec}\big(R^{\widehat{M}}\big)$.}

Note that the composition 
$$
R^{\widehat{M}} \xrightarrow{\tilde{\zeta}_{0}}\prod_{(H, \chi) \in 
\tilde{G}_0} S_{H} \twoheadrightarrow S_H
$$
is exactly $\tilde{\xi}_H$ by definition (see (\ref{DEpp})). By 
$\eqref{VaH}$, 
$\tilde{\xi}_H$ is surjective since $\varphi_H$ is surjective. 
Then  
for any pair $(H, W_H)$,  
$$
(\tilde{\xi}_H)^{\#} :\mathrm{Spec}(S_H) \hookrightarrow \mathrm{Spec}(R^{\widehat{M}})
$$
is an injection of schemes, where $(\tilde{\xi}_{H})^{\#}$ is the homomorphism of 
schemes induced by $\tilde{\xi}^\chi_H$. 
Since $\tilde{\zeta}_0 = \prod_{(H, \chi)} \tilde{\xi}_H$ 
is an injection of algebras 
(see Theorem \ref{Key4} and (\ref{Summm})), 
we also get a cover of $\mathrm{Spec}(R^{\widehat{M}})$
$$
\bigsqcup\limits_{(H, \chi)} (\tilde{\xi}_H)^{\#} : 
\bigsqcup\limits_{(H, \chi)} \mathrm{Spec}(S_H) 
\twoheadrightarrow \mathrm{Spec}(R^{\widehat{M}}). 
$$
By definition of $\tilde{\xi}_H$ (see (\ref{DEpp})), 
it is independent of the character $\chi$. 
Hence, if we only consider the subcover
$\bigsqcup\limits_{H} (\tilde{\xi}_H)^{\#}$ instead of 
$\bigsqcup\limits_{(H, \chi)} (\tilde{\xi}_H)^{\#}$, since 
$\mathrm{Im}\big( \bigsqcup\limits_{H} (\tilde{\xi}_H)^{\#} \big) 
= \mathrm{Im}\big(\bigsqcup\limits_{(H, \chi)} (\tilde{\xi}_H)^{\#} \big),$
then we get a new 
cover of $\mathrm{Spec}(R^{\widehat{M}})$
\begin{align}\label{Cove112}
\bigsqcup\limits_{H} (\tilde{\xi}_H)^{\#} : 
\bigsqcup\limits_{H} \mathrm{Spec}(S_H) \twoheadrightarrow \mathrm{Spec}(R^{\widehat{M}}). 
\end{align}

\noindent{\it $\bm{Step}$ 3: We give a morphism 
$\mu^{\#}:  \mathrm{Spec}(R^{\widehat{M}}) \rightarrow \mathrm{Spec}
\big(\varphi(S)\big/ \sqrt{\varphi(S) \cap \Lambda  e \Lambda} \big)$ of schemes, 
and show that this morphism takes the cover (\ref{Cove112}) to 
the cover(\ref{Cove11234}) and is surjective.} 

In the following, denote by $\mu$ the following composition of natural algebra homomorphisms
$$
\varphi(S) \big/ \sqrt{\varphi(S) \cap (\Lambda e \Lambda)}  
\hookrightarrow  Z(\Lambda_{\mathrm{con}}) \twoheadrightarrow R^{\widehat{M}},
$$
the first induced by $\varphi$ and the second by quotient 
by the nilradical in $Z(\Lambda_{\mathrm{con}})$. 
Let $$\mu^{\#}: \mathrm{Spec}(R^{\widehat{M}}) 
\rightarrow \mathrm{Spec}\big( \varphi(S) \big/ 
\sqrt{\varphi(S) \cap (\Lambda e \Lambda)} \big)$$ be the homomorphism of schemes 
given by $\mu$. 
Now, by the definition of 
$\tilde{\xi}_H$ (see (\ref{DEpp})) 
and $\eqref{VaH}$, we have 
$\tilde{\xi}_H \circ \mu = \varphi_H$ 
for any pair $(H, W_H)$. Then we have the following commutative diagram:
$$
\xymatrixcolsep{2.5pc}
\xymatrixrowsep{2.5pc}
\xymatrix{ \bigsqcup\limits_{H} \mathrm{Spec}(S_H) 
\ar[rr]^{\bigsqcup\limits_{H} \varphi_H^{\#}} 
\ar[rd]_{\bigsqcup\limits_{H} (\tilde{\xi}_H)^{\#}} &&  
\mathrm{Spec}\big(\varphi(S) \big/ \sqrt{\varphi(S) \cap (\Lambda e \Lambda)} \big)  \\
& \mathrm{Spec}(R^{\widehat{M}}) \ar[ru]_{\mu^{\#}}.
}
$$
Since 
$\bigsqcup\limits_{H} \varphi_H^{\#}$ is a cover of 
$\mathrm{Spec}\big(\varphi(S) \big/ \sqrt{\varphi(S) \cap (\Lambda e \Lambda)} \big)$,
it implies that $\mu^{\#}$ is a surjection of schemes.

\indexspace
\noindent{\it $\bm{Step}$ 4: We show that $\mu^{\#}$ is injective.}

Since for any pair $(H, W_H)$, $\mu^{\#} \circ  (\tilde{\xi}_H)^{\#} = \varphi_H^{\#}$ is an injection, 
$\mu^{\#}|_{\mathrm{Im}\big((\tilde{\xi}_H)^{\#}\big)}$ is an 
injection when restricting $\mu^{\#}$ on $\mathrm{Im}\big((\tilde{\xi}_H)^{\#}\big)$. 
Note that $\bigsqcup\limits_{H} (\tilde{\xi}_H)^{\#}$ is a cover of $\mathrm{Spec}(R^{\widehat{M}})$.  
Hence, to prove that $\mu^{\#}$ is an injection of schemes, 
it suffices to show that for any two different pairs $(H, W_H)$ 
and $(H', W_{H'})$, and any $p \in \mathrm{Im}\big((\tilde{\xi}_H)^{\#}\big)$ and 
$p' \in \mathrm{Im}\big((\tilde{\xi}_{H'})^{\#}\big)$, 
$\mu^{\#}(p) = \mu^{\#}(p')$ if and only if 
$p = p'$ in $\mathrm{Spec}(R^{\widehat{M}})$. 

In fact, by Theorem \ref{Singlocus1} and Lemma \ref{KeyPr}, we have 
\begin{align*}
\bigcup\limits_{(H, W_H)} \mathrm{Spec}(S_{H}) 
\cong \sqrt{\mathrm{Sing}\big(\mathrm{Spec}(S)\big)} 
\cong \mathrm{Spec}\big(\varphi(S) / 
\sqrt{\varphi(S) \cap (\Lambda e \Lambda)} \big). 
\end{align*}
The composition of the above two isomorphisms 
coincides with 
$$
  \bigcup\limits_{H} 
\mathrm{Im}\big((\varphi_H)^{\#}\big) \cong 
\mathrm{Spec}\big(\varphi(S) / \sqrt{\varphi(S) \cap (\Lambda e \Lambda)} \big)
$$
induced by $\prod_{H} \varphi_H^{\#}$. 
Here, $\mathrm{Spec}(S_{H}) \cong \mathrm{Im}\big((\varphi_H)^{\#}\big)$ is in scheme 
$\sqrt{\mathrm{Sing}\big(\mathrm{Spec}(S)\big)}$ for any pair $(H, W_H)$. 
Then for any two different pairs $(H, W_H)$ 
and $(H', W_{H'})$, the intersection of $\mathrm{Im}\big(\varphi_H^{\#}\big)$ and $\mathrm{Im}\big(\varphi_{H'}^{\#}\big)$ in 
$\sqrt{\mathrm{Sing}\big(\mathrm{Spec}(S)\big)}$ is isomorphic to 
the intersection of $\mathrm{Spec}(S_{H})$ and 
$\mathrm{Spec}(S_{H'})$ in $\sqrt{\mathrm{Sing}\big(\mathrm{Spec}(S)\big)}$. 

Recall that $S_{H} = S / I_{H}$ and $S_H$ is associated to
the subscheme $W_H / (G/H)$ of $V/G$. 
In $\sqrt{\mathrm{Sing}\big(\mathrm{Spec}(S)\big)}$, 
the intersection of $\mathrm{Spec}(S_{H})$ and $\mathrm{Spec}(S_{H'})$ 
is given by the following canonical pullback commutative diagram: 
\begin{equation}\label{Dia77}
\begin{split}
\xymatrix{ 
\mathrm{Spec}(S/(I_H, I_{H'})) \ar[rr]  \ar[d] &&  \mathrm{Spec}(S/I_{H})  \ar[d] \\
\mathrm{Spec}(S/I_{H'}) \ar[rr] &&
\mathrm{Spec}(S). 
}
\end{split}
\end{equation}
Thus, 
the intersection of $\mathrm{Im}\big(\varphi_H^{\#}\big)$ and 
$\mathrm{Im}\big(\varphi_{H'}^{\#}\big)$ is isomorphic to 
\begin{align*}
\mathrm{Spec}\big(S / (I_H, I_H')\big) & \cong W_H / (G/H) \bigcap W_{H'} / (G/H') \\ 
& \cong (W_H \cap W_{H'}) / (G/G_{H,H'}) \\ 
& \cong \mathrm{Spec}(S_{G_{H, H'}}) 
\end{align*} 
in $\sqrt{\mathrm{Sing}\big(\mathrm{Spec}(S)\big)}$,
which is exactly $\mathrm{Im}\big(\varphi_{G_{H, H'}}^{\#}\big)$. 
Let $\tilde{p} \in \mathrm{Spec}(S_H)$ and $\tilde{p}' \in \mathrm{Spec}(S_{H'})$ such that 
$(\tilde{\xi}_H)^{\#}(\tilde{p}) = p$ and $(\tilde{\xi}_{H'})^{\#}(\tilde{p}') = p'$. 
If $\mu^{\#}(p) = \mu^{\#}(p')$, then 
$\varphi_H^{\#}(\tilde{p}) = \varphi_{H'}^{\#}(\tilde{p}')$. It implies that 
\begin{align}\label{INssse}
\mu^{\#}(p) = \mu^{\#}(p') \in \mathrm{Im}\big(\varphi_H^{\#}\big) 
\cap \mathrm{Im}\big(\varphi_{H'}^{\#}\big) = 
\mathrm{Im}\big(\varphi_{G_{H, H'}}^{\#}\big) \cong \mathrm{Spec}(S_{G_{H, H'}}).
\end{align}

In the meantime, by the commutative diagram 
\eqref{LemmaN} in Lemma \ref{Inters}, we know that  
$$
Z(\zeta_{(G_{H, H'}/H)}) \circ \tilde{\xi}_{H} 
= \tilde{\xi}_{G_{H, H'}}. 
$$
It implies that $Z(\zeta_{(G_{H, H'}/H)})$ is a surjection of algebras, since 
$\tilde{\xi}_{G_{H, H'}}$ is surjective. This induces an injection 
$$
\big( Z(\zeta_{(G_{H, H'}/H)}) \big)^{\#} : \mathrm{Im}
\big((\tilde{\xi}_{G_{H, H'}})^{\#}\big) \hookrightarrow 
\mathrm{Im}\big((\tilde{\xi}_{H})^{\#}\big) 
$$ of schemes.  
Meanwhile, there is a natural injection 
$$
 \mathrm{Im}\big((\varphi_{G_{H, H'}})^{\#}\big) \cong 
 W_{G_{H, H'}} \big/ (G/G_{H, H'}) \hookrightarrow W_{H} \big/ (G/H) \cong
\mathrm{Im}\big((\varphi_{H})^{\#}\big)
$$
of schemes. Furthermore, 
since $\varphi_H^{\#}$ is an injection, when $\mu^{\#}$ restricts on 
$\mathrm{Im}\big((\tilde{\xi}_{H})^{\#}\big)$, $\mu^{\#}$ is also an injection by 
$\mu^{\#} \circ (\tilde{\xi}_H)^{\#} = \varphi_H^{\#}$.  
Then $\mu^{\#}$ gives an isomorphism between $\mathrm{Im}\big((\tilde{\xi}_{H})^{\#}\big)$ 
and $\mathrm{Im}\big((\varphi_H)^{\#}\big)$. 
In the same way, $\mu^{\#}$ also 
gives an isomorphism between $\mathrm{Im}\big((\tilde{\xi}_{G_{H, H'}})^{\#}\big)$ 
and $\mathrm{Im}\big((\varphi_{G_{H, H'}})^{\#}\big)$.
By combining the above morphisms, we then get the following pullback commutative 
diagram:
\begin{equation}\label{Dia11}
\begin{split}
\xymatrix{ 
\mathrm{Im}\big((\tilde{\xi}_{G_{H, H'}})^{\#}\big) \ar[rr]^{\sim}_{ \mu^{\#}} \ar@{^{(}->}[d]_{\big( Z(\zeta_{(G_{H, H'}/H)}) \big)^{\#}} &&  
\mathrm{Im}\big(\varphi_{G_{H, H'}}^{\#}\big) \ar@{^{(}->}[d] \\
 \mathrm{Im}\big((\tilde{\xi}_H)^{\#}\big) \ar[rr]^{\sim}_{ \mu^{\#}}  && 
\mathrm{Im}\big(\varphi_{H}^{\#}\big). 
}
\end{split}
\end{equation}

Since $p  \in \mathrm{Im}\big((\tilde{\xi}_H)^{\#}\big)$ and 
$\mu^{\#}(p) \in \mathrm{Im}\big(\varphi_{G_{H, H'}}^{\#}\big)$ by (\ref{INssse}),  
we have $p \in \mathrm{Im}\big((\tilde{\xi}_{G_{H, H'}})^{\#}\big)$ by (\ref{Dia11}). 
In the same way, we also have 
$p' \in \mathrm{Im}\big((\tilde{\xi}_{G_{H, H'}})^{\#}\big)$.
Meanwhile, we know that  
$\mu^{\#}$ is also an injection when it restricts on 
$\mathrm{Im}\big((\tilde{\xi}_{G_{H, H'}})^{\#}\big)$. 
Thus, considering this injection $\mu^{\#}: \mathrm{Im}
\big((\tilde{\xi}_{G_{H, H'}})^{\#}\big) \hookrightarrow 
\mathrm{Spec}(R^{\widehat{M}})$, we obtain that
$p = p'$ since $\mu^{\#}(p) = \mu^{\#}(p')$. 

Summarizing the above Steps 1-4, we get
the theorem.
\end{proof}

Combining Theorem \ref{HatM} and Corollary \ref{cor:reducedloci}
we get the following theorem. 

\begin{theorem}\label{Theo22}
There is an algebra isomorphism $\varphi(S)/ 
\sqrt{\varphi(S) \cap (\Lambda e \Lambda)} \cong R^{\widehat{M}}$ 
induced from the map $\varphi: S \rightarrow \Lambda$. 
\end{theorem}

\section{Proof of the main result}\label{ST}

In this section, we prove
Theorem \ref{main4}. 

For any Cohen-Macaulay $S$-module $M$, recall that 
$\Lambda^{M}_{\mathrm{con}} \cong 
\mathrm{End}_{\underline{\mathrm{CM}}(S)}(M)$ (see Proposition \ref{CML}).
Denote by $Z(\Lambda^{M}_{\mathrm{con}})$
the center of $\Lambda^{M}_{\mathrm{con}}$,
and denote by $R^{M}$ the reduced ring of 
$Z(\Lambda^{M}_{\mathrm{con}})$. 
Note that since $M$ is a finitely generated $S$-module,  
$\mathrm{End}_{S}(M)$ is also a finitely generated 
$S$-module. It implies that $\Lambda^{M}_{\mathrm{con}}$ and 
$R^{M}$ are both finitely generated $S$-modules. 
Therefore, $R^{M}$ is a commutative Noetherian ring since 
$S$ is a commutative Noetherian ring.

Our strategy to the proof of Theorem \ref{main4} is as follows.
\begin{enumerate}
\item[(1)] First, we show that $\mathcal{R}^S := \big\{R^{M}, 
\varphi^{M'}_{M} \big\}_{M, M' \in \underline{\mathrm{CM}}(S)}$,
depending only on 
the triangulated category $\underline{\mathrm{CM}}(S)$,
forms an inverse system;

\item[(2)] Then we prove that the inverse limit
is exactly $\varphi(S) / \sqrt{\varphi(S) \cap \Lambda e \Lambda}$. 

\item[(3)] Finally, since 
by Theorem \ref{Singlocus1}
the reduced ring of the singular locus of $S$ is
$\varphi(S) / \sqrt{\varphi(S) \cap \Lambda e \Lambda}$, it only depends on  
 $\underline{\mathrm{CM}}(S)$, and the theorem follows.
\end{enumerate}
Let us accomplish them step by step.

\subsection{An inverse system}
We first claim that there exists the following inverse system:  
$$
\mathcal{R}^S := \big\{R^{M}, \varphi^{M'}_{M} \big\}_{M, M' \in \underline{\mathrm{CM}}(S)},
$$
where $\varphi^{M'}_{M}$ is given as follows. 

Let $M'$ be a Cohen-Macaulay $S$-module such that $M$ is a direct summand of $M'$, i.e., 
$M' \cong M \oplus N$ for some 
$S$-module $N$. 
Then there is a natural $S$-module injection
$M \hookrightarrow M'$ and a natural $S$-module surjection $M' \twoheadrightarrow M$. 
It follows that there is an $S$-module surjection  
$$
\mathrm{Hom}_{S}(M', M') \twoheadrightarrow \mathrm{Hom}_{S}(M, M)
$$
given by composing with these two $S$-module homomorphisms. 
This $S$-module homomorphism naturally induces the $S$-module surjection   
$
\Lambda^{M'}_{\mathrm{con}} \twoheadrightarrow \Lambda^{M}_{\mathrm{con}}, 
$
which we denote by $\phi^{M'}_{M}$. 

In the meantime, by taking the centers of the algebras, we have 
\begin{align*}
Z(\Lambda^{M'}_{\mathrm{con}}) 
& \cong Z\big(\mathrm{End}_{\underline{\mathrm{CM}}(S)}(M')\big) 
\cong Z\big(\mathrm{End}_{\underline{\mathrm{CM}}(S)}(M \oplus N)\big) \\
& \cong \Big\{(f_1, f_2) \in Z\big(\mathrm{End}_{\underline{\mathrm{CM}}(S)}(M)\big) 
\oplus Z\big(\mathrm{End}_{\underline{\mathrm{CM}}(S)}(N)\big) 
\Big| \forall \, g \in \mathrm{Hom}_{\underline{\mathrm{CM}}(S)}(M, N) \, \& \\
&\quad\quad h \in \mathrm{Hom}_{\underline{\mathrm{CM}}(S)}(N, M),
\,\,  g \circ f_1 = f_2 \circ g, \, \& \, f_1 \circ h = h \circ f_2   \Big\}.
\end{align*}
Thus we obtain a natural $S$-algebra homomorphism 
$$\widetilde{\phi}^{M'}_{M}:Z(\Lambda^{M'}_{\mathrm{con}}) 
\rightarrow Z\big(\mathrm{End}_{\underline{\mathrm{CM}}(S)}(M)\big) 
\cong Z(\Lambda^{M}_{\mathrm{con}}).
$$ 
Moreover, it is easy to check that this $S$-algebra 
homomorphism gives the following commutative diagram of 
$S$-module homomorphisms:
$$
\xymatrixcolsep{4pc}
\xymatrix{ Z(\Lambda^{M'}_{\mathrm{con}}) \ar@{^{(}->}[r] 
\ar[d]_{\widetilde{\phi}^{M'}_{M}} & \Lambda^{M'}_{\mathrm{con}} 
\ar[d]^{\phi^{M'}_{M}} \\
Z(\Lambda^{M}_{\mathrm{con}}) \ar@{^{(}->}[r] 
& \Lambda^{M}_{\mathrm{con}}. 
} 
$$
From the $S$-algebra homomorphism 
$\widetilde{\phi}^{M'}_{M}$, 
by taking the corresponding reduced rings, we get the  
$S$-algebra homomorphism $R^{M'} \twoheadrightarrow R^{M}$, 
which we denote by $\varphi^{M'}_{M}$.

\begin{lemma}
$
\mathcal{R}^S:= \big\{R^{M}, \varphi^{M'}_{M} \big\}_{M, M' \in \underline{\mathrm{CM}}(S)}
$ is an inverse system.  
\end{lemma}

\begin{proof}
We only need to show that $\{\varphi_{M}^{M'}\}$ 
defined above satisfies the cocycle condition.
In fact, to show $\varphi^{M''}_{M}
= \varphi^{M'}_{M}\circ \varphi^{M''}_{M'}$, 
it is enough to show
$$
\widetilde{\phi}^{M''}_{M} 
= \widetilde{\phi}^{M'}_{M} \circ \widetilde{\phi}^{M''}_{M'},
$$ 
for any $M, M', M''$ such that $M$ is a direct summand of $M'$ 
and $M'$ is a direct summand of $M''$.

In fact, notice that
we have the following two commutative diagrams of 
$S$-module homomorphisms:
$$
\xymatrixcolsep{2pc}
\xymatrix{ Z(\Lambda^{M''}_{\mathrm{con}}) \ar@{^{(}->}[rr] 
\ar[d]^{\widetilde{\phi}^{M''}_{M'}} 
&& \Lambda^{M''}_{\mathrm{con}} \ar[d]^{\phi^{M''}_{M'}} 
&&\Lambda^{M''}_{\mathrm{con}} \ar[d]^{\phi^{M''}_{M'}} 
\ar@/^1.4cm/[dd]^{\phi_{M}^{M''}}\\
Z(\Lambda^{M'}_{\mathrm{con}}) \ar@{^{(}->}[rr] \ar[d]^{\widetilde{\phi}^{M'}_{M}} 
&& \Lambda^{M'}_{\mathrm{con}} \ar[d]^{\phi^{M'}_{M}} 
&\mbox{and}& \Lambda^{M'}_{\mathrm{con}}\ar[d]^{\phi^{M'}_{M}}\\
Z(\Lambda^{M}_{\mathrm{con}}) \ar@{^{(}->}[rr] && \Lambda^{M}_{\mathrm{con}}&& 
\Lambda^{M}_{\mathrm{con}}
 } 
$$
Since in the left diagram the horizontal morphisms
are all injective, 
combining it with the right commutative diagram we then get
$\widetilde{\phi}^{M''}_{M} = \widetilde{\phi}^{M'}_{M} \circ \widetilde{\phi}^{M''}_{M'}$. 
\end{proof}

Now set
$\widehat{M} = \bigoplus\limits_{W \in \hat{G}}(W \otimes R)^G$ as before. 
In the next two subsections, we
are going to show that $R^{\widehat M} \cong \varphi(S) / \sqrt{\varphi(S) \cap \Lambda e \Lambda}$ is 
the inverse limit of $\mathcal R^S$.

\subsection{The universal algebra homomorphism}

In this subsection, we show that for
any Cohen-Macaulay $S$-module $M$, 
there is a (universal) algebra homomorphism 
$$\varphi_{M}:\varphi(S) / \sqrt{\varphi(S) \cap \Lambda e \Lambda } \rightarrow R^{M}.$$
We shall also prove a technical result (see Proposition \ref{Cl22}), 
which says that for any Cohen-Macaulay $S$-module $N$, 
$\varphi_{\widehat{M}}$ can be extended to a new 
algebra isomorphism $\varphi_{\widehat{M}(l, l') \oplus N}$ for some integers $l' > 0 > l$.

To this end, notice first that we have a natural algebra homomorphism  
\begin{align}\label{Modell}
\phi_M:
S \rightarrow \mathrm{End}_{\underline{\mathrm{CM}}(S)}(M),
\end{align}
given by the $S$-module structure of $M$. It is obvious that 
the image of \eqref{Modell}
 is in the center of 
$\mathrm{End}_{\underline{\mathrm{CM}}(S)}(M)$. 
By slightly abusing the notations, we still denote this morphism by 
$\phi_M: S \rightarrow Z\big(\mathrm{End}_{\underline{\mathrm{CM}}(S)}(M)\big)$.
Next, from the proof of Theorem \ref{Singlocus1}, the support of the $S$-module 
$\mathrm{End}_{\underline{\mathrm{CM}}(S)}(M)$ is contained in the singular locus of 
$\mathrm{Spec}(S)$.  
When identifying $S$ with $\varphi(S)$, 
we have 
$$
\mathrm{End}_{\underline{\mathrm{CM}}(S)}(M) 
\circ \big(\varphi(S) \cap \Lambda e \Lambda \big)  =\{ 0\},
$$
where $\circ$ means the right action given by the $S$-module structure on $\mathrm{End}_{\underline{\mathrm{CM}}(S)}(M)$;
this is because by Theorem \ref{Singlocus1},
$\mathrm{Spec}\big( \varphi(S) \big/ \big(\varphi(S) 
\cap \Lambda e \Lambda \big) \big)$ is the singular locus of 
$\mathrm{Spec}(S)$. 
Hence, 
$$
Z\big(\mathrm{End}_{\underline{\mathrm{CM}}(S)}(M)\big) \circ
\big(\varphi(S) \cap \Lambda e \Lambda \big)  =\{0\}.
$$ 
Thus, $\phi_M$ induces the following well-defined algebra homomorphism 
$$
\widetilde{\phi}_M: \varphi(S) \big/ \big(\varphi(S) \cap \Lambda e \Lambda \big) 
\rightarrow Z\big(\mathrm{End}_{\underline{\mathrm{CM}}(S)}(M)\big).
$$
Taking the reduced rings of the above algebras, 
we get the following algebra homomorphism
$$\varphi_M:
\varphi(S) \big/ \sqrt{\varphi(S) \cap \Lambda e \Lambda} \rightarrow R^{M}.
$$
Note that in the special case $M=\widehat{M}$, 
$\varphi_{\widehat{M}}$ is the algebra isomorphism given by Theorem \ref{Theo22}. 
$\varphi_M$ thus defined satisfies the following.

\begin{lemma}\label{Lefttt}
For any two
Cohen-Macaulay $S$-modules $M$ and $M'$ such that $M$ is a direct summand of $M'$,
$$\varphi_M = \varphi^{M'}_{M} \circ \varphi_{M'}.$$
\end{lemma}

\begin{proof}
Recalling that
$$
\phi^{M'}_{M}:
\mathrm{End}_{\underline{\mathrm{CM}}(S)}(M') \twoheadrightarrow 
\mathrm{End}_{\underline{\mathrm{CM}}(S)}(M)
$$ 
is the canonical projection,
we have $\phi_M = \phi^{M'}_{M} \circ \phi_{M'}$ by definition (see (\ref{Modell})). 
By taking the centers of algebras, and since the images of $\phi_M$ and $\phi_{M'}$ are both in the
centers respectively, 
we have  
$$\phi_M = \widetilde{\phi}^{M'}_{M} \circ \phi_{M'}.$$  
Since both centers are
supported on the singular locus $\mathrm{Spec}\big( \varphi(S) 
/ \sqrt{\varphi(S) \cap (\Lambda e \Lambda)}\big)$ 
of $ \mathrm{Spec}(S)$ (see the proof of Theorem \ref{Singlocus1}), 
$\phi_M$ and $\phi_{M'}$ induce the following algebra  homomorphisms 
$$
\widetilde{\phi}_M:
\varphi(S) \big/ \big(\varphi(S) \cap \Lambda e 
\Lambda \big) \rightarrow Z\big( \mathrm{End}_{\underline{\mathrm{CM}}(S)}(M) \big)
$$
and 
$$\widetilde{\phi}_{M'}:
\varphi(S) \big/ \big(\varphi(S) \cap \Lambda e \Lambda \big) 
\rightarrow Z\big( \mathrm{End}_{\underline{\mathrm{CM}}(S)}(M') \big).
$$ 
Moreover, the natural algebra homomorphism 
$$
S \rightarrow \varphi(S) \big/ \big(\varphi(S) \cap \Lambda e \Lambda \big),
$$
induced by $\varphi$, is a surjection. In summary, 
we have the following commutative diagram:  
$$\xymatrix{
&S\ar[d]\ar@/_0.4cm/[ldd]_{\phi_{M'}}\ar@/^0.4cm/[rdd]^{\phi_{M}}&\\
& \varphi(S) / \big(\varphi(S) \cap \Lambda e \Lambda \big)
\ar[ld]_{\tilde\phi_{M'}}\ar[rd]^{\tilde\phi_{M}}&\\
Z\big( \mathrm{End}_{\underline{\mathrm{CM}}(S)}(M') \big)\ar[rr]^{\widetilde{\phi}^{M'}_{M}}
&& Z\big( \mathrm{End}_{\underline{\mathrm{CM}}(S)}(M)
}$$
Then, by $\phi_M = \widetilde{\phi}^{M'}_{M}\circ \phi_{M'}$,   
we have $\widetilde{\phi}_M = 
\widetilde{\phi}^{M'}_{M} \circ \widetilde{\phi}_{M'}$,
since $S \rightarrow \varphi(S) / \big(\varphi(S) \cap 
\Lambda e \Lambda \big)$ is surjective.
Taking the reduced rings for these rings, 
we get $\varphi_M = \varphi^{M'}_{M} \circ \varphi_{M'}$.
\end{proof}

Next, we introduce a 
result (Proposition \ref{Cl22}), which will be to be used in next section. 
The background is that, for a Cohen-Macauley module $N$,
we expect to have an isomorphism $R^{\widehat M\oplus N}\cong R^{\widehat M}$;
however, at present we are not
able to show this isomorphism. The idea to solve this issue is that, 
we may find two integers $l, l'$ and consider
$\widehat M(l, l'):=\bigoplus_{i=l}^{l'}\widehat M[i]$, and then 
we have: (a) $R^{\widehat M}\cong R^{\widehat M(l, l')}$
and (b)
$R^{\widehat M(l, l')\oplus N}\cong R^{\widehat M}$.
On the other hand, 
considering $\widehat M(l, l')$ instead of $\widehat M$ is natural in the
sense that $\widehat M$ is a generator of
$\underline{\mathrm{CM}}(S)$, where the degree shifts appear naturally.

\begin{proposition}\label{Cl22}
For any 
Cohen-Macaulay $S$-module $N$, there exist $l, l'\in\mathbb Z$ with $l'>0>l$,
and a commutative diagram 
$$
\xymatrixcolsep{2.5pc}
\xymatrixrowsep{2.5pc}
\xymatrix{ \varphi(S) \big/ \sqrt{\varphi(S) \cap \Lambda e \Lambda} 
\ar[rr]^-{\varphi_{\widehat{M}}} \ar[rd]_-{\varphi_{\widehat{M}(l, l') \oplus N}} 
&&  R^{\widehat{M}} \\ 
& R^{\widehat{M}(l, l') \oplus N} \ar[ru]_-{\varphi_{\widehat{M}}^{\widehat{M}(l, l') \oplus N}} &
}
$$
of algebra isomorphisms.  
\end{proposition}

The proof is postponed to Appendix \ref{App:E}.

\subsection{The limit of the inverse system}

Now we are ready to show the main result of this section.

\begin{theorem}\label{kkeyl}
The inverse limit of $\mathcal{R}^S$ is 
$$\big(\varphi(S) \big/ \sqrt{\varphi(S) \cap \Lambda e \Lambda }, \{\varphi_{M}\}_{M \in \underline{\mathrm{CM}}(S)} \big).$$ 
\end{theorem}

\begin{proof}
To prove this theorem, we only need
to show the universal property of 
$$\big(\varphi(S) \big/ \sqrt{\varphi(S) \cap \Lambda e \Lambda }, \{\varphi_{M}\}_{M \in \underline{\mathrm{CM}}(S)} \big).$$
That is, for any 
$\big(T, \{\psi_{M}\}_{(M \in \underline{\mathrm{CM}}(S))} \big)$, where 
$\psi_M : T \rightarrow R^M$,  
satisfying that if $M$ is a direct summand of $M'$ as $S$-modules and
$\psi_M = \varphi^{M'}_{M} \circ \psi_{M'}$,
there is an unique algebra homomorphism $\eta_{T}: 
T \rightarrow \varphi(S) \big/ \sqrt{\varphi(S) \cap \Lambda e \Lambda }$ such that 
\begin{align}\label{FFFinl888}
\psi_M = \varphi_M \circ \eta_T.
\end{align} 

To this end, define $\eta_T$ to be  
$(\varphi_{\widehat{M}})^{-1} \circ \psi_{\widehat{M}}$. 
Let $N$ be an arbitrary Cohen-Macaulay $S$-module. 
Note that by Proposition \ref{Cl22}, $\varphi_{\widehat{M}}$, 
$\varphi_{\widehat{M}(l, l') \oplus N}$ and 
$\varphi_{\widehat{M}}^{\widehat{M}(l, l') \oplus N}$ are all invertible.   
We then have 
\begin{align*}
\varphi_N \circ \eta_T & = \varphi_N \circ \big( (\varphi_{\widehat{M}})^{-1} 
\circ \psi_{\widehat{M}} \big) \\
& = \big( \varphi_{N}^{\widehat{M}(l, l') \oplus N} \circ \varphi_{\widehat{M}(l, l') \oplus N}  
\big)\circ \big( (\varphi_{\widehat{M}})^{-1} \circ \psi_{\widehat{M}} \big) \\
& = \big( \varphi_{N}^{\widehat{M}(l, l') \oplus N} \circ \varphi_{\widehat{M}(l, l') \oplus N}  
\big)\circ \big( (\varphi_{\widehat{M}(l, l') \oplus N})^{-1} \circ  
( \varphi_{\widehat{M}}^{\widehat{M}(l, l') \oplus N})^{-1} \circ \psi_{\widehat{M}} \big) \\ 
& = \varphi_{N}^{\widehat{M}(l, l') \oplus N} \circ ( \varphi_{\widehat{M}}^{\widehat{M}(l, l') 
\oplus N} )^{-1} \circ \psi_{\widehat{M}}. 
\end{align*}
Moreover, by the assumption of $\big(T, \{\psi_{M}\}_{M \in \underline{\mathrm{CM}}(S)} \big)$, 
$$
\psi_{N} = \varphi_{N}^{\widehat{M}(l, l') \oplus N} \circ \psi_{\widehat{M}(l, l') \oplus N}.
$$
Hence, to prove 
$\psi_N = \varphi_N \circ \eta_T$, it suffices to show that 
$$
\psi_{\widehat{M}(l, l') \oplus N} = ( \varphi_{\widehat{M}}^{\widehat{M}(l, l') 
\oplus N} )^{-1} \circ \psi_{\widehat{M}}.
$$ 
In fact, this is true by composing 
$( \varphi_{\widehat{M}}^{\widehat{M}(l, l') \oplus N} )^{-1}$ on both sides of the following equality
$$
\varphi_{\widehat{M}}^{\widehat{M}(l, l')\oplus N} \circ \psi_{\widehat{M}(l, l') \oplus N} 
= \psi_{\widehat{M}} . 
$$
Now letting $N=M$, we get  
$\psi_M = \varphi_M \circ \eta_T$.

Finally, we show the uniqueness of $\eta_T$. 
Assume that there is another algebra homomorphism $\eta'_T$ satisfying that 
$\psi_M = \varphi_M \circ \eta'_T$ for any Cohen-Macaulay $S$-module $M$. Choosing 
the Cohen-Macaulay $S$-module $\widehat{M}$, we have   
$$
\varphi_{\widehat{M}} \circ \eta'_{T} = \psi_{\widehat{M}} = 
\varphi_{\widehat{M}} \circ \eta_T.
$$
Next, composing the inverse $\varphi_{\widehat{M}}^{-1}$ 
on the both sides of the above equality, we get that 
$\eta'_T = \eta_T$. 
This shows the uniqueness of $\eta_T$. 
QED. 
\end{proof}

Combining Theorem \ref{kkeyl} with Corollary \ref{cor:reducedloci} we get the following. 

\begin{corollary}\label{Theo33}
The coordinate ring of $\sqrt{\mathrm{Sing}\big(\mathrm{Spec}(S)\big)}$, which is
$\varphi(S) / \sqrt{\varphi(S) \cap \Lambda e \Lambda }$, is isomorphic to the ring of
the inverse limit of $\mathcal {R}^S$. 	
\end{corollary}

\subsection{Proof of Theorem \ref{main4}}

Now we are ready to prove our main theorem.

\begin{proof}[Proof of Theorem \ref{main4}] 
Let $S$ be the Gorenstein commutative Noetherian ring 
as in Example \ref{keyexample}. 
By Corollary \ref{Theo33}, we get the coordinate ring of 
$\sqrt{\mathrm{Sing}\big(\mathrm{Spec}(S)\big)}$ from  
the inverse limit of $\mathcal {R}^S$.  
Note that 
the inverse system $\mathcal{R}^S$ only depends on the triangulated category 
$D_{sg}(S) \cong \underline{\mathrm{CM}}(S)$, so does
the coordinate ring of $\sqrt{\mathrm{Sing}\big(\mathrm{Spec}(S)\big)}$.

If $\Upsilon: D_{sg}(S_1) \rightarrow D_{sg}(S_2)$
is a triangle equivalence, then the inverse systems 
$\mathcal R^{S_1}$ and $\mathcal R^{S_2}$
are isomorphic, and so are their limits.
Thus by 
Corollary \ref{Theo33} we obtain that 	
$$
\sqrt{\mathrm{Sing}\big(\mathrm{Spec}(S_1)\big)} \cong \sqrt{\mathrm{Sing}
\big(\mathrm{Spec}(S_2)\big)}
$$	
as schemes. 	
\end{proof}

\section{Examples}\label{EX}

In this section, we give two simple examples of our main theorem.

\begin{example}\label{Ex1}
Let $R = k[x_{1}, x_{2}, x_{3}]$	 and $G \subseteq \mathrm{SL}(3, k)$ be generated by elements 
$$	
f_1 :=\mathrm{Diag}(1, -1, -1) \,\, ,  f_2 := \mathrm{Diag}(-1, 1, -1),
$$	
where $f_1$ sends $x_1$ to $x_1$, $x_2$ to $-x_2$ and 
$x_3$ to $-x_3$. $f_2$ sends $x_1$ to $-x_1$, $x_2$ to $x_2$ and $x_3$ to $-x_3$. 
Then $S= R^{G} = k[x_1^2, x_2^2, x_3^2, x_1 x_2x_3] \cong k[A, B, C, D] \big/ (ABC -D^{2})$, 
where $A = x_{1}^{2}$, $B = x_{2}^{2}$, $C = x_{3}^{2}$ and $D = x_{1}x_{2}x_3$. 	

Consider the Jacobi ring 
$k[A, B, C, D] \big/ (AB, BC, CA, D)$ of the hypersurface $\mathrm{Spec}(S)$.    
Since $ABC - D^2 \subseteq (AB, BC, CA, D)$, we have that 
$$
\mathrm{Sing}\big(\mathrm{Spec}(S) \big) \cong \mathrm{Spec}
\big( k[A, B, C, D] \big/ (AB, BC, CA, D) \big).
$$

We now apply the method in the previous sections to compute 
$\mathrm{Sing}\big(\mathrm{Spec}(S) \big)$. To this end, consider the commutative ring 
$$
Z\big(\mathrm{End}_{\underline{\mathrm{CM}}(S)}( \bigoplus_{i} M_i)\big) 
\cong Z\big(\Lambda / \Lambda e \Lambda\big). 
$$
It is given as follows: observe that
$\widehat{M} \cong S \oplus (V_1 \otimes R)^{G} 
\oplus (V_2 \otimes R)^{G} \oplus (V_3 \otimes R)^{G}$ and
$\Lambda = \mathrm{End}_{S}(\widehat{M})$, where
\begin{itemize}
\item[$-$] $V_1$ is the irreducible representation of
$G$ given by $f_{1}v_1 = v_1$ and $f_{2}v_1 = -v_1$ for any $v_1 \in V_1$, 

\item[$-$] $V_2$ is the irreducible representation of
$G$ given by $f_{1}v_2 = -v_2$ and $f_{2}v_2 = v_2$ for any $v_2 \in V_2$, and  

\item[$-$] $V_3$ is the irreducible representation of
$G$ given by $f_{1}v_3 = -v_3$ and $f_{2}v_3 = -v_3$ for any $v_3 \in V_3$.  	
\end{itemize}
Let $M_i = (V_i \otimes R)^G$ and $M_0 = S$, which are Cohen-Macaulay $S$-modules. 
It is easy to check that
\begin{itemize}
\item[$-$] $\mathrm{End}_{\underline{\mathrm{CM}}(S)}(M_1) = S/I_1$, 
where $I_1 := (x_1^{2}, x^{2}_{2}x^{2}_3, x_1 x_2 x_3) \cong (A, BC, D)$;

\item[$-$] $\mathrm{End}_{\underline{\mathrm{CM}}(S)}(M_2) = S/I_2$, 
where $I_2 := (x_2^{2}, x^{2}_{1}x^{2}_3, x_1 x_2 x_3) \cong (B, AC, D)$;

\item[$-$] $\mathrm{End}_{\underline{\mathrm{CM}}(S)}(M_3) = S/I_3$, where
$I_3 := (x_3^{2}, x^{2}_{1}x^{2}_2, x_1 x_2 x_3) \cong (C, AB, D)$.
\end{itemize}
For $i=1,2,3$,
$\mathrm{Spec}(S/I_i)$ is a subscheme of the singular locus 
$\mathrm{Spec}\big( S\big/(AB, BC, AC, D) \big)$.
The intersection of $\mathrm{Spec}(S/I_i)$ and $\mathrm{Spec}(S/I_j)$ is 
$\mathrm{Spec}\big(S/(I_{i}, I_j) \big)$ for any $i, j \in \{1, 2, 3\}$, 
and the intersection of the three schemes is 
the origin in $\mathrm{Spec}(S)$. 

By Auslander's theorem,
$\Lambda\cong G \sharp R$.
The associated quiver
$Q_{\Lambda} $ is
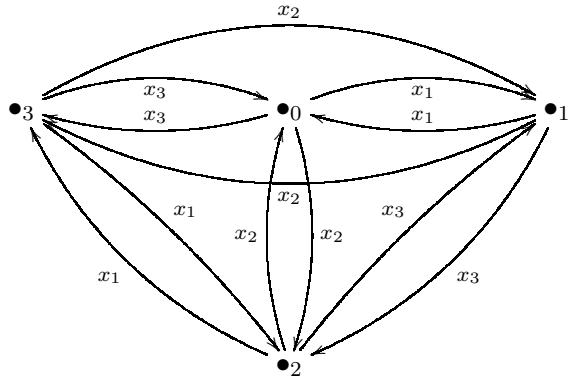
\begin{displaymath}
\xymatrix{
\bullet_{3} \ar@/^0.4cm/[rrr]_{x_{3}} \ar@/^1.1cm/[rrrrrr]^{x_{2}}
\ar@/^0.20cm/[drdrdr]^{x_{1}}
&&& \bullet_{0} \ar@/^0.4cm/[rrr]_{x_{1}}
\ar@/^0.3cm/[lll]_{x_{3}}
\ar@/^0.3cm/[ddd]^{x_{2}}
&&& \bullet_{1} \ar@/^0.3cm/[lll]_{x_{1}}
\ar@/^1cm/[llllll]^{x_{2}}
\ar@/^0.50cm/[dldldl]^{x_{3}} \\ \\ \\ 
&&& \bullet_{2} \ar@/^0.3cm/[uuu]^{x_{2}}
\ar@/^0.50cm/[ululul]^{x_{1}}
\ar@/^0.20cm/[ururur]^{x_{3}}
}
\end{displaymath}
In the above diagram, vertex $i$ 
corresponds to the $S$-module $M_i$.
The quiver of 
$\Lambda_{\mathrm{con}}  \cong \Lambda / \Lambda e \Lambda$ is given as follows: 
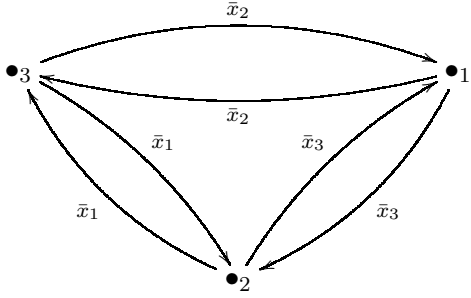
\begin{displaymath}
\xymatrixcolsep{1.5pc}
\xymatrixrowsep{1.5pc}
\xymatrix{
\bullet_{3} \ar@/^0.6cm/[rrrrrr]^{\bar{x}_{2}}
\ar@/^0.30cm/[drdrdr]^{\bar{x}_{1}}
&&&&&& \bullet_{1} 
\ar@/^0.4cm/[llllll]^{\bar{x}_{2}}
\ar@/^0.40cm/[dldldl]^{\bar{x}_{3}} \\ \\ \\ 
&&& \bullet_{2} 
\ar@/^0.40cm/[ululul]^{\bar{x}_{1}}
\ar@/^0.30cm/[ururur]^{\bar{x}_{3}}
}
\end{displaymath}
where $\bar{x}_2\bar{x}_3 = \bar{x}_3 \bar{x}_2 = 0$, 
$\bar{x}_2 \bar{x}_1 = \bar{x}_1 \bar{x}_2 = 0$ 
and $\bar{x}_1 \bar{x}_3 = \bar{x}_3 \bar{x}_1 = 0$.
Let $J_{1, 2}$ be the annihilator of $e_1 \bar{x}_3 e_2 
\in \Lambda / \Lambda e \Lambda$. It is easy to check that   
$J_{1, 2} = (\bar{x}_1^2, \bar{x}_2^2, \bar{x}_1\bar{x}_2\bar{x}_3) 
\cong (A, B, D) \cong (I_1, I_2)$. 
Analogously, 
\begin{align*}
&J_{2, 1} \cong J_{1,2}\cong (A, B, D) \cong (I_1, I_2),\\
&J_{2, 3} \cong J_{3, 2} \cong (B, C, D) \cong (I_2, I_3),\\
&J_{3, 1} \cong J_{3, 1} \cong (A, C, D) \cong (I_3, I_1).
\end{align*}
Thus, we see that 
\begin{align*}
Z(\Lambda / \Lambda e \Lambda) 
& = \Big\{(\bar{h}_1, \bar{h}_2, \bar{h}_3)\in S/I_1 \oplus S/I_2 
\oplus S/I_3 \,  \big| h_j -h_{j'} \in J_{j, j'}, \forall j, j' \in \{1, 2, 3 \}\Big\} \\
& \cong S/(AB, BC, AC, D), 
\end{align*}
where $h_j \in S$. 	
The above second equality is given by the following. 
Since $I_1 \cap I_2 \cap I_3 = (AB, BC, AC, D)$, there is an injection 
$$
S\big/(AB, BC, AC, D) \hookrightarrow  S/I_1 \oplus S/I_2 \oplus S/I_3,  
$$
given by the three natural projections 
\begin{align*}
S\big/(AB, BC, AC, D) \twoheadrightarrow S\big/(A, BC, D) =S/I_1,\\
S\big/(AB, BC, AC, D)
 \twoheadrightarrow S\big/(B, AC, D) =S/I_2,\\
S\big/(AB, BC, AC, D) \twoheadrightarrow S\big/(C, AB, D) =S/I_3.
\end{align*}
It is clear that the image of this injection is exactly the subring:  
$$
\big\{(\bar{h}_1, \bar{h}_2, \bar{h}_3) \in S/I_1 \oplus 
S/I_2 \oplus S/I_3 \,  \big| h_j -h_{j'} \in J_{j, j'}, \forall j, j' \in \{1, 2, 3 \}\big\},
$$
since the intersection of $\mathrm{Spec}(S/I_j)$ and $\mathrm{Spec}(S/I_{j'})$ is exactly
$\mathrm{Spec}(S/J_{j, j'})$. 

Thus we have 	
$$	
\mathrm{Sing} \big(\mathrm{Spec}(S)\big) \cong \mathrm{Spec}\big( S\big/(AB, BC, AC, D) \big) 
\cong \mathrm{Spec}\big(Z(\Lambda / \Lambda e \Lambda)\big).  
$$		
Hence, the reduced scheme  
of $\mathrm{Spec}( Z\big(\Lambda /\Lambda e\Lambda) \big)$ is exactly 
$\sqrt{\mathrm{Sing}\big(\mathrm{Spec}(S)\big)}$.
\end{example}

\begin{example}\label{Ex2}
Let $V$ be a three-dimensional vector space, $R = k[V] = k[x_{1}, x_{2}, x_{3}]$ and 
$G \subseteq \mathrm{SL}(V)$ be generated by elements 
$$	
g := \mathrm{Diag}(\sigma, \sigma, \sigma^2),  	
$$	
where $\sigma$ is a $4$-th primitive root of unit. 
Here, $g$ sends $x_1$ to $\sigma x_1$, $x_2$ to $\sigma x_2$ and $x_3$ to $ \sigma^2 x_3$. 
Then
$$
S = R^G = k[x^{4}_{1}, x^{4}_{2}, x^{2}_{3}, x_{1}x^{3}_{2}, x^{2}_{1}x^{2}_{2}, 
x^{3}_{1}x_{2}, x^{2}_1 x_{3}, x^{2}_2 x_{3}, x_{1}x_{2}x_{3}]. 
$$ 		

Since $G$ does not contain any pseudo-reflection,
by the Chevalley-Shephard-Todd theorem (see Appendix \ref{App:B}), 
the reduced scheme of the singular locus in $\mathrm{Spec}(S)$
consists of the 
images under the quotient map of the points whose isotropy subgroup  
is nontrivial. 	
Let $W_3 \subseteq V$ be the one-dimensional subspace with the associated 
surjection of algebras $\pi_3 : k[V] \twoheadrightarrow k[W_3] \cong k[x_3]$ given by $\pi_3(x_3) = x_3$ and 
$\pi_3(x_1) = \pi_3(x_2)= 0$.  
It is direct to see that the set of points whose isotropy subgroups  
are nontrivial is exactly $W_3$. 
Now, the isotropy subgroup of $W_3$ is 
$$
G_3 := \big< \mathrm{Diag}(-1, -1,  1)\big> \subseteq G.
$$
Hence, 
$$
\sqrt{\mathrm{Sing}\big(\mathrm{Spec}(S)\big)} 
\cong W_3 / (G/G_3) = \mathrm{Spec}(k[x
_3^2]).
$$

Consider the skew group algebra $\Lambda$, the quiver $Q_{\Lambda}$ is given as follows:
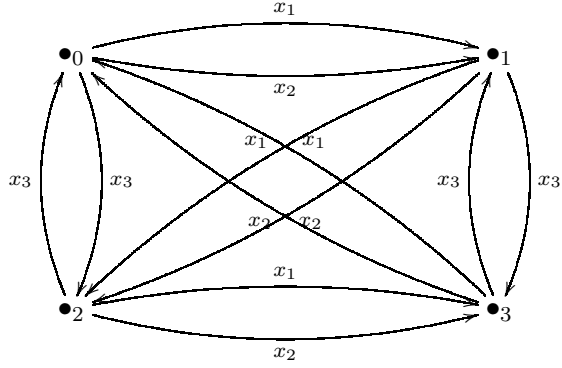
\begin{displaymath}
\xymatrix{
\bullet_{0} \ar@/^0.4cm/[rrrrr]^{x_{1}} 
\ar@/_0.3cm/[rrrrr]_{{{x_{2}}}} 
\ar@/^0.4cm/[ddd]^{{{x_{3}}}}
 &&&&& \bullet_{1} 
\ar@/^0.4cm/[ddd]^{{{x_{3}}}} 
\ar@/^0.4cm/[lllllddd]^-{{{x_{2}}}}  
\ar@/_0.4cm/[lllllddd]_-{x_{1}} \\ \\ \\ 
\bullet_{2} \ar@/^0.4cm/[uuu]^{{{x_{3}}}} 
\ar@/^0.3cm/[rrrrr]^{x_{1}} 
\ar@/_0.4cm/[rrrrr]_{{{x_{2}}}}  
&&&&& \bullet_{3}  
\ar@/^0.4cm/[llllluuu]^-{{{x_{2}}}} 
\ar@/_0.4cm/[llllluuu]_-{x_{1}} \ar@/^0.4cm/[uuu]^{{{x_{3}}}} } 
\end{displaymath}
The quiver $Q_{\mathrm{con}}$ associated to 
$\Lambda_{\mathrm{con}} \cong \Lambda / \Lambda e \Lambda$ is the following: 
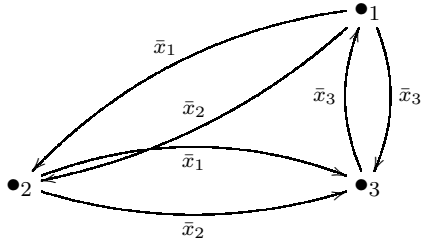
\begin{displaymath}
\xymatrix{
&&&& \bullet_{1} 
\ar@/^0.3cm/[dd]^{{{\bar{x}_{3}}}} 
\ar@/^0.4cm/[lllldd]_-{{{\bar{x}_{2}}}}  
\ar@/_0.5cm/[lllldd]_-{\bar{x}_{1}} \\ \\ 
\bullet_{2} \ar@/^0.5cm/[rrrr]_{\bar{x}_{1}} 
\ar@/_0.4cm/[rrrr]_{{{\bar{x}_{2}}}}  &&&&
\bullet_{3} \ar@/^0.3cm/[uu]^{{{\bar{x}_{3}}}}  }  
\end{displaymath}
Let $e_r$ be the indecomposable idempotent corresponding 
to vertex $\bullet_r$ in the quiver $Q_{\Lambda}$. 
Then from the above quiver, we obtain
\begin{align}\label{FNVJH12975}
e_{r} (1 \otimes \bar{x}_i \bar{x}_j \bar{x}_3) e_{r'} = 0 
\quad\mbox{and}\quad
e_{r} (1 \otimes \bar{x}_i \bar{x}^2_3) e_{r'} = 0 
\end{align}
in $\Lambda / \Lambda e \Lambda$, for any $r, r' \in \{1, 2, 3 \}$ and any $i, j \in 
1, 2$. In fact, one of the following paths
 $$
 x_i x_j x_{3}, \,\,  x_i x_3 x_{j} \,\, x_{3} x_i {x_j}, \,\, x_i x_3 x_3 \,\,  x_3 x_i x_3 \,\, x_3 x_3 x_i
 $$ 
 starting from vertex $\bullet_r$ must pass through vertex $\bullet_0$ in $Q_\Lambda$. By 
(\ref{FNVJH12975}), 
it is then straightforward to check that 
$$
\mathrm{End}_{\underline{\mathrm{CM}}(S)}(M_1) = k[\bar{x}^2_3],\,\, 
\mathrm{End}_{\underline{\mathrm{CM}}(S)}(M_2) = k,\,\, 
\mathrm{End}_{\underline{\mathrm{CM}}(S)}(M_1) = k[\bar{x}^2_3],
$$
where $M_i$ is the direct summand of $R$ corresponding to the idempotent $e_i$. 
Thus, by (\ref{FNVJH12975}) again, it is easy to verify that 
$$
Z(\Lambda /\Lambda e\Lambda) = Z\big(\mathrm{End}_{\underline{\mathrm{CM}}(S)}
(\bigoplus\limits_{i=1}^3 M_i ) \big) = k[\bar{x}^2_3] 
$$
as algebras.   
Hence, the reduced scheme  
of $\mathrm{Spec}( Z\big(\Lambda /\Lambda e\Lambda) \big)$ is 
$\sqrt{\mathrm{Sing}\big(\mathrm{Spec}(S)\big)}$. 
\end{example}

\appendix

\section{Proof of Lemma \ref{Keyy11}}\label{App:A}

\begin{proof}[Proof of Lemma \ref{Keyy11}]
First, from the constructions of $\tau^H$ and $Pr_{H}$, 
we know that they are both surjective maps, and thus so is
$\zeta_{H}$. 
Thus to prove this lemma, it is enough to prove that $\zeta_{H}$ is an algebra homomorphism,
which is further sufficient to show that 
for any $\alpha, \beta \in \Lambda$, 
$$
\zeta_{H}(\alpha \beta) = \zeta_{H}(\alpha) \zeta_{H}(\beta).
$$

Without loss of generality, let $\alpha \in \mathrm{Hom}_{S}(M_{\chi_2}, M_{\chi_3})$ and $\beta 
\in \mathrm{Hom}_{S}(M_{\chi_1}, M_{\chi_2})$ for 
characters $\chi_1, \chi_2, \chi_3$ of $G$. 
Moreover, suppose that $\alpha$ and $\beta$ are monomials in $(V_{\chi_3 \chi_2^{-1}} \otimes R)^G \subseteq R$
and $(V_{\chi_2 \chi_1^{-1}} \otimes R)^G \subseteq R$ respectively. 
We have the following two possibilities about $\chi_1,\chi_2$ and $\chi_3$: 
\begin{enumerate}
\item[(1)]  $\chi_1$, $\chi_2$ and $\chi_3$ are all in $\{\lambda^{H}_i \}_{i}$;
\item[(2)]  at least one of $\{ \chi_{1}, \chi_{2}, \chi_3\}$ is 
not in $\{\lambda^{H}_i \}_{i}$. 
\end{enumerate}

For case (1), by the
definition of $Pr_H$ (see (\ref{PrHHH})), the images of 
$\alpha$, $\beta$ and $\alpha \beta$
under the projection
$$Pr_H:
\mathrm{Hom}_{S}(\widehat{M}, \widehat{M}) \twoheadrightarrow 
\mathrm{Hom}_{S}\big(\bigoplus_i  M_{\lambda^{H}_i}, \bigoplus_i  M_{\lambda^{H}_i} \big) 
$$
are also $\alpha, \beta, \alpha \beta \in 
 \mathrm{Hom}_{S}\big(\bigoplus_i  M_{\lambda^{H}_i}, \bigoplus_i  M_{\lambda^{H}_i} \big)$ 
as monomials. 
Then we have 
$Pr_{H}(\alpha) = \alpha$, $Pr_{H}(\beta) = \beta  
$ and $
Pr_{H}(\alpha \beta)= \alpha \beta$.

In the meantime, we know that $\chi_1$, $\chi_2$ and $\chi_3$ are all in $\{\lambda^{H}_i \}_{i}$. 
Then by (\ref{Lab11}),
we have  
$\alpha \in \big(  V_{\chi_3 \chi_2^{-1}}\otimes k[W_H ]\otimes k[W'_H]^H\big)^{G/H}$ 
and $\beta \in \big(  V_{\chi_2 \chi_1^{-1}}\otimes k[W_H ]\otimes k[W'_H]^H\big)^{G/H}$ as monomials. 
Thus, we have 
$\alpha = \alpha^{H} \alpha'$ and $\beta = \beta^{H} \beta'$
as monomials, where 
$\alpha^H \in  V_{\chi_3 \chi_2^{-1}}\otimes k[W_H ] \cong k[W_H ]$, $\beta^H \in  V_{\chi_2} \chi_1^{-1} \otimes k[W_H ] \cong k[W_H ]$, and 
$\alpha', \beta' \in k[W'_H]^H$. 
Moreover we have  
 $\alpha \beta = \alpha^H \beta^H \alpha' \beta'$
as monomials. By the definition 
of $\tau_{\chi_2, \chi_3}^H$, we have
$$ \tau_{\chi_2, \chi_3}^H (\alpha)=\left\{
\begin{array}{cl}
\alpha, &\alpha' \in k \subseteq k[W'_H]^H,\\
0,&\alpha' \in (k[W'_H]^H)^{+},
\end{array}
\right.
$$
where $(k[W'_H]^H)^{+}$ is the augmentation ideal
of $k[W'_H]^H$.
Similarly, we have
$$ \tau_{\chi_1, \chi_2}^H (\beta)=\left\{
\begin{array}{cl}
\beta, & \beta' \in k \subseteq k[W'_H]^H, \\
0, & \beta' \in (k[W'_H]^H)^{+},
\end{array}
\right.
$$
and 
$$ \tau_{\chi_1, \chi_3}^H (\alpha \beta)=\left\{
\begin{array}{cl}
\alpha \beta, & \alpha' \beta' \in k \subseteq k[W'_H]^H,   \\
0, &\alpha' \beta' \in (k[W'_H]^H)^{+}.
\end{array}
\right.
$$
If $\alpha', \beta' \in k$, then 
$\tau_{\chi_2, \chi_3}^H (\alpha) 
\tau_{\chi_1, \chi_3}^H (\beta) = \alpha \beta = \tau_{\chi_1, \chi_3}^H (\alpha \beta)$. 
If one of $\alpha', \beta'$ is not in $k$, $\alpha' \beta'$ 
is not in $k$, either. 
Then 
$\tau_{\chi_2, \chi_3}^H (\alpha) \tau_{\chi_1, \chi_3}^H (\beta) = 
0 = \tau_{\chi_1, \chi_3}^H (\alpha \beta)$. 
It follows that 
$$
\tau_{\chi_2, \chi_3}^H (\alpha) \tau_{\chi_1, \chi_3}^H (\beta) = \tau_{\chi_1, \chi_3}^H (\alpha \beta).
$$ 
It suggests that 
$\tau^H (\alpha) \tau^H (\beta) = \tau^H (\alpha \beta)$ 
since $ \tau^H = \bigoplus_{i, j} \tau_{\lambda^{H}_i, \lambda^{H}_j}^H$. 
By composing with the projection $Pr_{H}$, 
we have  
\begin{align*}
\zeta_{H}(\alpha \beta) & = \tau^{H} \circ Pr_{H}(\alpha \beta) = \tau^{H}(\alpha \beta) \\
& = \tau^{H}(\alpha) \tau^{H}(\beta)= \big(\tau^{H} 
\circ Pr_{H}(\alpha)\big) \big( \tau^{H} \circ Pr_{H}(\beta)\big)\\
& = \zeta_{H}(\alpha) \zeta_{H}(\beta)
\end{align*} 
in this case. 
 
For case (2), if 
one of $\chi_1, \chi_3$ is not in $\{\lambda^{H}_i \}_{i}$, 
then one of the images of 
$\alpha$ and $\beta$ is trivial under 
the projection  
$Pr_{H}$. In this case, 
the image of $\alpha \beta$ is also trivial under $Pr_H$. Hence, 
when one of $\chi_1, \chi_3$ is not in $\{\lambda^{H}_i \}_{i}$, we have 
$Pr_{H}(\alpha \beta) = 0 = Pr_{H}(\alpha) Pr_{H}(\beta)$. 
By composing with $\tau^H$, we obtain that 
$$
\zeta_{H}(\alpha \beta) = 0 = \zeta_{H}(\alpha) \zeta_{H}(\beta). 
$$

If $\chi_2$ is not in $\{\lambda^{H}_i \}_{i}$ and 
$\chi_1, \chi_3$ are both in $\{\lambda^{H}_i \}_{i}$, 
we first have 
$\zeta_{H}(\alpha) = \zeta_{H}(\beta) 
=0$ since the images of 
$\alpha$ and $\beta$ is trivial under 
projection $Pr_H$.  
In the meantime, by Lemma \ref{Equi} we have
$$ 
\alpha \in \mathrm{Hom}_{S}(M_{\chi_2}, M_{\chi_3}) \cong 
\big(V_{\chi_2^{-1} \chi_3} \otimes k[W_H] \otimes k[W'_H] \big)^G.
$$
Then we can write $\alpha=\alpha^{*} \alpha^{\vee}$ as monomials, 
where $\alpha^{*} \in V_{\chi_2^{-1} \chi_3} \otimes k[W_H] \cong k[W_H]$ 
and $\alpha^{\vee} \in k[W'_H]$. 
At this time, neither of $\chi_{1}^{-1}\chi_2, \chi_{2}^{-1}\chi_3$ 
is in $\{\lambda^{H}_i \}_{i}$. 
Then $V_{\chi_{2}^{-1}\chi_3}$ is not invariant under action 
of $H$. 
Meanwhile, $k[W_H]$ is invariant under action 
of $H$. 
Hence, $\alpha^{*}$ is not invariant under action of $H$. It implies that $\alpha^{*}$ is not invariant under action of $G$. 
But  
$\alpha^{*} \alpha^{\vee}$ is invariant under action of $G$. Thus, 
$\alpha^{\vee} \in (k[W'_H]^H)^{+}$ since $k$ is invariant under action of $G$. 

On the other hand, let $\beta = \beta^{*} \beta^{\vee}$ 
as monomials, 
where $\beta^{*} \in V_{\chi_1^{-1} \chi_2} \otimes k[W_H]$ and $\beta^{\vee} 
\in k[W'_H]$. In the same way, we get that $\beta^{\vee} \in (k[W'_H]^H)^{+}$
as a monomial. 

Thus, 
we have $\alpha \beta = \alpha^{*} \beta^{*} \alpha^{\vee} \beta^{\vee}$ as 
monomials such that $\alpha^{\vee} \beta^{\vee} \in (k[W'_H]^H)^{+}$.  
Then $\tau_{\chi_1, \chi_3}^H (\alpha \beta) = 0$ 
by the definition of $\tau_{\chi_1, \chi_3}^H$. 
It implies that $\tau^{H}(\alpha \beta) = 0$. 
Finally, from $\zeta_{H} = \tau^{H} \circ Pr_H$, we have
$\zeta_{H}(\alpha \beta) = 0$. It follows that if $\chi_2$ is not in 
$\{\lambda^{H}_i \}_{i}$ and 
$\chi_1, \chi_3$ are both in $\{\lambda^{H}_i \}_{i}$,
$$
\zeta_{H}(\alpha \beta) = 0 = \zeta_{H}(\alpha) \zeta_{H} (\beta).
$$
Thus we have $\zeta_{H}(\alpha \beta) = 0 = \zeta_{H}(\alpha) \zeta_{H} (\beta)$ 
in this case. 

In summary, in both cases
$
\zeta_{H}(\alpha \beta) = \zeta_{H}(\alpha) \zeta_{H} (\beta)
$.
This completes the proof.
\end{proof}

\section{Proof of Theorem \ref{Key4}}\label{App:B}

Before proving this theorem, let us first recall the 
Chevalley-Shephard-Todd theorem (see,
for example, \cite{Bo}):  
{\it
Let $G$ be a finite group acting faithfully on an affine space $V$. 
For $x \in V$, let $G_x$ be the stabilizer of $x$. Let $\bar{x}$ be the image of $x$ in the quotient scheme $V / G$. 
Then $\bar{x}$ is 
a non-singular point if and only if $G_x$ is generated by pseudo-reflections. 
}

Here, a pseudo-reflection is a linear isomorphism 
$s: V \xrightarrow{\sim} V$ that 
leaves a hyperplane   
$W \subseteq V$ pointwise invariant.
When $G$ is a finite subgroup of $\mathrm{SL}(V)$, the unique pseudo-reflection is the identity. 
Thus by the Chevalley-Shephard-Todd theorem, 
we know that $\bar{x}$ is not a singular point if and only if  $G_x$ is trivial in our case, 
and obtain the following.

\begin{lemma}\label{KeyPr} 
With the notations from \S\ref{subsect:singularlocus}, 
$$
\sqrt{\mathrm{Sing}\big( \mathrm{Spec}(S)\big)} 
\cong \bigcup\limits_{(H, W_H)} W_{H}\big/(G/H) \cong  
\bigcup\limits_{(H, W_H)} \mathrm{Spec}(S_{H})
$$
as subschemes of $\mathrm{Spec}(S)$. 
\end{lemma}

Now, let $I_H \subseteq S$ be the ideal 
associated to 
the component $W_{H}\big/(G/H)$ of singular locus in 
$\mathrm{Spec}(S)$. By Lemma \ref{KeyPr},
$$
\sqrt{\mathrm{Sing}\big( \mathrm{Spec}(S)\big)} = 
\bigcup\limits_{(H, W_H)} \mathrm{Spec}(S_{H}) = \mathrm{Spec}\big( S \big/ \bigcap_{(H, W_H)} I_{H}  \big)
$$
as subschemes of $\mathrm{Spec}(S)$. 
We have the following lemma. 

\begin{lemma}\label{KeyPrLe1}
Let $\mathfrak{S}$ be a set of formal variable in $R$ such that $\mathfrak{S} \not\subset R_H$ for any  
$(H, W_H)$. Then the image of 
$\bigodot\limits_{x_i \in \mathfrak{S}} x_i^{\#(G)} \in S$ is trivial in the coordinate ring of 
$\sqrt{\mathrm{Sing}\big( \mathrm{Spec}(S)\big)}$, where  
$\bigodot$ is the product of $R$. 
\end{lemma}

\begin{proof}
To prove this lemma, it is sufficient to show that 
$\bigodot\limits_{x_i \in \mathfrak{S}} x_i^{\#(G)} \in I_{H}$ for any pairs $(H, W_H)$
as above. 
Note that $\bigodot\limits_{x_i \in \mathfrak{S}} x_i^{\#(G)} 
\in S$ since $x_i^{\#(G)}$ is invariant under the action of $G$. 

By the definition of $S_H \cong k[W_H]^G$, 
$I_{H}$ is the kernel of the natural 
quotient map $$S \cong k[W_H \oplus W'_H]^G \twoheadrightarrow k[W_H]^G \cong k[W_H]^{G/H} = S_H.$$
It implies that 
$$
I_{H} = S \cap k[W'_H]. 
$$
Thus to prove $\bigodot\limits_{x_i \in \mathfrak{S}} x_i^{\#(G)} \in I_{H}$, 
it is enough to show that $\mathfrak{S} \cap k[W'_H] \neq \emptyset$. 
Here, $\mathfrak{S}$ is a set of formal variable in $R_H$. Note that 
$R_H = k[W_H]$ and $k[W'_H]$ are both generated by formal variables as algebras. Then  
a formal variable in $R$ is either in $R_H$, or in $k[W'_H]$. 
Since $\mathfrak{S} \not\subset R_H = k[W_{H}]$, we have 
 $\mathfrak{S} \cap k[W'_H] \neq \emptyset$. 
\end{proof}

\begin{proof}[Proof of Theorem \ref{Key4}]
Since 
$$
\mathrm{Ker}\big(Z(\zeta_0)\big) \cong \mathrm{Ker}(\zeta_0) \cap Z( \Lambda_{\mathrm{con}} ) 
\quad\mbox{
and}
\quad
\mathrm{Ker}(\tilde{\zeta}_0) \cong \mathrm{Ker}\big(Z(\zeta_0)\big) 
 / \big(\mathrm{Ker}\big(Z(\zeta_0)\big) \cap N_{\Lambda} \big),
$$
where $N_{\Lambda}$ 
is the radical of $Z( \Lambda_{\mathrm{con}})$, we have 
$$
\mathrm{Ker}(\tilde{\zeta}_0) \cong \mathrm{Ker}(\zeta_0) \cap Z( \Lambda_{\mathrm{con}} ) 
/ \big(  \mathrm{Ker}(\zeta_0) \cap Z( \Lambda_{\mathrm{con}} ) \cap N_{\Lambda} \big). 
$$
Thus, to prove this theorem, it suffices to show that 
$$
\mathrm{Ker}(\zeta_0) \cap Z( \Lambda_{\mathrm{con}}) \subseteq N_{\Lambda}. 
$$

In fact, since
any element $g \in Z( \Lambda_{\mathrm{con}}) \cong  Z\big( \mathrm{End}
_{\underline{\mathrm{CM}}(S)}
( \bigoplus_{\chi} M_\chi ) \big)$,
we have $$g = \sum\limits_{\chi} g \circ \mathrm{Id}_{M_{\chi}} = 
\sum\limits_{\chi} g \circ \mathrm{Id}_{M_{\chi}} \circ \mathrm{Id}_{M_{\chi}}
 = \sum\limits_{\chi} (\mathrm{Id}_{M_{\chi}} \circ g \circ \mathrm{Id}_{M_{\chi}}) \in 
\prod_{\chi} \mathrm{End}_{\underline{\mathrm{CM}}(S)}( M_\chi ).$$ 
Here, 
$\mathrm{Id}_{M_{\chi}}$ is the identity map in 
$\mathrm{End}_{\underline{\mathrm{CM}}(S)}( M_{\chi} )$. 
And then we have  
\begin{align}\label{Inclu889}
Z\big(\mathrm{End}_{\underline{\mathrm{CM}}(S)}(\widehat{M}) \big) 
\cong Z\big( \mathrm{End}_{\underline{\mathrm{CM}}(S)}
( \oplus_{\chi} M_\chi ) \big) 
\subseteq  \prod_{\chi} \mathrm{End}_{\underline{\mathrm{CM}}(S)} ( M_\chi  )  
\end{align}
as algebras. Now by the argument 
in \S\ref{SSG} and Lemma \ref{IdemSpli}, we have  
$$
\prod_{\chi} \mathrm{End}_{\underline{\mathrm{CM}}(S)}( M_\chi) \cong 
\prod_{\chi} e_{\chi} \Lambda_{\mathrm{con}} 
e_{\chi} \cong \prod_{\chi} S/ \bar{I}_{\chi},
$$
and thus for any $\bar{h} \in  Z\big( \mathrm{End}_{\underline{\mathrm{CM}}(S)} 
( \bigoplus_{\chi} M_\chi  ) \big)$,
$\bar{h}$ can be written in the form
$
\sum_{\chi} \bar{h}_{\chi}
$, 
where $\bar{h}_\chi \in \mathrm{End}_{\underline{\mathrm{CM}}(S)}\big( M_\chi \big)
\cong  S/ \bar{I}_{\chi}$. 

Let $\bar{h}$ be an element in  
$\mathrm{Ker}(\zeta_0) \cap Z( \Lambda_{\mathrm{con}})$ and $\chi$ be a character of $G$.  
Without loss of generality, suppose that there is a nontrivial monomial $f \in S$ such that its image in 
$ e_{\chi} \Lambda_{\mathrm{con}} e_{\chi} \cong  S/ \bar{I}_{\chi}$ is exactly $\bar{h}_\chi$. 
Then to prove this theorem,  
it is enough to prove that $$f \in I_\chi.$$ 
In fact, if
$f \in I_\chi$, then $f^{m_\chi} \in \bar{I}_\chi$ 
for some $m_\chi \in \mathbb{N}$. Then $\bar{h}_\chi^{m_\chi} = 0$. Let $m$ be the maximal number in $\{m_\chi \}_\chi$. 
Then $\bar{h}^{m} = \big( \sum_{\chi} \bar{h}_{\chi} \big)^{m} = \sum_{\chi} \bar{h}^{m}_{\chi} = 0$. 
It implies that $\bar{h} \in N_{\Lambda}$.

We have the following two possibilities of $\chi$:
\begin{enumerate}
	\item[(1)] There is no pair $(H', \chi)$ 
	in $\tilde{G}_0$ such that $W_{H'}$ is nontrivial; 
	 
	\item[(2)] There is a pair $(H', \chi)$ in  
$\tilde{G}_0$ such that $W_{H'}$ is nontrivial.
\end{enumerate}
We shall show that in each case, $f\in I_\chi$.

\indexspace
\noindent{\it $\bm{Case }$ (1): There is no pair $(H', \chi)$ 
in $\tilde{G}_0$ such that $W_{H'}$ is nontrivial.}

Note that $W_{H'}$ is nontrivial means that the dimension of $W_{H'}$ is not zero. 

Let $\{f_r\}_r$ be the 
set of degree-one factors of $f$, i.e., $\{f_r\}_r \subseteq \{x_i\}_i$. 
For this case, 
without loss of generality, suppose that 
$x_1 \in \{f_r\}_r$. Let $G^\vee$ be the dual group of $G$. It consists of all character of $G$. 
Also let $\langle\chi_{x_1}\rangle \subseteq G^\vee$ be the subgroup generated by $\chi_{x_1}$. 
Since
$$
\#\big( \langle\chi_{x_1}\rangle \big)  \#\big(\mathrm{Ker}(\chi_{x_1})\big) = \#(G),
$$
where $\#(-)$ represents the number of elements in group,
to prove the theorem for case (1), 
we only need to consider the following 
two situations: either (i) $\mathrm{Ker}(\chi_{x_1})$ is trivial,
or (ii) $\mathrm{Ker}(\chi_{x_1})$ is non-trivial.
Let us study them case by case.

\indexspace
 
\noindent{\it $\bm{Subcase}$ (i): $\mathrm{Ker}(\chi_{x_1})$ is trivial.}
	
When $\mathrm{Ker}(\chi_{x_1})$ is trivial,  
we have $\#\big( \langle\chi_{x_1}\rangle \big) =  \#(G)$. 
It follows that $\langle\chi_{x_1}\rangle$ generates $G^\vee$ since $\#(G) = \#(G^\vee)$. 
Then $\chi_{x_1}^{m_1} = \chi_0 \chi^{-1}$ for some $m_1 \in \mathbb{N}$. It implies that 
$\chi \chi_{x_1^{m_1}} = \chi_0$. 
It means that $  
e_{\chi} (1 \otimes x_1^{m_1}) = e_{\chi} (1 \otimes x_1^{m_1}) e$ 
by (\ref{Comppp}). Thus, we get 
$$
e_\chi (1 \otimes x_1^{m_1}) \in e_{\chi} \Lambda e. 
$$
Then 
$$
e_\chi (1 \otimes x_1^{n_1 \#(G)}) \in e_{\chi} \Lambda e \Lambda e_{\chi}
$$
for some $n_1 \in \mathbb{N}$ such that $n_1 \#(G) \geq m_1$. 
Hence, when we view $f$ as a monomial, 
$e_{\chi}(1 \otimes f^{n_1 \#(G)}) \in e_{\chi} \Lambda e \Lambda e_{\chi}$,
since $x_1$ is a factor of $f$. 
By the isomorphism 
$S \cong e_\chi \Lambda e_{\chi}$,
it suggests that $f^{n_1 \#(G)} \in \bar{I}_{\chi}$ . Then we have 
$f \in I_{\chi}$. 

\indexspace
\noindent{\it $\bm{Subcase}$ (ii): $\mathrm{Ker}(\chi_{x_1})$ is non-trivial. }

If $\mathrm{Ker}(\chi_{x_1})$ is non-trivial, letting
$K := \mathrm{Ker}(\chi_{x_1})$, then we have   
$E_1 \in W_K$. Hence, since $K = \mathrm{Ker}(\chi_{x_1})$, the maximal group
whose invariant subspace is $W_K$ is exactly $K$. 
Then we get a pair $(K, W_K)$. 
Now, set
\begin{align}\label{126hyi}
\varpi^{K}: (G/K) \sharp R_K \cong 
\mathrm{End}_{S_K}(\widehat{M}^K) \rightarrow  \mathrm{End}_{S}(\widehat{M}) \cong \Lambda 
\end{align}
 be the composition $\rho^K$ (see (\ref{RhoH})) with the canonical 
 injection $\bigoplus_{i, j} \mathrm{Hom}_{S}
 \big( M_{\lambda^{K}_i}, M_{\lambda^{K}_j} \big) 
 \hookrightarrow \mathrm{End}_{S}(\widehat{M})$.
Since $\rho^{K}$ is an injection of algebras, 
$\varpi^{K}$ is also an injection of algebras. 

Since $\chi_0(K) = \{1 \}$, the trivial 
character $\chi_0$ of $G$ is in $\{\lambda^{K}_i\}$. Hence, 
by the definition of $\rho^K = \bigoplus_{i, j} \rho_{\lambda^{K}_i, \lambda^{K}_j}^K$, 
it is obvious that $e$ is in the image of $\varpi^{K}$. 
In fact, we have $\varpi^{K}(e^K) = e$, where $e^K$ is  
the idempotent in $(G/K) \sharp R_K$ corresponding to summand $S_K$ of $\widehat{M}^K$. 
Let $$\big((G/K) \sharp R_K\big)_{\mathrm{con}} := (G/K) \sharp R_K 
\big/ \big((G/K) \sharp R_K \big) e^K \big((G/K) \sharp R_K \big).$$ 
Then $\varpi^{K}$ induces another injection of algebras: 
$$\widetilde{\varpi}^{K}:
\big((G/K) \sharp R_K\big)_{\mathrm{con}} \hookrightarrow \Lambda_{\mathrm{con}}
$$
such that
the following diagram
\begin{align}\label{COM902738}
	\xymatrix{(G/K) \sharp R_K \ar@{^{(}->}[d]_{\varpi^K}  \ar@{->>}[rr]    && \big((G/K) \sharp R_K\big)_{\mathrm{con}}  \ar@{^{(}->}[d]^{\widetilde{\varpi}^{K}} 
\\
\Lambda \ar@{->>}[rr] && \Lambda_{\mathrm{con}}
}
\end{align}
commutes,
where 
the morphisms of horizontal direction are given by quotients by 
ideal $(e^K)$ and $(e)$ respectively.

By the above argument, there is a pair $(K, W_K)$ which means that $K$ is maximal among all subgroups which
have the same invariant subspace as that of $K$. 
Then $(K, \chi) \in \tilde{G}$. But $(K, \chi) \notin \tilde{G}_0$ by our
assumption. Thus $\chi(K) = \{1 \}$. 
Hence, $e_\chi$ is in the image of  
 $\varpi^K$. Let $\varpi^K(e_{\chi}^{K}) = e_{\chi}$. 
Thus, we get that $x_1 \in k[W_K] = R_K$ by $E_1 \in W_K$, 
and $e^{K}_\chi \in (G/K) \sharp R_K$. 
Hence, by the definition of $\varpi^{K}$, 
this algebra homomorphism maps $e^K_\chi (1 \otimes x_1^{m})$ 
to $e_\chi (1 \otimes x_1^{m})$ for any $m \in \mathbb{N}$.
Since $\widetilde{\varpi}^{K}$ is injective, 
$$e^K_\chi (1 \otimes x_1^{m}) \in 
e_{\chi}^K (G \sharp R) e^K$$ for some $m \in \mathbb{N}$ if and 
only if $e_\chi (1 \otimes x_1^{m}) \in e_{\chi} \Lambda e.$
Moreover, considering the character $\chi_{x_1}$ of $G/K$,
we know that the kernel of $\chi_{x_1}$
is trivial in $G/K$. 
From the above argument in subcase (i) and the fact that 
$\mathrm{Ker}( \chi_{x_1})$
is trivial in $G/K$, 
by replacing $\Lambda$ by $(G/K) \sharp R_K$, 
we have 
$$
e^K_\chi (1 \otimes x_1^{m}) \in e^K_{\chi} \big((G/K) \sharp R_K\big) e^K 
$$
for some $m_1 \in \mathbb{N}$. It implies that   
$e_\chi (1 \otimes x_1^{m}) \in e_{\chi} \Lambda e$. Thus, we get that $f \in I_{\chi} $ 
in the same way as in subcase (i). 

Combining the above two subcases (i) and (ii), we have $f \in I_{\chi} $ for case (1). 

\indexspace
\noindent{\it $\bm{Case}$ (2): There is a pair $(H', \chi)$ in  
$\tilde{G}_0$ such that $W_{H'}$ is nontrivial.}

For this case, observing that 
there is no pair $(H, \chi)$ in $\tilde{G}_0$ such that 
$\{f_r\}_r \subseteq R_H$, we know that $\{f_r\}_r \not\subset R_H$ 
for any $(H, \chi) \in \tilde{G}$ unless 
$\{f_r\}_r \subseteq R_{\bar{H}}$ 
for some pair $(\bar{H}, \chi) \in \tilde{G} \backslash 
\tilde{G}_0$. 
Thus to prove the theorem in this case, 
we only need to consider the following two situations:
either (i) $\{f_r\}_r \not\subset R_H$ 
for any $(H, \chi) \in \tilde{G}$, or (ii) $\{f_r\}_r \subseteq R_{\bar{H}}$ 
for some $(\bar{H}, \chi) \in \tilde{G} \backslash \tilde{G}_0$. 
Again, we study them case by case.

\indexspace
\noindent{\it $\bm{Subcase}$ (i): $\{f_r\}_r \not\subset R_H$ 
for any $(H, \chi) \in \tilde{G}$. }

When $\{f_r\}_r \not\subset R_H$ 
for any $(H, \chi) \in \tilde{G}$, by Lemma \ref{KeyPrLe1}, 
the image of 
$\bigodot\limits_{r} f_r^{\#(G)} \in S$ is trivial  in the coordinate ring of 
$\sqrt{\mathrm{Sing}\big( \mathrm{Spec}(S)\big)}$. It implies that 
$$
\bigodot\limits_{r} f_r^{\#(G)} \in \sqrt{\varphi(S) \cap \Lambda e \Lambda} \cong \bigcap\limits_\chi \sqrt{\bar{I}_{\chi}} \cong \bigcap\limits_\chi I_{\chi}
$$ 
by Theorem \ref{Singlocus1}. 
Then we have $\bigodot\limits_{r} f_r^{\#(G)} \in \sqrt{ e_\chi 
\Lambda e \Lambda e_\chi} = I_\chi$. Since 
$\bigodot\limits_{r} f_r$ is a factor of $f$, 
$f^{\#(G)} \in I_\chi$. 
It follows that $f \in I_\chi$.

\indexspace
\noindent{\it $\bm{Subcase}$ (ii): $\{f_r\}_r \subseteq R_{\bar{H}}$ 
for some $(\bar{H}, \chi) \in \tilde{G} \big\backslash \tilde{G}_0$. }

If $\{f_r\}_r \subseteq R_{\bar{H}}$ 
for some $(\bar{H}, \chi) \in \tilde{G} \big\backslash \tilde{G}_0$, then
we have 
$\chi(\bar{H}) = \{1\}$ and $f \in R_H$ since $\{f_r\}_r$ is the set of degree one 
factors of $f$. 
Let $e^{\bar{H}}_\chi$ and $e^{\bar{H}}$ be the idempotents of $G/\bar{H} \sharp R_{\bar{H}}$ 
corresponding to characters 
$\chi|_{G/\bar{H}}$ and $\chi_{0}|_{G/\bar{H}}$ respectively, and 
$\varpi^{\bar{H}}$ be the algebra homomorphism in (\ref{126hyi}) by 
replacing $K$ by $\bar{H}$. 
Then we have $\varpi^{\bar{H}}(e^{\bar{H}}_\chi) = e_\chi$ and 
$\varpi^{\bar{H}}(e^{\bar{H}}) = e$.  
Thus, since $f \in R_H$,  
$$
e_\chi (1 \otimes f) =  
\varpi^{\bar{H}}\big( e^{\bar{H}}_\chi (1 \otimes f) \big) 
\in \varpi^{\bar{H}}\big(G/\bar{H} \sharp R_{\bar{H}} \big)
$$ by the definition of $\varpi^{\bar{H}}$(see (\ref{126hyi})). 
By the same argument and the
commutative diagram (\ref{COM902738}) as in case (1), 
$e^{\bar{H}}_\chi (1 \otimes f) \in e^{\bar{H}}_{\chi} 
(G/\bar{H} \sharp R_{\bar{H}})e^{\bar{H}} $ if and only if 
$e_\chi (1 \otimes f) \in e_\chi \Lambda e$. 

Now, we consider $G/\bar{H} \sharp R_{\bar{H}}$ instead of $\Lambda$. 
Applying Lemma \ref{KeyPrLe1} to $G/\bar{H} \sharp R_{\bar{H}}$, we then
get $f \in I_\chi$ with the same argument
as in Case (1) or in Subcase (i) of Case (2), or reduce the 
discussion from $G/\bar{H} \sharp R_{\bar{H}}$ to a 
even smaller skew group algebra.  
Apply this procedure recursively, and finally we get that,
for any pair $(H, \chi)$ in $\tilde{G}$, $\chi(H) = \{1\}$.  
This case is exactly the case (1). Thus in this case we again have 
$f \in I_\chi$. 

Combining the above subcases (i) and (ii), we have $f \in I_{\chi} $ for case (2). 

In summary, we get that $f \in I_{\chi}$ for both cases,
which finishes the proof. 
\end{proof}

\section{Proof of Lemma \ref{Key7}}\label{App:C}

\begin{proof}[Proof of Lemma \ref{Key7}]
Fix a pair $(H, W_H)$. The set of morphisms 
$\{Z(\tilde{\zeta}^{\chi}_{H}) \}_{\chi}$ is in one to one correspondence with
the classes of characters in
 $\tilde{G}_0$. 
Moreover, by definition of $\tilde{G}_0$, for any $(H, \chi), (H, \chi') \in \tilde{G}_0$, $(H, \chi) \sim (H, \chi')$ in 
$\tilde{G}_0$ if and only if $\chi|_{H} = \chi'|_{H}$. It means that  
the classes of characters in $\tilde{G}_0$ is 
in one to one correspondence with the set of characters of $H$. 
Thus the set of morphisms 
$\{Z(\tilde{\zeta}^{\chi}_{H}) \}_{\chi}$ is in one to one correspondence 
with the set of characters of $H$, 
where this 
correspondence takes a character $\lambda$ of $H$ to morphism 
$Z(\tilde{\zeta}^{\tilde{\lambda}}_H)$ and $\tilde{\lambda}$ is a 
preimage of $\lambda$ under the group homomorphism 
$G^\vee \twoheadrightarrow H^\vee$ given by restriction of characters on $H$. 
Note that for any preimages $\tilde{\lambda}_1, \tilde{\lambda}_2$ of 
$\lambda$, $(\tilde{\lambda}_1, H) \sim (\tilde{\lambda}_2, H)$ in $\tilde{G}_0$.

In the meantime, the set of characters of $H$ is also in one to one correspondence with
the set of indecomposable idempotents of 
$\Lambda_{H} := H \sharp k[W/W_H]$. Then the set of morphisms 
$\{Z(\tilde{\zeta}^{\chi}_{H}) \}_{\chi}$ 
is again one to one correspondence with
the set of indecomposable idempotents of $\Lambda_{H}$.

Since the set of indecomposable idempotents of $\Lambda_{H}$ is in 
one to one correspondence 
with the vertices of the 
quiver $Q_{\Lambda_{H}}$ whose path algebra is $\Lambda_{H}$, 
the set of morphisms 
$\{Z(\tilde{\zeta}^{\chi}_{H}) \}_{\chi}$ is in one to one correspondence 
with the vertices of the 
quiver $Q_{\Lambda_{H}}$.

Let $(H, \chi), (H, \chi')$ be two different pairs in $\tilde{G}_0$. 
They correspond to two vertices of $Q_{\Lambda_{H}}$, say, 
$e'_{\chi}$ and $e'_{\chi'}$ respectively. Note that 
$H \subseteq \mathrm{SL}(W/W_H)$. Let $Q^{H}_{\mathrm{con}}$ be the quiver 
which is associated to the algebra $\Lambda_{H} \big/  \Lambda_H e^{H} \Lambda_{H}$, 
where $e^{H}$ is the idempotent corresponding to the trivial 
character of $H$. 
Consider $\Lambda_{H}$ instead of $\Lambda$, by Proposition \ref{Cl11}, we 
know that the two vertices corresponding to $e'_{\chi}$ and $e'_{\chi'}$ 
in $Q^{H}_{\mathrm{con}}$ are connected by a path, say $l$, in $Q^{H}_{\mathrm{con}}$.  

The path $l$ may be the concatenation of several arrows. 
To prove $Z(\tilde{\zeta}^{\chi}_{H}) = Z(\tilde{\zeta}^{\chi'}_{H})$, 
it is enough to show that for any arrow 
$\mathfrak{r}$ in the path $l$, 
$$
Z(\tilde{\zeta}^{\tilde{\lambda}^{\mathfrak{r}}_s}_{H}) 
= Z(\tilde{\zeta}^{\tilde{\lambda}_t^{\mathfrak{r}}}_{H}),
$$ 
where the characters $\tilde{\lambda}_s^{\mathfrak{r}}$ and 
$\tilde{\lambda}_t^{\mathfrak{r}}$ of $G$ are defined as follows. 
Let $\lambda_s^{\mathfrak{r}}$ and 
$\lambda_t^{\mathfrak{r}}$ be characters of $H$ such that they correspond to idempotents 
$e'_{\lambda_s^{\mathfrak{r}}}$ and 
$e'_{\lambda_t^{\mathfrak{r}}}$ respectively, where 
$e'_{\lambda_s^{\mathfrak{r}}}$ is the source of $\mathfrak{r}$
and $e'_{\lambda_t^{\mathfrak{r}}}$ is the target of 
$\mathfrak{r}$ in quiver $Q^{H}_{\mathrm{con}}$.  
Then $\tilde{\lambda}_s^{\mathfrak{r}}$ and 
$\tilde{\lambda}_t^{\mathfrak{r}}$ are two preimages of  
$\lambda_s^{\mathfrak{r}}$ and 
$\lambda_t^{\mathfrak{r}}$ under the surjection of groups
$G^\vee \twoheadrightarrow H^\vee$.  

Next, note that $W/W_H \cong W'_{H}$ as $G$-representations. 
Then $\Lambda_{H} \cong H \sharp k[W'_H]$. 
Since 
$\{E_i\}_i \big\backslash \mathfrak{S}_H$ 
gives a basis of $W'_H$,  
each arrow in the
McKay quiver $Q_{\Lambda_H}$ corresponds to a formal variable,
say $x_r \in \{x_i\}_i \big\backslash \mathfrak{S}_H^{\ast}$, 
where $\mathfrak{S}_H^{\ast}$ is given by the canonical 
dual of $\mathfrak{S}_H$ (which sends $x_i$ to $E_i$).
Hence, each arrow in $Q^H_{\mathrm{con}}$ corresponds to a formal variable,
say $x_r \in \{x_i\}_i \big\backslash \mathfrak{S}_H^{\ast}$, 

Fix an arrow $\mathfrak{r}$ in $l$. 
Now, we suppose that $\mathfrak{r}$ corresponds to 
formal variable $x_r$. For simplicity, we also suppose that 
$\lambda_s^{\mathfrak{r}} = \chi|_{H}$ and 
$\tilde{\lambda}_s^{\mathfrak{r}} = \chi$. 
Then, 
it is easy to check that the image of 
$\chi\chi_{x_r}$ under $G^\vee \twoheadrightarrow H^\vee$ 
is exactly $\lambda_t^{\mathfrak{r}} = \lambda_s^{\mathfrak{r}} \chi'_{x_r} = \chi|_{H} \chi'_{x_r}$,
since the group homomorphism  
$G^\vee \twoheadrightarrow H^\vee$ takes $\chi_{x_r}$ to 
$\chi'_{x_r}$,  
where $\chi'_{x_r}$ is the character of $H$ given by the formal variable $x_r$.
Hence, we can choose 
 $\chi\chi_{x_r}$ to be $\tilde{\lambda}_t^{\mathfrak{r}}$. 
Thus, it is left to show that 
$$
Z(\tilde{\zeta}^{\chi}_{H}) = Z(\tilde{\zeta}^{\chi\chi_{x_r}}_{H}). 
$$

Let $\bar{h} = \sum_{\lambda} \bar{h}_{\lambda}$ 
be an element in $Z(\Lambda_{\mathrm{con}})$. 
In the following, for any character $\lambda$ of $G$, 
let $h_{\lambda}$ be an element in $S$ such that 
its image is exactly $\bar{h}_{\lambda}$ 
in $S/\bar{I}_{\lambda} = 
e_{\lambda} \Lambda_{\mathrm{con}} e_{\lambda}$.
From the construction of 
$Z(\tilde{\zeta}^{\chi}_{H})$, we have 
\begin{align*}
Z(\tilde{\zeta}^{\chi}_{H})(\bar{h}) & = \zeta^{\chi}_{H}(\sum_{\lambda} 
h_\lambda ) = \tau^{H} \circ \Pi_{(H, \chi)}(\sum_{\lambda} h_\lambda)  \\ 
& = \tau^{H}\big(\sum_{ \chi^{-1}\lambda \in \{\lambda_i^{H}\}_i } h_\lambda \big)  
= \sum_{\chi^{-1}\lambda \in \{\lambda_i^{H}\}_i} 
\tau_{\chi^{-1}\lambda , \chi^{-1}\lambda }^{H}(h_{\lambda}) \in Z(G/H \sharp R_H).
\end{align*} 
Note that the center of $G/H \sharp R_H$ is the images of canonical injection 
$S_H \hookrightarrow  G/H \sharp R_H \cong \mathrm{End}_{S_H}(\widehat{M}^H)$,
which is given by its $S_H$-module structure. 
This injection induces the isomorphism of algebras 
$$
S_H \cong Z(G/H \sharp R_H).
$$
Under the inverse of this isomorphism 
we identify $\sum\limits_{\chi^{-1}\lambda \in \{\lambda_i^{H}\}_i} 
\tau_{\chi^{-1}\lambda , \chi^{-1}\lambda }^{H}(h_{\lambda})$ with $\tau_{\chi_0, \chi_0}^{H}(h_{\chi}) \in S_H$ by letting $\lambda = \chi$. 
Here, note that $\tau_{\chi^{-1}\lambda , \chi^{-1}\lambda }^{H}(h_{\lambda}) = \tau_{\chi^{-1}\lambda', \chi^{-1}\lambda' }^{H}(h_{\lambda'})$ as elements in $S_H$ for any 
$\chi^{-1}\lambda, \chi^{-1}\lambda' \in \{\lambda_i^{H}\}$.
Then we get that 
$$Z(\tilde{\zeta}^{\chi}_{H})(\bar{h}) = \tau_{\chi_0, \chi_0}^{H}(h_{\chi}) \in S_H.$$
In the same way, we get that 
$$
Z(\tilde{\zeta}^{\chi \chi_{x_r}}_{H})(\bar{h}) = \tau_{\chi_0, 
\chi_0}^{H}(h_{\chi \chi_{x_r}}) \in S_H. 
$$  
Hence, to prove that 
$Z(\tilde{\zeta}^{\chi}_{H}(\bar{h})) = Z(\tilde{\zeta}^{\chi \chi_{x_r}}_{H}(\bar{h}))$, 
it is sufficient to show that 
$\tau_{\chi_0, \chi_0}^{H}(h_{\chi}) = \tau_{\chi_0, \chi_0}^{H}(h_{\chi \chi_{x_r}})$. 
Since the algebra homomorphism 
$\tau_{\chi_0 , \chi_0}^{H}$ is
given by the canonical quotient map
$$
S \cong (k[W_H] \otimes k[W'_{H}])^G \twoheadrightarrow
k[W_H]^{(G/H)} = S_{H}  
$$
with kernel $I_H$, it is left to prove that $h_\chi - h_{\chi \chi_{x_r}}
\in I_{H}$.

Recall that $\bar{I}_\chi = e_{\chi} \Lambda e \Lambda e_{\chi}$. 
Note that $\bar{I}_{\chi} \cong e_{\chi} \Lambda e 
\Lambda e_{\chi} \cong  e \Lambda e_{\chi^{-1}} \Lambda e$ as 
ideals of $S$, since $e_{\chi} \Lambda e \cong e 
\Lambda e_{\chi^{-1}}$ and $e \Lambda e_{\chi} \cong e_{\chi^{-1}} \Lambda e$ 
by Lemma \ref{Equi}. 
Since 
\begin{align*}
 \tau_{\chi_0, \chi_0}^{H}(\bar{I}_{\chi}) & \cong \tau_{\chi_0, \chi_0}^{H}(e \Lambda e_{\chi^{-1}} \Lambda e)  \cong \tau^{H}( e \Lambda e_{\chi^{-1}} \Lambda e)\\
 & \cong \tau^{H} \circ \Pi_{(H, \chi)} \big( e_{\chi} \Lambda e \Lambda e_{\chi} \big) 
  \cong \zeta^{\chi}_{H}
 \big( e_{\chi} \Lambda e \Lambda e_{\chi} \big)
  \overset{(\ref{Key2})}{=} \{0\},
\end{align*}
we have 
$
\bar{I}_{\chi} \subseteq \mathrm{Ker}\big( \tau_{\chi_0, \chi_0}^{H} \big) \cong I_H
$.
In the same way, we also have 
$
\bar{I}_{\chi \chi_{x_r}} \subseteq I_H
$.  
It is straightforward to see that either $h_\chi - h_{\chi \chi_{x_r}} \in \bar{I}_{\chi}$, 
or $h_\chi - h_{\chi \chi_{x_r}} \in \bar{I}_{\chi \chi_{x_r}}$. 	
Then we get that $h_\chi - h_{\chi \chi_{x_r}}
\in I_{H}$ since both $\bar{I}_{\chi} \subseteq I_H$ 
and $\bar{I}_{\chi \chi_{x_r}} \subseteq I_H$. 
It follows that 
$Z(\tilde{\zeta}^{\chi}_H) = Z(\tilde{\zeta}^{\chi \chi_{x_r}}_H)$.
We thus completed the proof. 
\end{proof}

\section{Proof of Lemma \ref{Inters}}\label{App:D}

\begin{proof}[Proof of Lemma \ref{Inters}]
First, let $(H, \chi) \in \tilde{G}_0$. Then 
since 
$H \subseteq G_{H, H'}$,
$(G_{H, H'}, \chi) \in \tilde{G}_0$.
Thus to prove this lemma, it suffices 
to show that  
\begin{align}\label{Ceee1}
Z(\zeta_{(G_{H, H'}/H)}) \circ Z(\zeta^{\chi}_{H}) = Z(\zeta^{\chi}_{G_{H, H'}}).
\end{align}
And to prove \eqref{Ceee1}, it suffices to show that 
\begin{align}\label{Ceee2}
\zeta_{(G_{H, H'}/H)} \circ \zeta^{\chi}_{H} = \zeta^{\chi}_{G_{H, H'}}. 
\end{align}

To this end, denote by 
$Pr^{H}_{(G_{H, H'}, \chi)}$ the canonical 
projection $$\mathrm{Hom}_{S}\big(\bigoplus_i M_{\lambda^{(H, \chi)}_i}, 
\bigoplus_i M_{\lambda^{(H, \chi)}_i} \big) 
\twoheadrightarrow \mathrm{Hom}_{S}\big(\bigoplus_i M_{\lambda^{(G_{H, H'}, \chi)}_i}, 
\bigoplus_i M_{\lambda^{(G_{H, H'}, \chi)}_i} \big), 
$$
and let $Pr^{H}_{G_{H, H'}} := Pr^{H}_{(G_{H, H'}, \chi_0)}$ due to 
$\{ \lambda^{(G_{H, H'}, \chi)}_i \} \subseteq \{\lambda^{(H, \chi)}_i \}$. 
Then we have 
\begin{align} 
Pr^{H}_{G_{H, H'}} \circ \Pi_{(H, \chi)} & = Pr^{H}_{G_{H, H'}} 
\circ \Delta_{(H, \chi)} \circ  Pr_{(H, \chi)} \nonumber \\
& = \Delta_{(G_{H, H'}, \chi)} \circ Pr^{H}_{(G_{H, H'}, \chi)} \circ  Pr_{(H, \chi)}  \nonumber \\
& = \Delta_{(G_{H, H'}, \chi)} \circ Pr_{(G_{H, H'}, \chi)} \nonumber \\ 
& = \Pi_{(G_{H, H'}, \chi)}. \label{COMM9864} 
\end{align}
In the meantime, it is easy to check that 
\begin{align}\label{COMM98655}
Pr_{G_{H, H'}/H} \circ \tau^{H} = 
\prod_{\chi_j, \chi_r \in\big\{ \lambda^{G_{H, H'}}_i \big\}_i} 
\tau^{H}_{\chi_j \chi_r } \circ Pr^{H}_{G_{H, H'}}.
\end{align}
Moreover, for any two
$\chi_j, \chi_r \in \{ \lambda^{G_{H, H'}}_i \}_i$,
both of them are in $\{ \lambda^{H}_i \}_i$ 
since $\{ \lambda^{G_{H, H'}}_i \}_i 
\subseteq \{ \lambda^{H}_i \}_i$.  
From the following commutative diagram
of canonical projections of polynomial rings
$$
\xymatrixcolsep{2pc}
\xymatrixrowsep{2pc}
\xymatrix{ k[V]
\ar[rr] \ar[rd] &&   k[W_{G_{H, H'}}] \\ 
& k[W_H]\ar[ru],&
}
$$
we obtain the following commutative diagram:
$$
\xymatrixcolsep{2pc}
\xymatrixrowsep{2pc}
\xymatrix{ \big( V_{\chi_j^{-1} \lambda_r} \otimes k[V] \big)^{G}
\ar[rr]^-{\tau^{^{G_{H, H'}}}_{\chi_j, \chi_r}} \ar[rd]_-{\tau^{H}_{\chi_j, \chi_r}} 
&&  \big( V_{\chi_j^{-1} \lambda_r} 
\otimes k[W_{G_{H, H'}}] \big)^{G /G_{H, H'}} \\ 
& \big( V_{\chi_j^{-1} \chi_r} \otimes k[W_H] \big)^{G/H}  
\ar[ru]_---{\tau^{^{(G_{H, H'}/H)}}_{\chi_j, \chi_r}} &
}
$$
That is,
\begin{align}\label{Labb213}
\tau^{^{G_{H, H'}}}_{\chi_j, \chi_r} = 
\tau^{^{(G_{H, H'}/H)}}_{\chi_j, \chi_r} \circ \tau^{H}_{\chi_j, \chi_r}.
\end{align}
It suggests that 
\begin{align}\label{Labb21334}
\tau^{(G_{H, H'}/H)} = \tau^{H} \circ 
\prod_{\chi_j, \chi_r \in\big\{ \lambda^{G_{H, H'}}_i \big\}_i} \tau^{H}_{\chi_j \chi_r }.
\end{align}
By the definition of $\zeta^{\chi}_{G_{H, H'}}$ (see (\ref{Seccc})), we know that 
\begin{align}\label{Bro12}
\zeta^{\chi}_{G_{H, H'}} = \tau^{G_{H, H'}} \circ \Pi_{(G_{H, H'}, \chi)}
\end{align}
and 
\begin{align}\label{Bro13}
\zeta_{(G_{H, H'}/H)} \circ \zeta^{\chi}_{H} 
= \big( \tau^{(G_{H, H'}/H)} \circ Pr_{G_{H, H'}/H} \big) \circ \big( \tau^{H} \circ \Pi_{(H, \chi)} \big).
\end{align}
Combining (\ref{COMM9864}), (\ref{COMM98655}) and (\ref{Labb21334}), we have 
\begin{align*}
& \big( \tau^{(G_{H, H'}/H)} \circ Pr_{G_{H, H'}/H} \big) 
\circ \big( \tau^{H} \circ \Pi_{(H, \chi)} \big) \\
&  = \tau^{(G_{H, H'}/H)} \circ 
\prod_{\chi_j, \chi_r \in
\big\{ \lambda^{G_{H, H'}}_i \big\}_i} \tau^{H}_{\chi_j \chi_r } 
\circ Pr^{H}_{G_{H, H'}} \circ \Pi_{(H, \chi)} \\ 
& = \tau^{G_{H, H'}} \circ \Pi_{(G_{H, H'}, \chi)}.
\end{align*}
It follows that 
$\zeta_{(G_{H, H'}/H)} \circ \zeta^{\chi}_{H} 
= \zeta^{\chi}_{G_{H, H'}}$ from (\ref{Bro12}) and (\ref{Bro13}). 
We thus get \eqref{Ceee2}, and the lemma follows.
\end{proof}

\section{Proof of Proposition \ref{Cl22}}\label{App:E}

Let us first 
recall the following fact on $\underline{\mathrm{CM}}(S)$. 
For any Cohen-Macaulay $S$-module $M$, 
let $\Omega(M)$ be the object
 in $\underline{\mathrm{CM}}(S)$, which is the kernel  
of a(ny) surjection $\pi_M : S^{\oplus n} \twoheadrightarrow M$.
Note that $\Omega(M)$ is unique in 
$\underline{\mathrm{CM}}(S)$, but not in 
$\mathrm{mod}(S)$. It is well-known that $[-1] = \Omega(-)$ in 
$\underline{\mathrm{CM}}(S)$.

For any two integers $l, l'\in\mathbb Z$ with $l' \geq l$,
let $\widehat{M}(l, l') = \bigoplus_{i = l}^{l'}\widehat{M}[i]$ as above.
Note that if $l'>0>l$, then $\widehat{M}[0] = \widehat{M} $ 
is a direct summand of $\widehat{M}(l, l')$. 
We have the following lemma. 

\begin{lemma}\label{Conne2}
 $R^{\widehat{M}} = R^{\widehat{M}(l, l')}$ for any 
 $l, l' \in \mathbb{Z}$ with $l' > 0 > l$. 
\end{lemma}

\begin{proof}
First, similarly to (\ref{Inclu889}) it is easy to check that 
\begin{align}\label{Inclu893567}
Z\Big( \mathrm{End}_{\underline{\mathrm{CM}}(S)}
\big( \widehat{M} \oplus \Omega(\widehat{M}) \big) \Big) 
\subseteq Z\Big( \mathrm{End}_{\underline{\mathrm{CM}}(S)}
\big( \widehat{M} \big) \Big) \oplus Z
\Big( \mathrm{End}_{\underline{\mathrm{CM}}(S)}\big( 
\Omega(\widehat{M}) \big) \Big)
\end{align} 
as algebras. 
When considering the associated reduced rings, 
since $\mathrm{End}_{\underline{\mathrm{CM}}(S)}
\big( \Omega(\widehat{M}) \big) \cong 
\mathrm{End}_{\underline{\mathrm{CM}}(S)}(\widehat{M} )$ we have  
\begin{align}\label{Incl673}
R^{\widehat{M} \oplus \Omega(\widehat{M})} 
\subseteq R^{\widehat{M}} \oplus R^{\Omega( \widehat{M})} 
\cong R^{\widehat{M}} \oplus R^{\widehat{M}}.
\end{align}
Here, the compositions 
\begin{align}\label{Comp073579}
R^{\widehat{M} \oplus \Omega(\widehat{M})} \subseteq 
R^{\widehat{M}} \oplus R^{\Omega( \widehat{M})} \twoheadrightarrow R^{\widehat{M}} 
\end{align}
and 
\begin{align}\label{Comp0735888}
R^{\widehat{M} \oplus \Omega(\widehat{M})} \subseteq 
R^{\widehat{M}} \oplus R^{\Omega( \widehat{M})} 
\twoheadrightarrow R^{\Omega( \widehat{M})} 
\end{align}
are exactly $\varphi^{\widehat{M} \oplus 
\Omega(\widehat{M})}_{\widehat{M}}$ and 
$\varphi^{\widehat{M} \oplus \Omega(\widehat{M})}_{\Omega(\widehat{M})}$ respectively, 
where the surjections are the natural projections. 
By Theorem \ref{Theo22}, from (\ref{Incl673}) we get that 
\begin{align}\label{Incl684}
R^{\widehat{M} \oplus \Omega(\widehat{M})} \subseteq 
\varphi(S) / \sqrt{\varphi(S) \cap \Lambda e \Lambda}
\oplus \varphi(S) / \sqrt{\varphi(S) \cap \Lambda e \Lambda}
\end{align}
as algebras. 
In the following, we identify $R^{\widehat{M} \oplus \Omega(\widehat{M})}$ 
with the image of the above inclusion (\ref{Incl684}). 

Observe that since 
$\varphi_{\Omega(\widehat{M})} = \varphi^{\widehat{M} \oplus 
\Omega(\widehat{M})}_{\Omega(\widehat{M})} \circ \varphi_{\widehat{M} 
\oplus \Omega(\widehat{M})}$ 
(see (\ref{Comp0735888})),
the composition  with (\ref{Comp0735888})
\begin{align}\label{Comp0946}
\varphi(S) / \sqrt{\varphi(S) \cap \Lambda e \Lambda} 
\xrightarrow{\varphi_{\widehat{M} \oplus \Omega(\widehat{M} )}}  
R^{\widehat{M} \oplus \Omega(\widehat{M})} \subseteq R^{\widehat{M}} 
\oplus R^{\Omega( \widehat{M})} \twoheadrightarrow  R^{\Omega( \widehat{M})}
\end{align}
is exactly $\varphi_{\Omega(\widehat{M})}$. 
Since 
the composition 
$$ 
\varphi(S) / \sqrt{\varphi(S) \cap \Lambda e \Lambda} 
\xrightarrow{\varphi_{\Omega( \widehat{M})}} 
R^{\Omega(\widehat{M})} \cong R^{\widehat{M}} 
$$
is exactly $\varphi_{\widehat{M}}$, by (\ref{Comp0946}) and (\ref{Comp073579}) 
the composition 
\begin{align}\label{Comp0977}
& \varphi(S) / \sqrt{\varphi(S) \cap \Lambda e \Lambda} 
\xrightarrow{\varphi_{\widehat{M} \oplus \Omega(\widehat{M} )}}  
R^{\widehat{M} \oplus \Omega(\widehat{M})}  \subseteq 
R^{\widehat{M}} \oplus R^{\Omega( \widehat{M})} \nonumber\\
 &\cong R^{\widehat{M}} \oplus R^ {\widehat{M}} \cong 
 \varphi(S) \big/ \sqrt{\varphi(S) \cap \Lambda e \Lambda} 
 \oplus \varphi(S) / \sqrt{\varphi(S) \cap \Lambda e \Lambda} 
\end{align}
is the diagonal algebra homomorphism,
where the isomorphism 
$R^ {\widehat{M}} \cong \varphi(S) / \sqrt{\varphi(S) 
\cap \Lambda e \Lambda}$ is the inverse of $\varphi_{\widehat{M}}$.

In summary, we have the following commutative diagram of homomorphisms:
$$
\xymatrixcolsep{2pc}
\xymatrix{
&\varphi(S) \big/ \sqrt{\varphi(S) \cap \Lambda e \Lambda}
\ar[d]^-{\varphi_{\widehat{M} \oplus \Omega(\widehat{M})}}
\ar@/_1cm/[lddd]_{\varphi_{\widehat{M}}} \ar@/^1cm/[rddd]^{\varphi_{\Omega(\widehat{M})}} 
&\\
&  R^{\widehat{M} \oplus \Omega(\widehat{M})}\ar@{^(->}[d]^{ \varphi^{\widehat{M} 
\oplus \Omega(\widehat{M})}_{\widehat{M}} \oplus \varphi^{\widehat{M} 
\oplus \Omega(\widehat{M})}_{\Omega(\widehat{M})} } &\\
& R^{\widehat{M}} \oplus R^{\Omega(\widehat{M})} \ar@{->>}[ld]\ar@{->>}[rd]&\\
R^{\widehat{M}}\ar[rr]^{\sim} \ar[d]_{\varphi^{-1}_{\widehat{M}}} &&
R^{\Omega(\widehat{M})} \ar[d]^{\varphi^{-1}_{\Omega(\widehat{M})}} \\
\varphi(S) \big/ \sqrt{\varphi(S) \cap \Lambda e \Lambda}\ar[rr]^{\mathrm{Id}}
&& \varphi(S) \big/ \sqrt{\varphi(S) \cap \Lambda e \Lambda}.
}$$

Now, let $(\bar{f}_1, \bar{f}_2) \in R^{\widehat{M} \oplus \Omega(\widehat{M})}$ such that 
$\bar{f}_1, \bar{f}_2 \in \varphi(S) 
/ \sqrt{\varphi(S) \cap \Lambda e \Lambda}$ (see \eqref{Incl684}). 
Since the composition 
(\ref{Comp0977}) is the diagonal algebra homomorphism,
the algebra homomorphism 
$\varphi_{\widehat{M} \oplus \Omega(\widehat{M})}$ takes $\bar{f}_1$ to 
$(\bar{f}_1, \bar{f}_1) \in  R^{\widehat{M} \oplus \Omega(\widehat{M})}$.
Then $(\bar{f}_1, \bar{f}_2) - (\bar{f}_1,\bar{f}_1) = (0, 
\bar{f}_2 -\bar{f}_1)$ is also in $R^{\widehat{M} \oplus \Omega(\widehat{M})}$. 
Let 
$(\delta, g_{1, 2}) \in Z\Big( \mathrm{End}_{\underline{\mathrm{CM}}(S)}\big( 
\widehat{M} \oplus \Omega(\widehat{M}) \big) \Big)$ whose image is 
$(0, \bar{f}_2 -\bar{f}_1)$ under the quotient of nilradical. 
Here, $\delta \in Z\big( \mathrm{End}_{\underline{\mathrm{CM}}(S)}( 
\widehat{M}) \big)$ and $g_{1, 2} \in Z\Big( \mathrm{End}_{\underline{\mathrm{CM}}(S)}\big(\Omega(\widehat{M}) \big) \Big)$.

Note that the nilradical of $Z\Big( 
\mathrm{End}_{\underline{\mathrm{CM}}(S)}
\big( \widehat{M} \oplus \Omega(\widehat{M}) \big) \Big)$ is contained in the 
 nilradical of $Z\Big( \mathrm{End}_{\underline{\mathrm{CM}}(S)}
 \big( \widehat{M} \big) \Big) \oplus Z\Big( \mathrm{End}_{\underline{\mathrm{CM}}(S)}\big( \Omega(\widehat{M}) \big) \Big)$ 
 by the inclusion (\ref{Inclu893567}). Moreover, 
 the 
 nilradical of $Z\Big( \mathrm{End}_{\underline{\mathrm{CM}}(S)}
 \big( \widehat{M} \big) \Big) \oplus Z\Big( \mathrm{End}_{\underline{\mathrm{CM}}(S)}\big( \Omega(\widehat{M}) \big) \Big)$ is the direct sum 
of the 
nilradicals of $Z\Big( \mathrm{End}_{\underline{\mathrm{CM}}(S)}
\big( \widehat{M} \big) \Big)$ and 
$Z\Big( \mathrm{End}_{\underline{\mathrm{CM}}(S)}\big( 
\Omega(\widehat{M}) \big) \Big)$. It suggests that 
$\delta$ is in the nilradicals of $Z\Big( \mathrm{End}_{\underline{\mathrm{CM}}(S)}
\big( \widehat{M} \big)$. 
Then, $\delta^{n'} = 0$ in $Z\Big( \mathrm{End}_{\underline{\mathrm{CM}}(S)}
( \widehat{M} ) \Big)$ for some $n' \in \mathbb{N}$. 
It implies that 
$$
(\delta, g_{1, 2})^{n'} = (0, g^{n'}_{1, 2}) \in 
Z\Big( \mathrm{End}_{\underline{\mathrm{CM}}(S)}\big( 
\widehat{M} \oplus \Omega(\widehat{M}) \big) \Big). 
$$
Thus, for any two characters $\lambda, 
\lambda'$ of $G$ and $f \in \mathrm{Hom}_{\underline{\mathrm{CM}}(S)}\big(M_{\lambda}, 
\Omega( M_{\lambda'})\big) \subseteq \mathrm{End}_{\underline{\mathrm{CM}}(S)}
\big( \widehat{M} \oplus \Omega(\widehat{M}) \big)$, 
$g^{n'}_{1, 2} \circ f = f \circ 0 = 0$.
It is straightforward to see that $g_{1, 2}$ is nilpotent, i.e., 
$\bar{f}_1 - \bar{f}_2 = 0$ in 
$\varphi(S) / \sqrt{\varphi(S) \cap \Lambda e}$. 
It follows that $(\bar{f}_1, \bar{f}_2) = (\bar{f}_1, \bar{f}_1)$ in 
$R^{\widehat{M} \oplus \Omega(\widehat{M})}$. 
Thus, the image of the inclusion 
$$
R^{\widehat{M} \oplus \Omega(\widehat{M})} \subseteq \varphi(S) 
/ \sqrt{\varphi(S) \cap \Lambda e \Lambda}
\oplus \varphi(S) / \sqrt{\varphi(S) \cap \Lambda e \Lambda}
$$
is exactly the image of the diagonal algebra homomorphism 
$$
\Delta_{S}: \varphi(S) / \sqrt{\varphi(S) \cap \Lambda e \Lambda}
\hookrightarrow \varphi(S) / \sqrt{\varphi(S) \cap \Lambda e \Lambda}
\oplus \varphi(S) / \sqrt{\varphi(S) \cap \Lambda e \Lambda}.
$$
It follows that the image of 
$R^{\widehat{M} \oplus \Omega(\widehat{M})}$ is isomorphic to 
$\mathrm{Im}(\Delta_{S}) \cong \varphi(S) / \sqrt{\varphi(S) 
\cap \Lambda e \Lambda}$ as algebras.  
We then have the isomorphism:
$$
R^{\widehat{M} \oplus \Omega(\widehat{M})} \cong 
\mathrm{Im}(\Delta_{S}) \cong \varphi(S) / \sqrt{\varphi(S) 
\cap \Lambda e \Lambda} \cong R^{\widehat{M}},  
$$ 
which, by (\ref{Comp073579}), is given by $\varphi^{\widehat{M} \oplus 
\Omega(\widehat{M})}_{\widehat{M}}$. 
With the same argument,
by induction on $i$ we obtain
$$
R^{\big(\bigoplus\limits^{0}_{i = l - l'} \widehat{M}[i]\big)} \cong R^{\widehat{M}} 
$$
given by $\varphi^{\big(\bigoplus\limits^{0}_{i = l - l'} \widehat{M}[i]\big)}_{\widehat{M}}$.  
Since $R^{\big(\bigoplus\limits^{ 0}_{i = l - l'} \widehat{M}[i]\big)} 
\cong R^{\big(\bigoplus\limits_{i = l}^{l'} 
\widehat{M}[i] \big)[l']} = R^{\widehat{M}(l, l')}$, we get
$
R^{\widehat{M}(l, l')} \cong R^{\widehat{M}}
$ given by $\varphi^{\widehat{M}(l, l')}_{\widehat{M}}$. 
\end{proof}

Now, let $N$ be a Cohen-Macaulay $S$-module. 
From \S\ref{Slgd}, we know that 
$\widehat{M}$ is a classical generator of $\underline{\mathrm{CM}}(S)$. 
Then we have $N \in \big< \widehat{M} \big>_{m}$ for some $m \in \mathbb{N}$. 
It suggests that there is a sequence of distinguished triangles
$$
\Big\{ M^{r-1}_N \rightarrow M^r_N \rightarrow W^{r-1}_N 
\rightarrow M^{r-1}_N [1] \Big\}_{m \geq r \geq 2}
$$
in $\underline{\mathrm{CM}}(S)$,
where $M^m_N := N$, $M^{r-1}_N  
\in \big< \widehat{M} \big>_{r-1}$ and 
$W^{r-1}_N  \in \big< \widehat{M} \big>_{1}$.
Here, for the triangulated category $\mathcal{T}$ and a class of objects 
$\mathcal{E}$ in $\mathcal{T}$,  
recall that $\langle\mathcal{E}\rangle_1$ is 
the smallest full subcategory of 
$\mathcal{T}$ containing the objects in $\mathcal{E}$ and 
closed under direct summands, finite direct sums and shifts.

Since $M^1_{N} \in \big<\widehat{M}\big>_1$ and $W^{r-1}_N \in \big<\widehat{M}\big>_1$ for any $2 \leq r \leq m$, 
$\bigoplus_{r=2}^{m}W^{r-1}_N \oplus M^1_{N} $ 
is a direct summand of some direct sums of 
$\bigoplus_{i=v}^{v'} \widehat{M}[i]$ for some $v', v \in \mathbb{Z}$. 
Let $l', l \in \mathbb{Z}$ such that $l' \geq v' \geq v \geq l$ and $l' > 0 > l$.
We have the following.

\begin{lemma}\label{Conne}
For any nontrivial $f \in \mathrm{Hom}_{\underline{\mathrm{CM}}(S)}(N, N)$ 
which is not nilpotent, 
there is an element 
$g \in \mathrm{Hom}_{\underline{\mathrm{CM}}(S)}(N, \widehat{M}[i])$   
for some $i \in [l, l']$ such 
that $g \circ f^j$ is nontrivial for any 
$j \in \mathbb{N}$, where $f^j$ is the 
$i$-fold self-composition of $f$. 
\end{lemma}

\begin{proof}
We prove by contradiction.
Assume that there is an element 
$f \in \mathrm{Hom}_{\underline{\mathrm{CM}}(S)}(N, N)$, 
which is not nilpotent, such that 
for any $i \in [l, l']$ and any $g \in \mathrm{Hom}_{\underline{\mathrm{CM}}(S)}(N, \widehat{M}[i])$, 
$g \circ f^{j} = 0$ for some $j \in \mathbb{N}$. 
It implies that for any $g \in \mathrm{Hom}_{\underline{\mathrm{CM}}(S)}\big(N, \bigoplus_{i=l}^{l'} \widehat{M}[i] \big)$, 
$g \circ f^{j} = 0$ for some $j \in \mathbb{N}$. 

We first consider the distinguished triangle $M^{1}_N 
\rightarrow M^2_N \rightarrow W^{1}_N \rightarrow M^{1}_N [1]$ in $\underline{\mathrm{CM}}(S)$. 
By applying the triangle functor $\mathrm{Hom}_{\underline{\mathrm{CM}}(S)}(N, -)$ 
to this distinguished triangle, 
we have the following long exact sequence 
\begin{align*}
\cdots \rightarrow & \mathrm{Hom}_{\underline{\mathrm{CM}}(S)}(N, M^{1}_N ) 
\rightarrow \mathrm{Hom}_{\underline{\mathrm{CM}}(S)}(N, M^2_N) \rightarrow \\ 
& \mathrm{Hom}_{\underline{\mathrm{CM}}(S)}(N, W^{1}_N ) \rightarrow 
\mathrm{Hom}_{\underline{\mathrm{CM}}(S)}(N, M^{1}_N [1]) \rightarrow \cdots.
\end{align*}
Denote by $\varrho$ and $\varepsilon$ the above morphisms
$\mathrm{Hom}_{\underline{\mathrm{CM}}(S)}(N, M^{1}_N) \rightarrow 
\mathrm{Hom}_{\underline{\mathrm{CM}}(S)}(N, M^2_N)$ and 
$\mathrm{Hom}_{\underline{\mathrm{CM}}(S)}(N, M^2_N) \rightarrow 
\mathrm{Hom}_{\underline{\mathrm{CM}}(S)}(N, W^{1}_N)$ respectively. 
Let $h \in \mathrm{Hom}_{\underline{\mathrm{CM}}(S)}(N, M^2_N)$. 
By assumption, we have 
$\varepsilon (h) \circ f^j = 0$ for some $j \in \mathbb{N}$, since $W^{1}_N$ is a direct summand of 
$\bigoplus_{i=l}^{l'} \widehat{M}[i]$. 
It implies that 
$$
\varepsilon (h \circ f^j) = 0
$$
since $\varepsilon$ is a $\mathrm{Hom}_{\underline{\mathrm{CM}}(S)}(N, N)$-module homomorphism.  
Then, from the above long exact sequence, there is an element, say 
$h' \in \mathrm{Hom}_{\underline{\mathrm{CM}}(S)}(N, M^{1}_N)$, such that 
$\varrho(h') = h \circ f^j$. 
By our assumption again, we also have 
$h' \circ f^{j'} = 0$ for some $j' \in \mathbb{N}$ since $M^{1}_N$ is a direct summand of 
$\bigoplus_{i=l}^{l'} \widehat{M}[i]$.  
In the meantime, we know that $\varrho$ is also a 
$\mathrm{Hom}_{\underline{\mathrm{CM}}(S)}(N, N)$-module 
homomorphism. Then we get that  
$$
 h \circ f^{j + j'}= h \circ f^j \circ f^{j'} = \varrho(h') \circ f^{j'} = \varrho(h' \circ f^{j'}) = 0. 
$$

Next, consider all distinguished triangles in
$$
\Big\{ M^{r-1}_N \rightarrow M^r_N \rightarrow W^{r-1}_N 
\rightarrow M^{r-1}_N [1] \Big\}_{m \geq r \geq 2}
$$
in $\underline{\mathrm{CM}}(S)$. 
By induction on $r$, we get that for any $M^r_N$, 
and any $\alpha \in  \mathrm{Hom}_{\underline{\mathrm{CM}}(S)}(N, M^r_N)$, 
$ \alpha \circ f^{j_r} = 0$ for some $j_r \in \mathbb{N}$. 
If $r = m$, let $\alpha$ be the identity 
morphism in $\mathrm{Hom}_{\underline{\mathrm{CM}}(S)}(N, N)$ since $N = M^m_N$. 
Then we obtain that 
$$
f^{j_m} = \alpha \circ f^{j_m} = 0.
$$ 
It implies that $f$ is a nilpotent element,
which contradicts 
to the assumption that $f$ is not a nilpotent element. 
\end{proof}

\begin{proof}[Proof of Proposition \ref{Cl22}]
The diagram is naturally commutative. 
Thus, to prove this proposition, we only need to show that 
the algebra homomorphisms in this diagram are all isomorphisms. 

Fix the Cohen-Macaulay $S$-module $N$ and keep the above settings for $l, l'$
(see the paragraph above Lemma \ref{Conne}). 
First, similarly to (\ref{Inclu889}), it obvious that 
$$
Z\Big( \mathrm{End}_{\underline{\mathrm{CM}}(S)}\big(\widehat{M}(l, l') \oplus N \big) \Big) 
\subseteq Z\Big( \mathrm{End}_{\underline{\mathrm{CM}}(S)}\big( \widehat{M}(l, l') \big) \Big) 
\oplus Z\Big( \mathrm{End}_{\underline{\mathrm{CM}}(S)}( N ) \Big) 
$$
as algebras. 
Moreover, when considering 
the corresponding reduced rings, by Lemma \ref{Conne2}
we have  
\begin{align}\label{Incl90624}
R^{\widehat{M}(l, l') \oplus N} \subseteq R^{\widehat{M}(l, l')} \oplus R^{N} \cong 
R^{\widehat{M}} \oplus R^{N}, 
\end{align}
which is exactly $\varphi_{\widehat{M}(l, l')}^{\widehat{M}(l, l') 
\oplus N} \bigoplus \varphi_{N}^{\widehat{M}(l, l') \oplus N}$, 
where $R^{\widehat{M}(l, l')} \cong R^{\widehat{M}}$ is given 
by $\varphi_{\widehat{M}}^{\widehat{M}(l, l')}$. We identify $R^{\widehat{M}(l, l') \oplus N}$ 
with the image of the above inclusion (\ref{Incl90624}).

Let $(\bar{g}_1, \bar{g}_2 ) \in R^{\widehat{M}} \oplus R^{N}$ be in 
$R^{\widehat{M}(l, l') \oplus N}$. 
By Theorem \ref{Theo22}, 
we have an isomphism $\varphi_{\widehat{M}}:
\varphi(S) \big/\sqrt{\varphi(S)\cap \Lambda e\Lambda}\cong R^{\widehat{M}}$. Let 
$\varphi_{\widehat{M}}^{-1}$ be its inverse morphism. 
The composition of
$$ 
\varphi(S) \big/\sqrt{\varphi(S)\cap \Lambda e\Lambda} 
\xrightarrow{\varphi_{\widehat{M}(l, l') \oplus N}} R^{\widehat{M}(l, l') \oplus N} 
$$
with (\ref{Incl90624}) is
$\varphi_{\widehat{M}(l, l')}^{\widehat{M}(l, l') \oplus N} \bigoplus 
\varphi_{N}^{\widehat{M}(l, l') \oplus N}$, and 
the image of $\varphi_{\widehat{M}}^{-1}
(\bar{g}_{1})$ under this composition is exactly 
$\Big(\bar{g}_1,  \varphi_{N}\big(\varphi_{\widehat{M}}^{-1}(\bar{g}_{1})\big) \Big)$. 
Here, we use the fact that the composition $ R^{\widehat{M}(l, l') \oplus N} 
\subseteq R^{\widehat{M}(l, l')} \oplus R^{N} 
\twoheadrightarrow R^{\widehat{M}(l, l')} \cong R^{\widehat{M}}$ is exactly 
$\varphi^{\widehat{M}(l, l') \oplus N}_{\widehat{M}}$, 
and the composition $ R^{\widehat{M}(l, l') \oplus N} \subseteq 
R^{\widehat{M}(l, l')} \oplus R^{N} 
\twoheadrightarrow R^{N}$ is exactly 
$\varphi^{\widehat{M}(l, l') \oplus N}_{N}$.  
Thus we obtain that 
$\Big(\bar{g}_1,  \varphi_{N}\big(\varphi_{\widehat{M}}^{-1}(\bar{g}_{1})\big) \Big) 
\in R^{\widehat{M}} \oplus R^{N}$ 
is also in $R^{\widehat{M}(l, l') \oplus N}$. 
Therefore,
$$
\big(\bar{g}_1, \bar{g}_2  \big) 
 - \Big(\bar{g}_1,  \varphi_{N}\big(\varphi_{\widehat{M}}^{-1}(\bar{g}_{1})\big) \Big) 
 = \Big(0, \bar{g}_2 - \varphi_{N}\big(\varphi_{\widehat{M}}^{-1}(\bar{g}_{1})\big) \Big) 
\in R^{\widehat{M}} \oplus R^{N}
$$ 
is also in 
$R^{\widehat{M}(l, l') \oplus N}$. 

Let $(\varepsilon, f_{1,2}) \in 
Z\Big( \mathrm{End}_{\underline{\mathrm{CM}}(S)}
\big(\widehat{M}(l, l') \oplus N \big) \Big)$ whose image 
is $\Big(0, \bar{g}_2 - \varphi_{N}\big(\varphi_{\widehat{M}}^{-1}(\bar{g}_{1})\big) \Big)$ 
under the quotient 
by the nilradical 
\begin{align}\label{Nilp885390}
Z\Big( \mathrm{End}_{\underline{\mathrm{CM}}(S)}\big(\widehat{M}(l, l') \oplus N \big) \Big) 
\twoheadrightarrow R^{\widehat{M}(l, l') \oplus N}. 
\end{align}
Here, $\varepsilon \in Z\Big( \mathrm{End}_{\underline{\mathrm{CM}}(S)}
\big(\widehat{M}(l, l') \big) \Big)$ and $f_{1,2} \in Z\big( 
\mathrm{End}_{\underline{\mathrm{CM}}(S)}( N ) \big)$. 
From the following commutative diagram 
\begin{align*}
\xymatrixcolsep{4pc}
\xymatrix{  
Z\Big( \mathrm{End}_{\underline{\mathrm{CM}}(S)}
\big(\widehat{M}(l, l') \oplus N \big) \Big) 
\ar@{->>}[d] \ar@{->>}[r]^-{\widetilde{\phi}^{\widehat{M}(l, l') \oplus N}_{\widehat{M}(l, l')}} & 
 Z\Big( \mathrm{End}_{\underline{\mathrm{CM}}(S)}\big(\widehat{M}(l, l') \big) \Big)
  \ar@{->>}[d] \\ 
R^{\widehat{M}(l, l') \oplus N} \ar@{->>}[r]^-{\varphi^{\widehat{M}(l, l') 
\oplus N}_{\widehat{M}(l, l')}} & R^{\widehat{M}(l, l')},
}	
\end{align*}
where the vertical morphisms
are projections to the quotients by 
the nilradicals respectively,  
the morphism $Z\Big( \mathrm{End}_{\underline{\mathrm{CM}}(S)}
\big(\widehat{M}(l, l') \big) \Big) \twoheadrightarrow R^{\widehat{M}(l, l')}$ 
takes $\varepsilon$ to $0 \in R^{\widehat{M}(l, l')}$, since the image of $(\varepsilon, f_{1,2})$ 
is $\Big(0, \bar{g}_2 - \varphi_{N}\big(\varphi_{\widehat{M}}^{-1}(\bar{g}_{1})\big) \Big)$ 
under morphism (\ref{Nilp885390}). 
Here, note that 
$\widetilde{\phi}^{\widehat{M}(l, l') \oplus N}_{\widehat{M}(l, l')}(\varepsilon, f_{1,2}) 
= \varepsilon$ and 
$\varphi^{\widehat{M}(l, l') \oplus N}_{\widehat{M}(l, l')}
\Big(0, \bar{g}_2 - \varphi_{N}\big(\varphi_{\widehat{M}}^{-1}(\bar{g}_{1})\big) \Big) = 0$.
Then $\varepsilon$ is a nilpotent element in 
$Z\Big( \mathrm{End}_{\underline{\mathrm{CM}}(S)}\big(\widehat{M}(l, l') \big) \Big)$; that is,
$\varepsilon^{n} = 0$ for some $n \in \mathbb{N}$. 
Thus $$ (\varepsilon, f_{1,2})^{n} = (0, f_{1, 2}^{n}) \in 
Z\Big( \mathrm{End}_{\underline{\mathrm{CM}}(S)}\big(\widehat{M}(l, l') 
\oplus N \big) \Big).$$
Note that 
$f_{1, 2}^{n} \in \mathrm{Hom}_{\underline{\mathrm{CM}}(S)}(N, N)$. 
Therefore, for any $g \in \mathrm{Hom}_{\underline{\mathrm{CM}}(S)}(N, \widehat{M}(l, l'))$, 
we have
that $g \circ f^{n}_{1, 2} = g \circ 0 = 0$. 
It implies that for any $i \in [l, l']$ and
for any $g' \in \mathrm{Hom}_{\underline{\mathrm{CM}}(S)}(N, \widehat{M}[i])$, 
$g' \circ f^{n}_{1, 2}$is trivial 
since $\widehat{M}[i]$ is a direct summand of $\widehat{M}(l, l')$.

By Lemma \ref{Conne}, we have that $f_{1, 2}$ is nilpotent, i.e., $f_{1, 2}^{n'} = 0$ 
for some $n'\in \mathbb{N}$. 
Then we get that $(\varepsilon, f_{1,2})^{n\cdot n'} = 0$. 
It suggests that $\Big(0, \bar{g}_2 - 
\varphi_{N}\big(\varphi_{\widehat{M}}^{-1}(\bar{g}_{1})\big) \Big) = 0$ 
in $R^{\widehat{M}(l, l') \oplus N}$. Since
$R^{\widehat{M}(l, l') \oplus N} \subseteq 
R^{\widehat{M}(l, l')} \oplus R^{N}$, 
$\Big(0, \bar{g}_2 - \varphi_{N}\big(\varphi_{\widehat{M}}^{-1}(\bar{g}_{1})\big) \Big) = 0$ 
in 
$R^{\widehat{M}(l, l')} \oplus R^{N}$. 
It follows that 
$$
\bar{g}_2 - \varphi_{N}\big(\varphi_{\widehat{M}}^{-1}(\bar{g}_{1})\big) = 0
$$
in $R^N$ since $(\varepsilon, f_{1,2})$ is a nilpotent. 
Then $(\bar{g}_1, \bar{g}_{2}) = (\bar{g}_1, 
\varphi_{N}\big(\varphi_{\widehat{M}}^{-1}(\bar{g}_{1})\big))$. 
Thus $(\bar{g}_1, \bar{g}_{2})$ is the image of $\varphi_{\widehat{M}}^{-1}(\bar{g}_{1})$ under
$$ 
\varphi(S) \big/\sqrt{\varphi(S)\cap \Lambda e\Lambda} 
\xrightarrow{\varphi_{\widehat{M}(l, l') \oplus N}} R^{\widehat{M}(l, l') 
\oplus N} \subseteq R^{\widehat{M}(l, l')} \oplus R^{N} \cong R^{\widehat{M}} \oplus R^{N}. 
$$
By the arbitrariness of $(\bar{g}_1, \bar{g}_{2})$, 
we get that $\varphi_{\widehat{M}(l, l') \oplus N}$ is a surjection. 

On the other hand, from
$$
\varphi_{\widehat{M}}^{\widehat{M}(l, l') \oplus N} \circ \varphi_{\widehat{M}(l, l') \oplus N} 
= \varphi_{\widehat{M}}, 
$$
we get $\varphi_{\widehat{M}(l, l') \oplus N}$ is an injection, 
since by Theorem \ref{Theo22}
$\varphi_{\widehat{M}}$ is an isomorphism of algebras. 

We thus get that 
$\varphi_{\widehat{M}(l, l') \oplus N}$ is an isomorphism of algebras. It implies that 
$\varphi_{\widehat{M}}^{\widehat{M}(l, l') \oplus N}$ 
is also an isomorphism of algebras by the commutativity of this diagram. 
\end{proof}


\begin{thebibliography}{100}


\setlength{\itemsep}{-0.2mm}

\bibitem{Bo}
Nicolas Bourbaki, 
{\it Groupes et alg\'ebres de Lie, Ch. IV-VI.} Hermann, Paris, 1968. 

\bibitem{RB}
Ragnar-Olaf Buchweitz,
{\it Maximal Cohen-Macaulay modules and Tate-cohomology over Gorenstein rings},
Mathematical Surveys and 
Monographs, 262. American Mathematical Society,
Providence, RI, 2021.

\bibitem{CS}
Xiao-Wu Chen and Long-Gang Sun, 
{\it Singular equivalences of Morita type}, 
preprint, 2012. 

\bibitem{DW}
Will Donovan and Michael Wemyss,
{\it Noncommutative deformations and flops},
Duke Math. J. 165 (2016), no. 8, 1397-1474.

\bibitem{YHHH}
Yuki Hirano, 
{\it Relative singular locus and Balmer spectrum of matrix factorizations}, 
Trans. Amer. Math. Soc. 371 (2019), 4993-5021. 

\bibitem{HuC}
Wei Hu and Changchang Xi, 
{\it Derived equivalences and stable equivalences of Morita type, I}, 
Nagoya Math. J. 200 (2010), 107-152.


\bibitem{HK0}
Zheng Hua and Bernhard Keller, 
{\it Cluster categories and rational curves},
arXiv: math.AG/1810.00749v6. 

\bibitem{IN}
 Osamu Iyama and Yusuke Nakajima, 
{\it On steady non-commutative crepant resolutions}.
J. Noncommut. Geom. 12 (2018), no. 2, 457-471. 

\bibitem{IT}
Osamu Iyama and Ryo Takahashi,
{\it Tilting and cluster tilting for quotient singularities},
Math. Ann. 356 (2013), no. 3, 1065-1105.


\bibitem{ITR}
Srikanth B. Iyengar and Ryo Takahashi,  
{\it The Jacobian ideal of a commutative ring and annihilators of cohomology},
J. Algebra 571 (2021), 280-296.

\bibitem{KB}Bernhard Keller,
{\it Singular Hochschild cohomology via the singularity category},
C. R. Math. Acad. Sci. Paris 356 (2018), no. 11-12, 1106-1111.

\bibitem{GJL}
Graham J. Leuschke, 
{\it Non-commutative crepant resolutions: scenes from categorical geometry},
Progress in commutative algebra 1, 293-361, de Gruyter, Berlin, 2012. 

\bibitem{Liu}
Jian Liu, 
{\it Annihilators and dimensions of the singularity category},  
Nagoya Math. J. 250 (2023), 533-548.

\bibitem{LXC}
Yuming Liu and Changchang Xi, 
{\it Construction of stable equivalences of Morita type for finite-dimensional algebras. I},
Trans. Amer. Math. Soc. 358 (2006), no. 6, 2537-2560.

\bibitem{SMRM}
Shani Meynet and Robert Moscrop, 
{\it McKay quivers and decomposition}, 
Letters in Mathematical Physics 113 (2023), no. 63.

\bibitem{MHH}
Hiroki Matsui, 
{\it Singular equivalences of commutative noetherian rings and reconstruction of singular loci}, 
J. Algebra 522 (2019), 170-194. 


\bibitem{MU}
Izuru Mori and Kenta Ueyama,
{\it Stable categories of graded maximal Cohen-Macaulay modules over noncommutative quotient singularities},
Adv. Math. 297 (2016), 54-92.


\bibitem{DR}Dmitri Orlov,
{\it Triangulated categories of singularities and D-branes in Landau-Ginzburg models},
Tr. Mat. Inst. Steklova 246 (2004), Algebr. Geom. Metody, Svyazi i Prilozh.,
240-262; translation in
Proc. Steklov Inst. Math. 2004, no. 3 (246), 227-248.

\bibitem{DR1}Dmitri Orlov,
{\it Triangulated categories of singularities, and equivalences between Landau-Ginzburg models},
(Russian) Mat. Sb. 197 (2006), no. 12, 117-132.

\bibitem{DR2}
Dmitri Orlov,
{\it Derived categories of coherent sheaves and triangulated categories of singularities},
Algebra, arithmetic, and geometry: in honor of Yu. I. Manin. Vol. II,
503-531, Progr. Math., 270, Birkh\"{a}user Boston, Boston, MA, 2009.


\bibitem{DO1}
Dmitri Orlov,
{\it Formal completions and idempotent completions of triangulated categories of singularities},  
Adv. Math. 226 (2011), no. 1, 206-217. 

 
\bibitem{SV0}
\v{S}pela \v{S}penko and Michel Van den Bergh,
{\it Non-commutative resolutions of quotient singularities for reductive groups}, 
Invent. Math. 210 (2017), no. 1, 3-67. 

\bibitem{SV1}
\v{S}pela \v{S}penko and Michel Van den Bergh,
{\it Non-commutative crepant resolutions for some toric singularities I},
Int. Math. Res. Not. IMRN 2020, no. 21, 8120-8138.

\bibitem{Stacks}The stacks project, {\it The singular locus},
\href{https://stacks.math.columbia.edu/tag/07R0}{https://stacks.math.columbia.edu/tag/07R0}.

\bibitem{Stacks0}The stacks project, {\it Generators of triangulated categories},
\href{https://stacks.math.columbia.edu/tag/09SI}{https://stacks.math.columbia.edu/tag/09SI}.


\bibitem{V}
Michel Van den Bergh,
{\it Three-dimensional flops and noncommutative rings}.
Duke Math. J. 122 (2004), no. 3, 423-455.

\bibitem{V2}
Michel Van den Bergh,
{\it Non-commutative crepant resolutions}.
The legacy of Niels Henrik Abel, 749-770, Springer, Berlin, 2004.

\bibitem{WZF}
 Zhengfang Wang, 
{\it Singular equivalence of Morita type with level}, 
J. Algebra 439 (2015), 245-269. 

\bibitem{W1}
Zhengfang Wang,
{\it Invariance of the Gerstenhaber algebra structure on Tate-Hochschild cohomology}, 
J. Inst. Math. Jussieu 20 (2021), no. 3, 893-928. 


\bibitem{XCC}
Changchang Xi, 
{\it Stable equivalences of adjoint type}, 
Forum Math. 20 (2008), no. 1, 81-97.


\bibitem{XYY}Xuan Yu,
{\it The triangular spectrum of matrix factorizations is the singular locus}, 
Proc. Amer. Math. Soc. 144 (2016), no. 8, 3283-3290.


\bibitem{ZZi}
Guodong Zhou and Alexander Zimmermann, 
{\it On singular equivalences of Morita type}, 
J. Algebra 385 (2013), 64-79. 
	
\end{thebibliography}
\end{document}